 \documentclass[12pt]{article}
\usepackage{amsfonts}

\usepackage{graphicx}
\usepackage{amsmath}
\usepackage{hyperref}
\usepackage{color}

\newtheorem{theorem}{Theorem}[section]

\newtheorem{corollary}[theorem]{Corollary}

\newtheorem{definition}[theorem]{Definition}

\newtheorem{lemma}[theorem]{Lemma}

\newtheorem{notation}[theorem]{Notation}

\newtheorem{proposition}[theorem]{Proposition}
\newtheorem{remark}[theorem]{Remark}

\newenvironment{proof}[1][Proof]{\textbf{#1.} }{\ \rule{0.5em}{0.5em}}

\def \L{\Lambda}
\def \<{\langle}
\def \>{\rangle}
\def \l{\lambda}
\def \({\Big(}
\def \){\Big)}
\def \R{\mathbb R}

\def \As{\mathsf{A}}

\def \H^0{{\cal H}^0 or}

\def \G{{\cal G}}

\def \Q{{\cal Q}}

\def \Y{{\cal Y}}

\def \w{\omega}

\def \kf{\frak k}

\def \p{\partial}
\def \beq{\begin{equation}}
\def \eeq{\end{equation}}

\def \nn{[]}

\def \n{\nabla}

\def \eref{\eqref}
\def \V{\mathcal V}

\def \s{\sigma}
\def \lrc{\lrcorner\,}
\numberwithin{equation}{section}

\begin{document}

\title{Stability of the Yang-Mills heat equation for finite action
\author{Leonard Gross \\
Department of Mathematics\\
Cornell University\\
Ithaca, NY 14853-4201\\
{\tt gross@math.cornell.edu}}
}

\maketitle

\begin{abstract}
The existence and uniqueness of solutions to the Yang-Mills heat equation over domains
in $\R^3$  was  proven in a previous paper for  initial data  lying in the Sobolev space of
order one-half, which is the critical Sobolev index for this equation.
In the present paper the stability of these solutions will be established.
The variational equation, which is only weakly parabolic, and has highly singular coefficients,
will be shown to have unique strong solutions up to addition of a vertical solution.
Initial data will be taken to be in Sobolev class one-half.
The proof relies on an infinitesimal version of the ZDS procedure: one solves first
an augmented, strictly parabolic version of the variational equation and then adds
to the solution a function which is vertical along the original path.
Energy inequalities and Neumann domination techniques will be used to establish
apriori initial behavior for solutions.

\end{abstract}



\tableofcontents


\section{Introduction} \label{secintro}

      Denote by  $K$  a compact Lie group with Lie algebra $\kf$.  Let 
\beq
A(t,x) \equiv \sum_{j=1}^3 A_j(t, x) dx^j
\eeq
 be a $\kf$ valued 1-form on $R^3$ for each
$t \ge 0$. Its curvature is given at time $t$ by $B(t) = dA(t) + A(t)\wedge A(t)$. Here 
$d$ denotes the spatial exterior derivative. The Yang-Mills heat equation is the weakly parabolic non-linear equation given by 
\beq
\p A/\p t = - d_{A(t)}^* B(t), \ \ \ t >0,     \label{I1}
\eeq
where $d_A^*$ denotes the gauge covariant exterior co-derivative. 
Ignoring the terms on the right side of \eref{I1} which are quadratic and cubic in $A(t)$,
 one finds the linear expression  $ -d^*dA$.
Since $d^*d$
is only a portion of the Laplacian, $-\Delta = d^*d + dd^*$, on 1-forms, the equation \eref{I1}
is only weakly parabolic. 

            The existence and uniqueness of solutions to \eref{I1} has already
been investigated in \cite{CG1, CG2} and \cite{G70} for the initial value problem over $\R^3$ 
as well as for the initial-boundary value problem over a bounded convex region in $\R^3$.  
In the following, $M$ will denote either all of $\R^3$ or the closure of a convex bounded
 open subset of $\R^3$ with smooth boundary.
    It was shown in \cite{G70} that  for any connection form $A_0 \in H_{1/2}(M; \L^1\otimes \kf)$
the equation \eref{I1} has a solution with initial value $A_0$. 
It was also shown that the degree of  regularity of the solution
depends not only on the regularity of $A_0$ but also on some algebraic features of $A_0$
which are intimately connected to the gauge invariance of the equation. The key distinction
in regularity properties of solutions is best understood from  
 the  following notion of a strong solution.

By a strong solution of \eref{I1} over an interval $(0, T)$ we mean a function 
$A: (0, T)\rightarrow \{\kf$  valued 1-forms on  $M$\}  such that, 
for each point $t \in (0, T)$, there holds

a) $A(t)$ is in $H_1(M; \L^1\otimes \kf)$ 

b) $B(t)$ is in  $H_1(M; \L^2\otimes \kf)$ 

c)\  Equation \eref{I1} holds.

\noindent
In addition, some continuity properties as a function of $t$ are assumed.
Condition a) allows one to define $B(t)$ while condition b) allows one to give
 meaning to the right side of Equation \eref{I1}. But it can easily happen that condition a) fails while
 condition b) holds. In this case we refer to the solution as an almost strong solution.
 Of course one must interpret the derivative that occurs in $B(t)$ 
 as a weak derivative in this case. 
           For example, if $K$ is  the circle group and  we  identify its Lie algebra with $i\R$, then 
 $\sqrt{-1}A$ is real valued, and if $A_0$ is exact, say $A_0 = \sqrt{-1}\, d\lambda$ for some real
 valued function $\lambda$ on $M$, 
 then the function $A(t) \equiv A_0$ is a solution to \eref{I1} because, in this commutative case,
 the curvature is simply given by $B(t) = dA(t)$, which is $\sqrt{-1}\, d^2 \lambda$ and which is zero
 in a weak sense, no matter how irregular $A_0$ is. Thus in this example condition a) can fail even though
 condition b) always holds.
           For a general compact Lie group $K$ the same phenomenon occurs:
 Let $g:M\rightarrow K$  
 be a function and let $A_0 = g^{-1} dg$. Then the (weak) curvature
 of $A_0$ is zero and the function $A(t) \equiv A_0$ is a solution in some sense to \eref{I1},
 no matter how irregular $A_0$ is.  In this case the flow of the equation does not smooth
 the  initial value $A_0$.  For a general connection form $A_0 \in H_{1/2}$ it appears to be  impossible
 to separate out a part which propagates without smoothing, as in this example,
 from a part  which is smoothed by the equation, without destroying gauge invariance 
 of the initial value problem \eref{I1}.
It was shown in
\cite{G70} that for any initial connection form $A_0 \in H_{1/2}(M;\L^1\otimes \kf)$,
 the equation \eref{I1} always has a solution in a generalized sense. The solution 
  may not have the regularity required by condition a) but does have the regularity required by condition b).
  It was also shown that there exists a gauge function $g$ such that the gauge transform 
\beq
 A_0^g \equiv g^{-1} A_0 g + g^{-1} dg                                    \label{I2}
 \eeq
  is indeed the initial value of a strong solution.
   The main result
 of \cite{G70} may thus be stated succinctly as ``any connection form 
 $A_0 \in H_{1/2}(M; \L^1\otimes \kf)$ is, after gauge transformation, the initial value
  of a strong solution".  Uniqueness also holds when properly formulated.
  
  The goal of the present paper is to prove the analogous theorem for the variational equation.
Along a solution $A(\cdot)$ to \eref{I1}, the variational equation 
 is given by
\beq
-\p v(t)/\p t = d_A^*d_A v + [v\lrc B].      \label{I3}
\eeq
Here $v(t)$ is, for each $t \ge 0$, a $\kf$ valued 1-form on $M$.  
The last term represents an interior product. 
The $t$ dependence in $A(t)$, $B(t)$ and $v(t)$ on the right is suppressed.
The second order
derivative terms that appear on the right side of \eref{I3} are $d^*d v$,  and  consequently
the  variational equation is only  weakly parabolic, as was \eref{I1}.   
The solutions $A(t,x)$ to 
\eref{I1} that are of interest to us have a strong singularity at $t =0$.
 Therefore the linear equation \eref{I3},
 in addition to being only weakly parabolic,
 has highly singular coefficients. 
 
     There are functions in the initial data space for the variational equation which are
     not smoothed by the flow of the equation, but 
     they can be singled out in a gauge invariant manner, unlike the circumstance for \eref{I1}.  
    They    are the vertical vectors at $A_0$, i.e. the tangent vectors to the 
       orbit of the gauge group through $A_0$.
       Such a vector can be represented as $v_1 = d_{A_0} \alpha$ for some 
       $\kf $ valued function $\alpha$ on $M$. 
       The solution to \eref{I3} with this initial value
       is simply  given by $v(t) = d_{A(t)} \alpha$ and experiences no smoothing under the flow.
        The main theorem of this paper
       asserts that, for any element 
        $v_0 \in H_{1/2}(M; \L^1\otimes \kf)$, there is a generalized solution to \eref{I3} with
          initial value $v_0$  and  in addition, there is also a strong solution $v(t)$ such that
       \beq
\lim_{t\downarrow 0} \Big(v(t) - v_0\Big) \ \ \text{is vertical at}\ A_0.     \label{I4}
\eeq       
     In short,  we will prove that  any vector $v_0 \in H_{1/2}$ is the initial value of
        a strong solution  modulo vertical vectors.  
              This is the precise infinitesimal analog of the main theorem of \cite{G70}, which asserts 
              existence of strong solutions to \eref{I1} modulo gauge transformations. 
                       In the case that $M\ne \R^3$, boundary conditions must be imposed on the solution
 $v(t)$ for $t >0$ in the  discussion above. These will be discussed in Section \ref{secEUaug}.
           
            The existence proof for the Yang-Mills heat equation itself, 
given in \cite{CG1} and \cite{G70}, relied on a method 
 that goes back in one form or another to Zwanziger \cite{Z}, Donaldson \cite{Do1}
 and Sadun \cite{Sa}.  To prove existence of solutions to the variational equation we are going
 to use an infinitesimal version of the ZDS procedure.  
 The infinitesimal version of the ZDS procedure introduced in this paper
  has proven to be advantageous 
  over other methods that    naturally present themselves
  for the problems at hand:
 In the infinite dimensional manifold  of connection forms  over $\R^3$ 
 (with some Sobolev restrictions)  
 the tangent space at a point $A$ decomposes
  into vertical and horizontal subspaces in a gauge invariant way. The vertical vectors
  play a distinguished role, as already noted.  Moreover the horizontal component of any
   solution to \eref{I3} propagates by a strictly parabolic equation. 
   Consequently, techniques that rely on projection into these two subspaces can be
    expected to be useful.  
             But the use of these projections entails use of the Green functions for  gauge covariant
    Laplacians under Dirichlet or Neumann boundary conditions. There does not appear
     to be a useful Poincar\'e inequality for  these Laplacians in the case of
      Neumann boundary conditions. Consequently, useful bounds on their Green operators
      are hard to get  
      in this important case.      
               Even  for a compact manifold
  without boundary Green operators exist only for irreducible connections. This  class
  of connections  is open in a Sobolev topology $H_k$ for sufficiently large $k$,
  and  has been used in the works \cite{AM,MV81}, for example, in their analysis of the quotient space
 \beq
 \Y = \Big(\{connection\ forms\}/ gauge\ group\Big)        \label{I6}
 \eeq
 and in \cite{NR}.
  But we wish to deal with connections in Sobolev class
  $1/2$, where restriction to irreducible connections is not feasible. Fortuitously, 
  the infinitesimal   version of the ZDS procedure circumvents 
  this problem.

   For  a solution $A(\cdot)$ to the Yang-Mills heat equation define  
   \begin{align}
   \rho(A) := \int_0^1 s^{-1/2} \|B(s)\|_2^2 ds.
   \end{align}
   This is a gauge invariant functional of the initial data. It plays a fundamental role,  
   both technically and conceptually,
   because it captures in a gauge invariant manner the notion of $H_{1/2}$ initial data.
To understand how this happens, it is illuminating to compute its value when the gauge group
$K$ is the circle group.  
       One finds in this case that it reduces to the $H_{1/2}$ norm of the initial value $A_0$
       when $A_0$ is in Coulomb gauge, i.e., $d^*A_0 = 0$. As is well known, the space $\{A_0: d^*A_0 = 0\}$
       constitutes  a section of the bundle $\Y$  when $K$ is the circle group.   This example is discussed
       further in Remark \ref{remsig}.
                  If $K$ is not commutative  this method of identifying solutions modulo gauge transformations
   with some section  for the quotient space  $\Y$     does not play well. 
         It is well understood that if $K$ is not commutative 
     there is no good analog for the Coulomb gauge.
   Problems associated with the Gribov ambiguity enter  \cite{Si1, NR, Z}. 
 But it is the quotient space     that plays the role of the configuration space
   for  the classical Yang-Mills field.
    Our objective, when $K$ is not commutative, 
 is to make the quotient space into a complete, infinite dimensional,
   Riemannian manifold, which in some suitable sense consists of $H_{1/2}$ connection forms
   on $\R^3$ modulo the corresponding  gauge group. This will be carried out in \cite{G72}.
   The 
   functional $\rho(A)$ will play a fundamental
   role in this procedure  by determining, in a gauge invariant way, 
   which  initial data are  to be regarded as being 
   ``in''  $H_{1/2}$ when $K$ is not commutative. 
   Our use of  the term ``action'' for  $\rho(A)$   is motivated by the
   fact that  if $A_0 \in H_{1/2}(\R^3)$ then it has an extension to a slab 
   in Minkowski space which makes a finite contribution to the magnetic component of the Lagrangian.

\section{Statement of results}    \label{secstate}

\subsection{Strong and almost strong solutions}

$M$ will denote  either $\R^3$ or the closure of a  bounded convex open subset of $\R^3$
 with smooth boundary. $K$ will denote a compact Lie group with Lie algebra $\kf$.  We will always take $K$ to be a subgroup of the orthogonal
 resp. unitary group of a finite dimensional real resp. complex   inner product  space $\V$. 
  We can identify $\kf$ with a real subspace  of $End\, \V$.  
  We denote by $\<\cdot, \cdot\>$ an Ad $K$ invariant inner product on $\kf$.    
  The induced norm on $\kf$  is   equivalent to the operator norm of $\kf \subset End\, \V$ since $\kf$ is finite dimensional.
 
  We will assume as given a time dependent, $\kf$ valued 1-form  $A(t)$ on $M$:
  $A(t)(x) = \sum_{j=1}^3 A_j(x,t) dx^j$,  where each $A_j$ is a $\kf$  
  valued function on $M\times [0, \infty)$.   
  $W_1(M; \L^p\otimes \kf)$ will denote the 
   set of those  p-forms in $L^2(M)$ whose weak first derivatives are in $L^2(M)$.
   We will usually write $W_1$    when the order, $p$, is clear from the context.
   $ H_1(M;\L^p\otimes \kf)$ will denote the set of  $\kf$-valued p-forms  in $ W_1$
   which satisfy the boundary conditions  specified  in Notation \ref{notSob}.
   If $M \ne \R^3$ then  $A(t)$ will always be assumed to satisfy the boundary conditions
   $A(t)_{norm} =0$ when Neumann boundary conditions are under discussion and $A(t)_{tan} =0$ when Dirichlet boundary conditions are under discussion. We will write $A(t) \in H_1(M)$ in all three cases. 
 Its $H_1$ norm is given by 
   \beq
   \| A(t)\|_{H_1}^2 = \int_M\Big(\sum_{j=1}^3 |\p_j A(t, x)|_{\L^1\otimes \kf}^2 
                       + |A(t,x)|_{\L^1\otimes \kf}^2 \Big)dx.  \label{ymh5}
   \eeq

\begin{definition}\label{defstrsol} {\rm   
A {\it strong solution} to the Yang-Mills heat  equation over $(0, \infty)$ 
 is a continuous function 
\beq
A(\cdot): (0,\infty) \rightarrow L^2(M;\L^1\otimes \kf) 
                              \subset \{\frak k\text{-valued 1-forms on}\ M \}   \label{ymh6}   
\eeq
such that
\begin{align}
a)&\  A(t) \in H_1 
           \ \text{for all}\ \ t\in(0,\infty)\ \text{and}\ A(\cdot):(0,\infty)\rightarrow H_1 
                                                             \    \text{is continuous},     \notag\\
b)& \ B(t):= dA(t) +A(t)\wedge A(t)  \in H_1 
             \ \text{for each}\ \  t\in (0,\infty),\notag\\ 
c)& \  \text{the strong $L^2(M)$ derivative $A'(t) \equiv dA(t)/dt $}\ 
\text{exists on}\ (0,\infty),   \text{and}     \notag\\
&\ \ \ \ \ A'(\cdot) : (0,\infty) \rightarrow L^2(M)\ \ \text{is continuous}, \notag\\    
d)& \  A'(t) = - d_{A(t)}^* B(t)\ \ \text{for each}\ t \in(0, \infty).           \label{ymh10}     
\end{align}
}
\end{definition}

 In \cite{G70} it was proven that for some connection forms 
 $A_0$ in $H_{1/2}(M;\L^1\otimes \kf)$ 
 there is a strong solution $A(\cdot)$   to \eref{ymh10} over  
 $(0, \infty)$ which converges to $A_0$ in $H_{1/2}$ as $t\downarrow 0$.  Here $H_{1/2}$ refers to
 to the Sobolev norm that interpolates between $L^2$ and  the $H_1$ norm given in \eref{ymh5}. 
 $A(\cdot)$ 
 extends to a continuous function on $[0,\infty)$ into $H_{1/2}$ and therefore into $L^3(M)$
 by Sobolev. The initial values $A_0$ which are permitted in this theorem include,
  up to gauge transformation,
  all connection forms in $H_{1/2}$. 
         In this paper we will make use only of the properties listed in Definition \ref{defstrsol} and
         such further properties as are explicitly spelled out. In particular $A(0)$ need not
          be in $L^3$ in most of this  paper.
  
   We will be concerned 
   with existence, uniqueness and
  properties of solutions to the variational equation \eref{I3} along such
   a strong solution to \eref{ymh10}.  The spatial derivatives of  $A(t)$ enter into
    the coefficients of the variational
   equation and  can have bad singularities near $t =0$. 
   The existence of solutions to the variational equation is 
    jeopardized by these singularities. 
   The singular nature of this initial behavior of $A(\cdot)$ was studied in \cite{G70} and much  
   of the     information derived  there  will be needed in this paper.   

The initial behavior of a strong solution to the Yang-Mills heat equation is deducible in large part 
from the following gauge invariant condition, which will often be a key hypothesis.
             \begin{definition}{\rm  \label{defa-act} Let $1/2\le a <1$. A strong solution to
 the Yang-Mills equation \eref{ymh10}  over $(0,\infty)$ has 
 {\it finite a-action} if
 \beq
  \rho_A(t) \equiv (1/2)\int_0^t s^{-a} \|B(s)\|_{L^2(M)}^2 ds < \infty  \label{ymh15a} 
  \eeq 
   for some $t >0$ (and therefore for all $t <\infty$ because $s \mapsto \|B(s)\|_2$ is non-increasing).
  In the important case $a = 1/2$ we will simply say that  $A$ has finite action.
  }
  \end{definition}

\begin{notation}\label{notAib}  
{\rm In addition to the gauge invariant condition
 \eref{ymh15a}      we will also     
 need   the following   gauge invariant condition on $A(\cdot)$.
 For each $s \in [0, \infty)$ the function
\begin{align}
[0,\infty) \ni t\mapsto A(t) - A(s)\ \  \text{is continuous into  $L^3(M;\L^1\otimes \kf)$}. \label{vst356}
\end{align}
This is strictly weaker than the assumption that  $A(\cdot)$ is continuous as a function
 from $[0, \infty)$     
 into $L^3$  because \eref{vst356} can hold even if $A(t) \notin L^3(M)$ for any $t \ge 0$. 
This is a relevant issue only in case $M = \R^3$.
Although continuity of $A(\cdot)$ into $L^3$ was proved in \cite{G70} under the condition
  that the initial value $A_0$ is in    $H_{1/2}(\R^3)$, 
we will remove this restriction on the initial data in \cite{G72} in order to incorporate instanton sections into these structures.
   Only \eref{vst356} will survive. 
           In this paper 
     two results will require that $A(t) \in L^3$ for some $t >0$ (namely in Sections \ref{secivas} and \ref{secrecab}) and this condition will be made explicit where used. 
   All other results are independent of these.
   }
   \end{notation}

\begin{notation}\label{wedcomm}{\rm    We continue to use the notation from \cite{CG1} for the exterior and interior commutator 
    products, given  by 
    $[u\wedge v] = \sum_{I,J}[u_I, v_J] dx^I \wedge dx^J$   when $u \equiv \sum_I u_I dx^I$
    and $v\equiv \sum_J v_J dx^J$ are $End\ \V$ valued forms, while
    $\<w, [u\lrc v]\>_{\L^r\otimes \kf} =\<[u\wedge w], v\>_{\L^{p+r}\otimes \kf}$ for
     all $w \in \L^r\otimes \kf$ when degree  $u =p$ and degree $v = p+r$. 
     Then $d_A u = du +[A\wedge u]$ and $d_A^* u =d^* u + [A\lrc u]$.
}\end{notation}

    \begin{definition} {\rm The {\it variational equation} for the Yang-Mills heat 
equation \eref{ymh10} is
\beq
-v'(s) = d_{A(s)}^* d_{A(s)} v(s) + [ v(s)\lrc B(s)].        \label{ve}
\eeq
}
\end{definition}
     \begin{notation} \label{notgisn}
{\rm (Gauge invariant Sobolev norms)     
Although there is no gauge invariant Sobolev norm for a gauge potential $A$,
there are gauge invariant Sobolev norms for variations of $A$. 
 For any connection form $A$ over $M$ that lies in $W_1(M)$ define
\begin{align}
\p_j^A \w =\p_j \w + [A_j, \w] = (\p_j + ad\, A_j)\w   \label{ps8}
\end{align}
for a $\kf$ valued $p$-form $\w$. 
Ignoring boundary conditions for the moment we define
\begin{align}
\|\w\|_{H_1^A}^2 = \int_M\( |\p_j^A \w(x)|_{\L^1\otimes \kf}^2 
       + | \w(x)|_{\L^1\otimes \kf}^2\)dx.                                             \label{st10}
\end{align}
This is the gauge invariant $H_1$ norm on forms that we will use in most of this paper.
The corresponding $H_b$ norms are given by 
\beq
     \|\w\|_{H_b^A} \equiv \|\w\|_{H_b^A(M)} = \| (1- \Delta_A)^{b/2}\w \|_{L^2(M)},\ \ \  b \ge 0,  \label{ST19}
     \eeq
 where  $\Delta_A$ denotes the Bochner Laplacian on $\kf$-valued 1-forms over $M$. 
 The precise domain of this gauge covariant operator will be explained in Notation  \ref{notSob}.
 These norms are gauge invariant in the sense that
      \beq
     \| \w^g \|_{H_b^{A^g}} = \|\w\|_{H_b^A}
     \eeq
for any sufficiently regular gauge function $g:M\rightarrow K$.   Here $A^g$ is defined in \eref{I2} and
$\w^g = g^{-1} \w g$. 
}
\end{notation}

Given a strong solution $A(\cdot)$ to the Yang-Mills heat equation  and a number $T \in (0, \infty)$
we will write $\As = A(T)$ and use this connection form to define gauge invariant Sobolev 
norms on forms, as in \eref{st10}.
We will see in Lemma  \ref{lemeqSob} that these Sobolev norms are equivalent for different $T$. But in Section \ref{secEUaug}
we will make a choice of  $T$ that is well adapted for use in the contraction principle.

           \begin{definition}{\rm
A  {\it strong solution} to the variational equation along $A(\cdot)$ over $[0,\infty)$  
 is a continuous function
\beq
v:[0,\infty) \rightarrow L^2(M; \L^1\otimes \kf)      \label{ve1}
\eeq
such that 
\begin{align}
a)&\  v(t) \in H_1^{\As}    
 \ \text{for all} \ t\in(0,\infty)\ \text{and}\ v(\cdot):(0,\infty)\rightarrow H_1^\As 
                                                                                     \  \text{is continuous},      \label{ve2}\\
b)& \ d_{A(t)}v(t) \in H_1^\As 
  \  \text{for each}\ \  t\in (0,\infty),\                \label{ve3}\\
c)& \  \text{the strong $L^2(M)$ derivative $v'(t) \equiv dv(t)/dt $}\ 
\text{exists on}\ (0,\infty),   \nolinebreak  \label{ve4}    \\
d)&\ \text{The variational equation \eref{ve} holds on}\  (0,\infty).     \label{ve5}
\end{align}
A function $v(\cdot)$ satisfying all of the preceding conditions except a) will be called an
{\it almost strong solution}. In this case the spatial exterior derivative $d\,v(t)$, which enters
into the  definition of $d_{A(t)} v(t)$, 
must be interpreted  as a weak derivative.  
   It can happen that for some $t >0$, the weak exterior derivative $d_{A(t)} v(t)$
is   in $W_1$, as required by \eref{ve3}, even though $v(t)$ is not  in $W_1(M)$.
This is, typically, a manifestation of the identity
$d^2\lambda  = 0$, which holds in a generalized sense even if  $d\lambda$ is not in $W_1$.
This was already pointed out in the introduction.
}
\end{definition}

\begin{definition}\label{defvert}{\rm (Vertical solutions) 
  A {\it vertical solution} to the variational equation along $A(\cdot)$ is a function 
   of the form 
\beq
z(t) = d_{A(t)} \alpha, \ \ \ 0 < t <  \infty                            \label{ve8}
\eeq
for some element $\alpha \in H_1^\As(M;\kf)$.
Recall the standard terminology: A  $\kf$-valued 1- form $\w\in L^2(M;\L^1\otimes \kf)$ is 
{\it horizontal} at a connection  form $A$ if $d_A^* \w =0$. The horizontal 1-forms at $A$ form a closed subspace of $L^2(M;\L^1\otimes \kf)$ because $d_A^*$ is a closed operator on $L^2$.
For each $t$    the 1-form  $d_{A(t)}\alpha$  is clearly orthogonal to the horizontal
 subspace at $A(t)$.
}
\end{definition}

           \begin{lemma} \label{lemvert1} $(${\rm Vertical solutions}$)$
Let $A(\cdot)$ be a strong solution to \eref{ymh10}  over $(0, \infty)$  of finite action and 
satisfying \eref{vst356}.   
 Let $\alpha \in H_1^\As(M;\kf)$   
 Define $z(t)$ as in \eref{ve8}.
Then $z(\cdot)$ is an almost strong solution to  \eref{ve}. 
 It is a strong solution if and only if $d_{A(t_0)}\alpha \in H_1^\As$         
  for some   $t_0 >0$.
\end{lemma}
The proof will be given in Section \ref{secinit}.

\begin{theorem} \label{thmveu1e} 
Assume that $1/2 \le a <1$ and $1/2 \le b <1$. 
 Let $A(\cdot)$ be a strong solution  to the Yang-Mills heat equation over $(0, \infty)$ with finite a-action
and satisfying \eref{vst356}.   
 Let $v_0 \in H_{b}^\As(M; \L^1\otimes \kf)$. Then

 1. There exists an almost strong solution $v(\cdot)$ to the variational equation \eref{ve} 
 over $[0, \infty)$ with initial value $v_0$.

  2. For each real number $\tau >0$ there exists a vertical almost strong 
solution $d_{A(t)}\alpha_\tau$ such  that the function
\beq
v_\tau(t) 
  \equiv v(t) - d_{A(t)}\alpha_\tau,\ \ t\ge 0             \label{ve21}
\eeq
is  a strong solution  to the variational equation with initial value $v_0 - d_{A(0)} \alpha_\tau$.
Moreover 
\beq
\sup_{0 \le t \le 1} \|v(t) - v_\tau(t)\|_2  
 \rightarrow 0\ \  \text{as} \ \ \tau\downarrow 0. \label{ve22a}
\eeq

3. If $\|A(t)\|_{L^3(M)} < \infty$ for some $t > 0$  then 
\begin{align}
v:[0,\infty) \rightarrow H_b^{\As}          \label{ve20}
\end{align}
is continuous.

4.  Strong solutions are unique when they exist. 
\end{theorem}   
The proof will be given in Section \ref{secrecpf}.

\begin{remark}{\rm   Theorem \ref{thmveu1e} is the precise infinitesimal analog 
of the main theorem of \cite{G70} since a vertical vector is the  infinitesimal analog
 of a gauge transformation.
 }
 \end{remark}

 \begin{remark}\label{remweak} {\rm  
 The assertion \eref{ve20} shows that  the almost strong solution $v$
 is continuous at $t =0$ as a function into $H_b^\As$. One should expect that the strong solution $ v_\tau$ is continuous at $t =0$ 
 as a function into $H_{b}^\As$ also and not just into $L^2$. 
 But my techniques fail  to produce this result in the doubly  critical case $a=1/2, b = 1/2$.
The issue may be  
 a  conceptual one, related to the nature of the
 gauge group associated to $a=1/2$,  rather than a matter of technique.  The critical gauge group $\G_{3/2}$
 just barely fails to be a Hilbert manifold.
 See for example  \cite[Remark 5.21]{G70} for a discussion of the breakdown of smoothness
 of this gauge group, which is  associated to initial data in $H_{1/2}^\As$. This remark applies also to Theorem
 \ref{thmveu2b}.
  We will actually prove that for $1/2 \le b <1$ the strong solution $v_\tau$   
   is a continuous
  function  on $[0,\infty)$ into $L^\rho(M; \L^1\otimes \kf)$ for $2\le \rho < 3$. But    
  for $b = 1/2$    the expected
  continuity of $v_\tau$ into $H_{1/2}^\As$ implies continuity into $L^3$, by Sobolev, which we fail to achieve.
 Continuity into $L^\rho$ will be proved in Lemma \ref{lemvert6} along with a strengthening of \eref{ve22a}
 to allow Lebesgue power $\rho$ with $2 \le  \rho <3$.
 }
 \end{remark}

        \begin{remark}\label{horcomp}{\rm  (The horizontal component)
The failure of $v_\tau(t)$ to be continuous into $H_b^\As$ at $t = 0$  
 is entirely due
to the poor behavior of the vertical component of $v_\tau(t)$. Even if $v_0$ is horizontal at $A_0$,
the solution $v(t)$ rapidly acquires a large vertical component.
 By contrast, the horizontal component of $v(t)$ is well behaved. Let $\bar v(t)$ denote the horizontal
 component of $v(t)$ at $A(t)$. Then $\bar v(t)$ satisfies its own differential equation, 
 independent of the vertical component. Moreover it relates to the variation of the action function
 $\rho_A(t)$ very well. This will be developed in  a sequel to this paper, \cite{G72}.
}  
 \end{remark}

\subsubsection{Solutions of finite action} 
 
       \begin{definition}\label{defb-act}{\rm 
       ($b$-action) Let $0 \le b <1$. 
A solution $v$ to the variational equation \eref{ve}
 has     {\it  finite  strong $b$-action}  if, for some number $\tau >0$, there holds 
\beq
\int_0^\tau s^{-b} \Big(\|\n^{\As} v(s)\|_{L^2(M)}^2 + \| v(s)\|_{L^2(M)}^2\Big) ds < \infty.      \label{ymh15b}
\eeq
The integrand is gauge invariant.
}
\end{definition}

\begin{theorem}\label{thmveu2b}  
 Assume that  $1/2 \le a < 1$ and  $1/2 \le b <1$ and that $max(a,b) > 1/2$. 
Let $c = \min (a,b)$.  
      Let $A(\cdot)$ be a strong solution  to the Yang-Mills heat equation over $(0, \infty)$ with finite a-action
   such that $\|\As\|_{L^3(M)} < \infty$. 
    Let $v_0 \in H_{b}^\As(M; \L^1\otimes \kf)$, wherein  either $M= \R^3$ or else Neumann
     boundary conditions  \eref{vst150n} hold. 
 Then  the strong solution $v_\tau$, 
 constructed in  Theorem \ref{thmveu1e} for $\tau >0$,
           has finite strong c-action in the sense of Definition \ref{defb-act}. 
\end{theorem}
This will be proved in Sections \ref{secrecab} and \ref{secrecpf}.

 See Remark \ref{remweak} for a discussion of the failure of Theorem \ref{thmveu2b} in the doubly critical case
 $a = b = 1/2$.

\begin{remark} {\rm Theorem \ref{thmveu2b} is the only theorem in this paper 
in which Dirichlet boundary conditions  fail to be encompassed  by our techniques.
}
\end{remark}

\subsection{The infinitesimal ZDS procedure}
\begin{remark} {\rm  
 The  Zwanziger, Donaldson, Sadun \cite{Z, Do1,  Sa} (ZDS) 
method for proving
 existence of solutions to the Yang-Mills heat equation consists in modifying the 
 equation so as to make it strictly parabolic and then recovering a solution to the
 Yang-Mills heat equation
 itself from the solution to the modified equation by making a time dependent gauge
  transformation. 
 See   \cite{CG1} or \cite{G70} for a more detailed description.
 In this paper we are  going to use an infinitesimal version of the ZDS procedure.
 To this end we  first modify the variational equation \eref{ve} by adding a term
  that makes it strictly parabolic.
 }
 \end{remark}

\begin{definition}{\rm The {\it augmented variational equation} for a time dependent $\kf$ valued
1-form  $w(t)$ over $M$ is
\beq
-w'(t) = (d_A^*d_A +d_Ad_A^*)w + [w\lrc B].      \label{av1}
\eeq
Here $A(\cdot)$ is a solution to \eref{ymh10}. The time dependence of $A, B$ and $w$ on the right side is suppressed.
}
\end{definition}

\begin{theorem}\label{thmwe} Suppose that $A(\cdot)$ is a strong solution to the Yang-Mills heat 
equation over $[0, \infty)$ of finite action. 
Let $0 < b <1$.   
 Assume that $v_0 \in H_b^\As(M;\L^1\otimes \kf)$. 
Then there exists a continuous function
 \beq
 w:[0, \infty) \rightarrow H_b^\As   \notag
 \eeq 
 such that $w(0) = v_0$ and
 
\   a$)$\ \  $w$ is a strong solution to the augmented variational equation \eref{av1} over $[0, \infty)$ 
satisfying the boundary conditions  \eref{vst150d} resp. \eref{vst150n} in case $M \ne \R^3$,

\ b$)$\ \  $t^{1-b} \| w(t)\|_{H_1^\As}^2 \rightarrow 0$ as $t \downarrow 0$.

\noindent
The solution is unique under the preceding conditions. Moreover
\begin{align}
c)\ \ \int_0^T s^{-b} \|w(s)\|_{H_1^\As}^2\, ds  \le \gamma_T \| v_0\|_{H_b^\As}^2
                                                                        \ \qquad\qquad\qquad \qquad\qquad   \label{av5}
\end{align}
for all $T \ge 0$ and for some constant $\gamma_T$ depending only on $T$ and $\rho_A(T)$.
\end{theorem}
Parts a) and b) of this theorem   will be proven in Section \ref{secEUaug}.  
Part c) will be proven in Section \ref{secboa}.

In the infinitesimal ZDS procedure we recover the solution $v$ to \eref{ve} by simply adding
on to $w$ an appropriate vertical correction 
 as follows.

 \begin{theorem}\label{thmrec1} $($Recovery theorem$)$ 
Assume that  $1/2 \le a < 1$ and $1/2\le b <1$. 
 Let $A$ be a solution to the YM heat equation \eref{ymh10}  over $[0, \infty)$ with finite a-action. 
 Let $v_0 \in H_{b}^\As(M)$.
 Let $w(s)$ be the strong solution to the augmented variational equation \eref{av1} 
 over $[0,\infty)$ with  initial value $v_0$, satisfying the conclusions a$)$ and b$)$ of Theorem \ref{thmwe}. 
 Define 
 \beq
  v(t) = w(t) + d_{A(t)}  \int_0^t d_{A(s)}^* w(s) ds \ \ \ \text{for}\ \ 0\le  t <  \infty.    
            \label{rec1}
\eeq
Then $v(\cdot)$ is an almost strong solution to the variational equation with initial value $v_0$.

Let $\tau >0$ 
and define
\beq
v_\tau(t)     
 =  w(t) + d_{A(t)}  \int_\tau^t d_{A(s)}^* w(s) ds \ \ \ \text{for}\ 0 < t <  \infty. 
                       \label{rec2}
\eeq 
Then $v_\tau$ 
 is a strong solution to the variational equation over $(0, \infty)$. 
  Let
\beq
\alpha_\tau = \int_0^\tau   d_{A(s)}^* w(s) ds.                                                          \label{rec3}
\eeq
Then $\alpha_\tau \in H_{1}^\As$ and the function 
\beq
t\mapsto d_{A(t)} \alpha_\tau, \ \ \  \ \ 0< t <  \infty          
            \label{rec4}
\eeq
is an almost strong  vertical solution to the variational equation. 
Moreover
\begin{align}
 v_\tau(t)   
 &= v(t) -   d_{A(t)} \alpha_\tau \ \  \text{for}\ \  0 < t<  \infty 
 \ \ \text{and}     \label{rec5} \\
 v_\tau(t) 
 &\ \ \text{converges to} 
 \ \ v_0 - d_{A_0} \alpha_\tau \ \ \  \text{in}\ \ L^2(M)\ \ \ \text{as} \ \ t\downarrow 0. \label{rec6}
\end{align}
 If  $\| \As\|_3 < \infty$  
 then
\begin{align}
v: [0,  \infty)  \rightarrow H_b^\As(M; \L^1\otimes \kf)                     \label{rec1.1}
\end{align}
is continuous.  In particular, $\|v(t) - v_0\|_{H_b^\As} \rightarrow 0$ as $ t\downarrow 0$.
\end{theorem}

We will refer to the  second term on the right in \eref{rec1} as  the {\it vertical correction}
 to the solution $w$ of the augmented variational equation \eref{av1}. It is not by itself a solution to the
 variational equation because $d_{A(t)}$ is  applied to a time dependent function.

The proof of Theorem \ref{thmrec1} will be given in Section \ref{secrecpf}.

          \begin{remark}\label{remsig}{\rm (Significance of $\rho_A(\tau)$) 
The functional $\rho_A(\tau)$ 
 defined in \eref{ymh15a} is a gauge invariant function of the initial data $A_0$
for the Yang-Mills heat equation for each $\tau >0$. 
 It therefore descends to a function on the quotient space $\Y$,
heuristically defined in \eref{I6}. Its significance can be understood by computing its value
in case $K= S^1$ when  $A_0$ lies in the  section of this bundle corresponding to the Coulomb gauge.
         In this case, after multiplying by
   $\sqrt{-1}$, we can take $A(t)$ to be a real valued 1-form on $\R^3$ for each
    $t \ge 0$. Since the magnetic field is now given by $B(t) = dA(t)$, the Yang-Mills heat equation
    reduces to the Maxwell  heat equation   $\p A(t)/\p t = -d^*d A(t)$.        
    The identity 
    $(\p/\p t) d^*A(t) = d^* (\p A(t)/\p t) = -d^* (d^*dA(t)) = 0$
    shows that if the initial
    data $A_0$ is in the Coulomb gauge, i.e., $d^*A_0 = 0$, then so is $A(t)$.    
  The Coulomb gauge space is therefore invariant under the Maxwell heat flow.
   Moreover the Maxwell heat equation 
   reduces to  $\p A(t)/\p t = \Delta A(t)$ for functions $A(t)$  in the Coulomb gauge because
   $-\Delta A(t) = (d^*d + dd^*) A (t) = d^*d A(t).$      
   Hence if $A_0$ is in Coulomb gauge then the solution to the Maxwell heat equation is simply given by  
   $A(t) = e^{t\Delta} A_0$.  We can compute $\rho_A(\tau)$ easily in this case:
   Since $d^*A(t) =0$, we have   $\|B(t)\|_2^2 = \|dA(t)\|_2^2 + \| d^*A(t)\|_2^2 =
   (-\Delta A(t), A(t)) = (-\Delta e^{2t\Delta} A_0, A_0)$. 
   Therefore, using the spectral theorem and the identity 
   $\int_0^\infty t^{-1/2} x e^{-2tx} dx = c_1 x^{1/2}$, we find, for $a = 1/2$
   \begin{align*}
   \rho_A(\tau)&=\int_0^\tau t^{-1/2} \|B(t)\|_2^2 dt\\
   &=\int_0^\tau t^{-1/2}  (-\Delta e^{2t\Delta} A_0, A_0)dt\\
   & = c_1 ((-\Delta)^{1/2} A_0, A_0)  -\int_\tau^\infty  t^{-1/2}  (-\Delta e^{2t\Delta} A_0, A_0)dt \\ 
   &= c_1 \|A_0\|_{H_{1/2}}^2+   O(\tau^{-1/2} \|e^{\tau \Delta}A_0\|_2^2) 
   \end{align*}
 Thus $\rho_A(\tau) $ gives the $\dot H_{1/2}$ norm of $A_0$ in the Coulomb gauge,
  exactly for $\tau = \infty$  and qualitatively for  finite $\tau >0$.   
 In this example gauge transformations are given by $A_0 \mapsto A_0 + d\lambda$, with $\lambda$
   a  real  valued function  on $\R^3$. Since the Coulomb gauge space is orthogonal to the 
    exact 1-forms it provides a section for the quotient space  $\{all\  A_0\}/ \{exact\ A_0\}$. 
    Thus  $\rho_A(\tau)$     
    descends to a function on the quotient space for each $\tau$, while at the same time 
    giving the $H_{1/2}$ norm, locally,  of the lift to the Coulomb section. 
}
\end{remark}

\section{Solutions for the augmented variational equation }\label{secEUaug}

     In this section we will prove existence and uniqueness of solutions to the augmented
variational equation \eref{av1} over a short time interval. The space $M$ on which the initial data sits
will be $\R^3$ or a bounded region in $\R^3$ with smooth boundary. In the latter case we will impose
Dirichlet or Neumann boundary conditions on the solution. 
A standard procedure for analyzing the equation \eref{av1} consists in  separating
  out the second order terms in \eref{av1} and writing the differential
 equation  as an equivalent integral equation, whose solution is  then established by a contraction principle.
 But this procedure, as stated, would lead to the equation $(d/dt)w(t) = \Delta w(t) + K_0(t) w(t)$,
  where $\Delta := -(d^*d + dd^*)$  is the Laplacian on $\kf$ valued 1-forms and $K_0(t)$ is a 
  first order differential operator.
 Unfortunately the coefficients in $K_0(t)$ depend on the         
  gauge potential $A(t) $ and its derivatives, which become highly singular as $t\downarrow 0$.
  As a result, the bounds needed to show that the related integral
  operator is a contraction in the relevant space fail.   Instead, we are going to separate
   out a gauge covariant   version of the Laplacian. Let $T >0$ and let $\As = A(T)$. Denote by
   $\Delta_\As$ the Bochner Laplacian defined in Section \ref{secstate}. Then we may write \eref{av1}
   in the form  $(d/dt)w(t) = \Delta_\As w(t) + K(t) w(t)$, where $K(t)$ is a first order differential operator
   whose coefficients  depend on the difference $\alpha(t) := A(t) - A(T)$ and its covariant derivative.
   We will see that the required bounds on $K(t)$ depend on bounds on $\alpha(t), 0 < t \le T,$ which in 
   turn depend on $T$ being small. We have thereby a circumstance in which the unperturbed linear
   operator   $\Delta_\As$ itself depends on the time needed for contraction.

  In order to carry this out it will
  be necessary to have detailed information about the nature of the singularity of $A(t)$ near $t =0$.
  The required initial behavior bounds on $A(t)$ and its derivatives will be derived in Section 
  \ref{secibA}, after the overall strategy is explained  in Section \ref{secps}.

\subsection{Path space and the integral equation } \label{secps}

\begin{notation}\label{notSob}{\rm  $A(t)$ will continue to denote  a strong solution to the Yang-Mills
 heat equation  \eref{ymh10} over $[0, \infty)$ with curvature $B(t)$.
 The gauge covariant exterior derivatives and co-derivatives that we need to use 
 were informally described in Notation \ref{wedcomm}.
We will elaborate here on their domains.
 Suppose that $A$ is a $\kf$ valued connection form on the closure $M$ of a region
 in $\R^3$ and lying in $W_1(M; \L^1\otimes \kf)$. In case $M = \R^3$ we define $d_A$ as the closure 
 in $L^2(\R^3; \L^p\otimes \kf)$ of the operator 
 $C_c^\infty(\R^3;\L^p\otimes \kf) \ni \w \mapsto d\w + [A\wedge \w]$.
  In case $M$ is the closure of a bounded open set in $\R^3$
 there are two versions of $d_A$ of interest to us. Corresponding to Dirichlet boundary conditions 
 (aka relative boundary conditions, \cite{RaS}), $d_A$ will denote the closure in $L^2(M; \L^p\otimes \kf)$ 
 of the operator    $C_c^\infty(M^{int};\L^p\otimes \kf) \ni \w \mapsto d\w + [A\wedge \w]$.
 This is the minimal version of $d_A$. Corresponding to Neumann boundary conditions 
 (aka absolute boundary conditions,  \cite{RaS}), $d_A$ will denote the closure in $L^2(M; \L^p\otimes \kf)$ 
 of the operator    $C^\infty(M;\L^p\otimes \kf) \ni \w \mapsto d\w + [A\wedge \w]$. 
 This is the maximal version of $d_A$. In all three cases $d_A^*$ denotes the Hilbert space adjoint.
 The Hodge Laplacian on $\kf$-valued p-forms over $M$ is $-(d_A^* d_A + d_A d_A^*)$. 
 This expression determines   a self-adjoint operator. But in this section we will more often want to
 use the Bochner Laplacian, which is defined by
 \begin{align}
\Delta_A = \sum_{j=1}^3 (\p_j^A)^2 ,     \label{ps9}
\end{align}
where $\p_j^A$ is defined in \eref{ps8}.
By the  Bochner-Weitzenb\"ock formula  we may write  for a $\kf$ valued 1-form 
\beq
(d_A^* d_A  +d_Ad_A^*)\w = -\Delta_{A} \w + [\w\lrc B].        \label{vst12}
\eeq
In all cases of interest to us, when using this formula, the curvature $B$ will be bounded.  
Consequently the operator $\w\mapsto [\w\lrc B]$ is  
 a bounded operator on $L^2$. 
We can therefore take  the closed version of \eref{ps9} to be given by \eref{vst12}. 
Its domain is the same as that of   $d_A^* d_A  +d_Ad_A^*$.
Then $\Delta_A$ is a self-adjoint operator with this domain.  Thus if $M \ne \R^3$ then the Bochner
Laplacian has a Dirichlet version and Neumann version on $\kf$ valued 1-forms.
On 0-forms both Laplacians are given by $d_A^*d_A$.

   Further discussion of the operators $d_A$, $d_A^*$ and their associated boundary conditions
 may be found in \cite[Section 3]{CG1}.   
  In the present paper we will need some  specific information
 about the boundary conditions, especially in the case of Neumann boundary conditions. 
 If  $\w$ is a $\kf$ valued 1-form over $M$ then its location in one of the following domains implies
 the boundary condition indicated. 
\begin{align}
&\text{Form domain of Dirichlet Laplacian}\ \Delta_A:\ \ \ \w_{tan} =0  \label{vst150d}\\
 &\text{Form domain of Neumann Laplacian}\ \Delta_A:\ \ \w_{norm} =0  \label{vst150n}\\
&\text{Domain of Neumann Laplacian}\ \Delta_A:\ \ \w_{norm} =0, \ \ \  (d_A \w)_{norm} =0  \label{vst151n}
\end{align}
It might be useful to note that if $A_{norm} =0$, which is the case of interest in dealing with Neumann
boundary conditions for the variational equation, then the pair of conditions in \eref{vst151n} is equivalent to 
the pair of  conditions $\w_{norm} = 0, (d\w)_{norm}= 0$ because $[A\wedge \w]_{norm} =0$ when
both $A_{norm} =0$ and $\w_{norm}=0$.
}
\end{notation}

\begin{notation}\label{notA} {\rm  Choose   $T \in (0, \infty)$  and define 
\begin{align}
\As = A(T), \ \ \ \ \ \ \alpha(t) = A(t) - A(T), \ \ 0\le t \le T.         \label{ps5}
\end{align}
Then
\begin{align}
A(t) = \As + \alpha(t),\ \ \ 0\le t \le T.               \label{ps7}
\end{align}
We are going to use $\Delta_\As$ as the ``unperturbed'' Laplacian in most of this paper.
}
\end{notation}

\begin{lemma}\label{lemaug1} 
 Define a multiplication operator $M(t)$ 
on $\kf$ valued 1-forms by
\begin{align}
M(t)\w = \sum_{j=1}^3 (ad\, \alpha_j(t))^2 \w + [ div_\As \alpha(t), \w] -2[\w \lrc B(t)] .   \label{ps11}
\end{align}
Denote by $K(t)$ the  first order differential operator given  by
\begin{align}
K(t)\w = 2 [(\alpha(t)\cdot \n^\As) \w] + M(t) \w,                                 \label{ps13}
\end{align}
where  $[\alpha(t)\cdot \n^\As \w] =\sum_{j=1}^3 [\alpha_j(t), \p_j^{\As} \w]$. 
Then the augmented variational equation \eref{av1} can be written
\begin{align}
w'(t) = \Delta_\As w(t) +K(t) w(t)    \label{vst20}
\end{align}
\end{lemma}
      \begin{proof} In view of Notation \ref{notA} we may write $\p_j^{A(t)} = \p_j^\As + ad\, \alpha_j(t)$.
Suppressing  $t$ on the right we therefore find. 
      \begin{align*}
\Delta_{A(t)}\w &= \sum_{j=1}^3 (\p_j^\As + ad\,\alpha_j)(\p_j^\As + ad\,\alpha_j) \w \\
&=\sum_{j=1}^3(\p_j^\As)^2 \w + \sum_j \{  \p_j^\As [\alpha_j, \w] + [\alpha_j, \p_j^\As \w] \} 
+ \sum_j (ad\, \alpha_j)^2 \w \\
& = \Delta_\As\w       
                   + [div_\As \alpha, \w] +2\sum_j [\alpha_j, \p_j^\As \w] 
   + \sum_j (ad\, \alpha_j)^2 \w .
\end{align*}
Hence, in view of the Bochner-Weitzenb\"ock formula \eref{vst12} we may write the augmented
 variational equation \eref{av1} as
\begin{align*}
w'(t) &= -(d_A^* d_A  +d_Ad_A^*)w(t) - [w(t)\lrc B(t)] \\
&= \Delta_{A(t)} w(t) - 2 [w(t)\lrc B(t)] \\
&=\Delta_\As w(t) + 2[\alpha\cdot \n^\As w(t)]   \\
&\ \ \ \ \ \ \ \ \ +  [div_\As \alpha, w(t) ]  + \sum_j (ad\, \alpha_j)^2 w(t) - 2 [w(t)\lrc B(t)] \\
& =  \Delta_\As w(t) + K(t) w(t).
\end{align*}
\end{proof}

\begin{remark}\label{strat4} {\rm (Strategy) 
Informally, the differential equation \eref{vst20} 
together with the initial condition $w(0) = w_0$, is equivalent to the integral equation
\begin{align}
 w(t) = e^{t\Delta_\As} w_0 + \int_0^t e^{(t-s) \Delta_\As}\Big(K(s)w(s) \Big) ds.                    \label{vst31}
 \end{align}
 We will first  prove the existence and uniqueness of solutions to the integral equation
  \eref{vst31}. These are so-called ``mild'' solutions. It will then be necessary to show
  that the mild solution is actually a strong solution to \eref{av1}. 
  To this end we will   establish  bounds on  the   operator $K(t)$ which will
   allow us to prove H\"older continuity of $w(\cdot)$ and $K(\cdot)$ 
  on intervals $[\tau, T]$, with $\tau >0$,  and thereby make applicable a general theorem
 \cite[Theorem 11.44]{RR},    
 ensuring that the mild
  solution is a strong solution. The required bounds on $K(t)$ will be derived by a common method for
   the three cases   $M=\R^3$, or $M\ne \R^3$ 
   with Neumann or Dirichlet boundary conditions.  
  The three cases will be encoded into appropriate Sobolev spaces, defined by \eref{ST19} in 
  terms of the associated Laplacians.      
  In all three cases the associated Laplacian is given by
  \eref{ps9} with $A = \As$ and with appropriate boundary conditions. 
   $H_1^\As$ is  the form domain of $\Delta_\As$. 
       Thus if $M \ne \R^3$ then a form $\w \in W_1(M)$ is in the Neumann version of $H_1^\As(M)$
      if and only if $\w_{norm}=0$ and is in the Dirichlet version of $H_1^\As(M)$
      if and only if $\w_{tan}=0$. This defines two distinct notions of $H_1^\As$ in case $M \ne \R^3$.  
    See  \cite[Remark 4.10]{CG1}   for further discussion of these   domains.  
  }
\end{remark}

      \begin{remark}\label{remMar} {\rm (Marini boundary conditions.) For the solution $A(\cdot)$ there is
a third kind of boundary condition that was studied in \cite{CG1} and \cite{CG2}. It consists in setting
the normal component of the curvature of $A(t)$ to zero on $\p M$ for $t >0$. This kind of boundary
condition was first used by A. Marini \cite{Ma3,Ma4,Ma7,Ma8} for the four dimensional elliptic Yang-Mills
boundary value problem. It will be used in a future work \cite{G73} for showing that the initial value
of a finite action solution  to the Yang-Mills heat  equation over $\R^3$ is, upon restriction
to a bounded region $M$,  an allowable initial value for a solution to the Yang-Mills
heat equation over $M$. Marini boundary conditions are forced in this context. This extension
of the present work will be derived from Neumann boundary conditions in \cite{G73}. This kind of localization
theorem seems indispensable for use of  the Yang-Mills heat equation as a regularization tool in local
quantum field theory in order to take into account that signals do not propagate faster than the speed of light.
}
\end{remark}

            \begin{notation} \label{notvst7b}{\rm (Path space)       
      Let $0\le b <1$ and let $0<T <\infty$.   Define
 \begin{align}
 \Q_T^{(b)} =&\Big\{w\in C\Big([0,T]; H_b^\As(M;\L^1\otimes \kf)\Big) 
                   \cap C\Big((0,T]; H_1^\As(M;\L^1\otimes \kf)\Big)    :  \notag \\       
&\qquad\qquad  \ \ \  \limsup_{t\downarrow 0}\  t^{(1-b)/2} \|w(t)\|_{H_1^\As} =0 \Big\}.         \label{vst350b}
 \end{align}
 For $w\in \Q_T^{(b)}$ define
 \beq
 |w|_t =\sup_{0 < s \le t} s^{(1-b)/2} \| w(s)\|_{H_1^\As(M)} , \ \ \ 0 \le t \le T .        \label{vst351b}
 \eeq
 Then
 \beq
 \|w(s)\|_{H_1^\As} \le s^{(b-1)/2} |w|_t,\ \ \ 0 < s \le t \le T  .                                \label{vst352b}
 \eeq
 The space $\Q_T^{(b)}$ is a Banach space in the norm
 \beq
 \|w\|_{\Q_T^{(b)}} = |w|_T + \sup_{0\le s \le T} \| w(s)\|_{H_b^\As}         .              \label{vst353b}
 \eeq
 }
 \end{notation}

            \begin{theorem}      \label{thmmild1b} 
Assume that $A(\cdot)$  is a strong solution to the Yang-Mills heat equation over $[0, \infty)$ with finite action. 
 Suppose that $ 0 < b <1$ 
 and that $w_0 \in H_b^\As(M)$. Then the  integral equation \eref{vst31} 
has a unique solution in the path space $\Q_T^{(b)}$ for a sufficiently small $T$ depending on $A(\cdot)$.
Moreover
\beq
\|w\|_{\Q_T^{(b)}} \le c_{b,T} \|w_0\|_{H_b^\As}      \label{vst360}
\eeq
for some constant $c_{b,T}$ depending only on $b, T$ and  $\rho_A(T)$, where $\rho_A(t)$ is defined in \eref{ymh15a} 
 with $a = 1/2$.  
 \end{theorem} 
 The proof will be given in the next three sections.

\subsection{Initial behavior of $A$}      \label{secibA}

The initial behavior of various $L^p(M)$  norms of $A(t)$, $B(t)$ and their time and space 
derivatives will be needed to prove  bounds for the operator $K(t)$ defined in \eref{ps13},
and later for establishing bounds on the solution $w(t)$ to the augmented variational equation.
 These in turn will be needed
to recover the desired solution to the variational equation from $w(\cdot)$. 
In this subsection we are going to derive the required  initial behavior bounds for $A$
and its derivatives. They extend the initial behavior bounds derived in \cite{G70} and will
be used  frequently throughout the rest of this paper.
  Many of these are not gauge invariant bounds. But their proofs depend on the
    gauge invariant bounds derived in \cite{G70}.
    
  We reiterate that $M$ can be chosen  to be all of $\R^3$ or to be a bounded subset with
  Dirichlet  or Neumann boundary conditions as in Section  \ref{notSob}.  
    
    The next lemma summarizes some  of the initial behavior bounds for $A$ 
established in \cite{G70}. Recall the notation from \eref{ymh15a}: 
$ \rho_A(t) = (1/2)\int_0^t s^{-a} \| B(s)\|_2^2 ds$.

          \begin{definition}\label{defsbf} {\rm   (Standard dominating function)
By a {\it standard dominating function} we mean a continuous function 
$C:[0, \infty)^2 \rightarrow [0, \infty)$ which is zero at $(0,0)$ and non-decreasing in each
variable.  On the right hand side of each of the following bounds is 
a function of time and of $A$ of the form $C(t, \rho_A(t))$ for some standard dominating 
 function $C(\cdot, \cdot)$.
 All of the  bounds are gauge invariant. 
 }
\end{definition}

\begin{lemma}\label{lemiby} 
Let $1/2 \le a <1$.  If $A(\cdot)$ is a strong solution to the 
      Yang-Mills heat equation over $(0, \infty)$  
   with finite a-action then  there 
      exist standard dominating functions $C_j$ such that
   \begin{align}
   \sup_{0< s \le t} s^{1-a} \|B(s)\|_2^2 &\le C_1(t, \rho_A(t))    \label{iby1} \\                                                                                                 
   \sup_{0< s \le t}s^{2-a} \|A'(s)\|_2^2   &\le C_2(t, \rho_A(t))   \label{iby2}\\                                                                                                    
   \sup_{0< s \le t} s^{2-a} \|B(s)\|_6^2 &\le   C_3(t, \rho_A(t))    \label{iby3}\\                                                                                                   
   \sup_{0< s \le t} s^{3-a} \|B'(s)\|_2^2& \le  C_4(t, \rho_A(t))    \label{iby4}\\    
   \sup_{0< s \le t} s^{3-a} \|A'(s)\|_6^2& \le   C_5(t, \rho_A(t))   \label{iby5}\\  
   \sup_{0< s \le t} s^{(3/2) -a} \|B(s)\|_3^2        &\le    C_6(t, \rho_A(t))   \label{iby6}\\
   \sup_{0< s \le t} s^{(5/2) -a} \|A'(s)\|_3^2      & \le  C_7(t, \rho_A(t))      \label{iby7}\\
   \int_0^t s^{1-a} \|A'(s)\|_2^2 ds &\le C_8(t, \rho_A(t))           \label{iby8}\\
   \int_0^t s^{2-a} \|A'(s)\|_6^2 ds & \le C_9(t, \rho_A(t))          \label{iby9}  \\   
   \int_0^t s \|A'(s)\|_3^2 ds &\le  C_{10}(t, \rho_A(t))                  \label{iby10}\\
   \int_0^t s^{1-a} \|B(s)\|_6^2 ds &\le C_{11}(t, \rho_A(t))        \label{iby11} \\
   \int_0^t   s^{2-a} \|B'(s)\|_2^2 ds & \le C_{12}(t, \rho_A(t))  \label{iby12} \\
   \int_0^t s^{3-a} \|B'(s)\|_6^2 ds   &  \le C_{13}(t, \rho_A(t))   \label{iby13} \\
   \int_0^t \|B(s)\|_3^2 ds & \le C_{14}(t, \rho_A(t))    \label{iby14}\\
      \sup_{0<s \le t} s^{3/2}\|A'(s)\|_\infty &\rightarrow 0\ \ \text{as}\ \ t\downarrow 0 \label{iby15} \\
    \sup_{0<s \le t} s\|B(s)\|_\infty &\rightarrow 0\ \ \text{as}\ \ t\downarrow 0   \label{iby16} \\
      \int_0^t s \|B(s)\|_\infty^2 ds &< \infty \label{iby17}
   \end{align}
   \end{lemma}        
          \begin{proof} The inequalities \eref{iby1} - \eref{iby5} and \eref{iby8}, \eref{iby9}, \eref{iby11},
          \eref{iby12}, \eref{iby13} are taken directly
    from the first, second and third order initial behavior estimates of \cite[Section 7.2]{G70}.
    The assertions \eref{iby15}, \eref{iby16} and \eref{iby17} are taken from \cite[Proposition 7.19]{G70}.
    These three assertions can be improved when $a > 1/2$. But we will only need them for $a = 1/2$.
    They can be formulated in terms of  bounds by standard dominating functions.
    The  remaining four inequalities involve $L^3$ norms and follow by interpolation thus:
    For $0 < s \le t$ one has
    \begin{align*}
    s^{(3/2) - a} \|B(s)\|_3^2 &\le \Big(s^{(1-a)/2} \|B(s)\|_2\Big) 
                            \Big(s^{(2-a)/2} \|B(s)\|_6\Big) \le (C_1 C_3)^{1/2}|_t\\
    s^{(5/2) - a} \|A'(s)\|_3^2 &\le \Big( s^{(2-a)/2} \|A'(s)\|_2 \Big) \Big( s^{(3-a)/2} \|A'(s)\|_6\Big) 
    \le (C_2 C_5)^{1/2}|_t.
    \end{align*}
    by \eref{iby1}, \eref{iby3} and then \eref{iby2}, \eref{iby5}.
  The inequalities \eref{iby6} and \eref{iby7} follow.
     Interpolation also shows that 
    \begin{align}
 \int_0^t s \|A'(s)\|_3^2 ds &\le \int_0^t s^{a- (1/2)} \(s^{(1-a)/2} \|A'(s)\|_2\) \( s^{(2-a)/2} \|A'(s)\|_6\) ds \notag\\
 &\le t^{a - (1/2)} \(\int_0^t s^{1-a} \|A'(s)\|_2^2 ds \)^{1/2} \( \int_0^t s^{2-a} \|A'(s)\|_6^2 ds \)^{1/2}\notag\\
 &\le t^{a - (1/2)} C_8^{1/2} C_9^{1/2},              \label{ibA42a}
\end{align}  
which is \eref{iby10} with $C_{10}(t, \rho_A(t)) = t^{a - (1/2)} (C_8 C_9)^{1/2}$.
    Similarly,  
  \begin{align}
    \int_0^t &\|B(s)\|_3^2 ds \le \int_0^t(s^{a- (1/2)})(s^{-a/2} \|B(s)\|_2) (s^{(1-a)/2} \|B(s)\|_6)ds    \notag \\
    &\le t^{a - (1/2)} \Big( \int_0^t s^{-a} \|B(s)\|_2^2 ds \Big)^{1/2}
     \Big( \int_0^t s^{1-a} \|B(s)\|_6^2 ds \Big)^{1/2}  \notag \\
     &\le   t^{a - (1/2)}\, (2\rho_A(t))^{1/2}\, C_{11}(t, \rho_A(t))^{1/2},       \label{rec848e}
     \end{align}   
     which establishes \eref{iby14}.
   \end{proof}

\bigskip
\noindent
Note: Whereas all  the bounds in Lemma \ref{lemiby}  are gauge invariant, 
the only gauge invariant bounds in the next theorem are \eref{ibA5}, \eref{ibA6}, \eref{ibA15c} and \eref{ibA72}.
In most of the non-gauge invariant inequalities a non-gauge invariant condition
 is imposed on $A(T)$ for some $T >0$, but  the last quantifier is omitted to save space.
         \begin{theorem}\label{thmibA1}  
Let $1/2 \le a <1$. Assume that $A$ is a strong solution to
 theYang-Mills heat equation \eref{ymh10} over $(0, \infty)$ with finite  $a$-action. 
Then

$L^6$ inequalities.
\begin{align}
&1.\ \  s^{1-a} \|A(s) - A(r)\|_6^2 \le C_{21}(r, \rho_A(r)), \ \ 0  <s \le r <\infty.               \label{ibA5} \\
&2.\ \int_0^t s^{-a} \| A(s)- A(t)\|_6^2 ds \le  C_{22}(t, \rho_A(t)),  \ \   0 \le t  < \infty.    \label{ibA6} \\ 
& 3.\ \  a_t := \sup_{0<s \le t}\  s^{(1-a)/2} \| A(s)\|_6 \rightarrow 0\ \ 
       \text{as}\ \ t\downarrow 0 \ \ \text{if}\ \ \|A(T)\|_6 < \infty                                         \label{ibA5a} \\
&4.\ \ \hat a_t :=    \int_0^t s^{-a} \| A(s)\|_6^2 ds < \infty\  \text{if}\  \|A(T)\|_6 < \infty   \label{ibA6a} \\
&\     \notag \\
&\ \ L^3 \text{ inequalities. Assume that $1/2 \le a <1$ and $\|A(T)\|_6 < \infty$. Then} \notag \\      
 &5. \ \  \int_0^t s^{1-2a}\|d^*A(s) - d^*A(t)\|_3^2 ds \le  (1-a)^{-2} a_t^2 C_9(t, \rho_A(t)) \label{ibA15b}\\        
 &6. \ \  t^{1/2}\| d_\As^* (A(t) - A(T)) \|_3 \le C_{24}(T, \rho_A(T)),\ \  0<t \le T\ \text{if}\ a = 1/2  \label{ibA15c} \\
&7. \ \ \int_0^t   s^{(1/2) -a}    \|d(A(s) - A(t))\|_3^2 ds  \le  C_{32}(t, \rho_A(t)) + a_t^2C_{33}(t, \rho_A(t))
                                                                   \label{ibA10a} \\
&8.
\ \ \limsup_{t\downarrow 0 }\ t^{(3-2a)/4}  \| d A(t)\|_3  = 0,\ \text{and}\ \   
\limsup_{t\downarrow 0 }\ t^{1/2}  \| d A(t)\|_3  = 0                                             \label{ibA14}  \\
&\    \notag \\
&\ \ L^\infty \text{inequality.} \notag \\         
&9. \ \  t^{1/2} \|  A(t) - A(T)\|_\infty    \le  C_{25}(T, \rho_A(T)) ,\ \ 0 < t \le T\ \text{if}\ a =1/2  \label{ibA72}\\
 &\     \notag\\
 &\ \  L^2 \text{ inequalities.     Assume that $1/2 \le a <1$ and $\|A(T)\|_6 < \infty$. Then} \notag \\
&10. \ \ \int_0^t s^{-a} \| d^*A(s) -d^*A(t) \|_2^2ds \le a_t^2 C_{40}(t, \rho_A(t))               \label{ibA10b}\\
 &11.    \ \   \int_0^t s^{-a} \| dA(s) -dA(t) \|_2^2ds \rightarrow 0\ \text{as}\ t\downarrow 0.  \label{ibA10c} \\
 &12. \ \ \ \int_0^T s^{-a}\| dA(s) \|_2^2ds < \infty\ \ \ \text{if} \ \ A(T) \in H_1\                     \label{ibA10c1} \\
 &13.  \ \ \ \ \sup_{0 < s \le t} s^{1-a}   \| d^*A(s) -d^*A(t) \|_2^2 \rightarrow 0\
                                                                                     \ \text{as}\ t\downarrow 0.   \label{ibA10d}    \\
&14. \ \ \ \ \sup_{0 < s \le t} s^{1-a}   \| dA(s) -dA(t) \|_2^2 \rightarrow 0\ \ \
                                                                                      \ \text{as}\ t\downarrow 0.   \label{ibA10e}                                                                                                                      
 \end{align}
 \begin{align}
&\ \ \text{$H_1$ inequalities.     Assume that  $\|A(T)\|_6 < \infty$ and that  $M = \R^3$. Then} \notag\\
&15.\ \  \ \infty >   \int_0^t s^{-a} \| \n (A(s)- A(t)) \|_2^2 ds  \rightarrow 0\ \ 
                                                                 \text{as}\ \ t \downarrow 0.       \label{ibA12}\\
&16. \ \ \ \infty >     \sup_{0 < s \le t} s^{(1-a)} \| \n (A(s) - A(t))\|_2^2 \rightarrow 0\ \ 
                                                                 \text{as}\ \ t \downarrow 0 \label{ibA13a} \\
&17. \ \ \   \int_0^T s^{-a} \| \n A(s) \|_2^2 ds < \infty\ \ \text{if}\ \ \|\n A(T)\|_2 < \infty  
                                                                                                         \label{ibA13b}\\
&18.  \ \ \limsup_{t\downarrow 0}\  t^{(1-a)/2} \| \n A(t)\|_2 = 0 \ \ 
                                                     \text{if}\ \ \|\n A(T)\|_2 < \infty.   \label{ibA13}       
\end{align}
\end{theorem}

       \begin{remark}
{\rm  The inequality \eref{ibA14} has no analog for $d^*$ because $\|d^* A(t)\|_3$
need not be finite for any $t >0$ under our hypotheses. 
For example let $K = S^1$ and take $A_0 = d\lambda$ with $d^*d\l \in L^2(\R^3)$
but $d^*d \l \notin L^3(\R^3)$. 
    Then  $A(t): = d\l \in H_1$ for all $t \ge 0$ and  is a strong solution to the Yang-Mills heat equation. 
In this case $dA(t) =0$ but  $d^*A(t) \notin L^3(\R^3)$ for any $t\ge 0$. This is a pure gauge solution.
\eref{ibA15b} and \eref{ibA15c} show that ``pure gauge contributions'' to 
 differences, such as $\|d^* A(s) - d^*A(t)\|_3$,   cancel  to some degree.
 }
\end{remark}

\bigskip
The proof of Theorem \ref{thmibA1} depends on  the following  special case of  the
 generalized Hardy's inequality \cite[Theorem 6.1.4]{SiPt3}.

\begin{lemma}\label{HI} $($Hardy's inequality.$)$
Let $g:(0, \infty)\rightarrow \R$ be locally integrable. 
 Suppose that $0 < T <\infty$. Define
\beq
G(t) = \int_t^T g(s)ds, \ \ \ 0<t \le T.   \label{H0}
\eeq
Let     $- \infty < \beta < 1$. Then
\begin{align} 
\int_0^T t^{-\beta} G(t)^2 dt \le \frac{4}{(1-\beta)^2} \int_0^T s^{2 -\beta} g(s)^2 ds,\ \  \label{H2}
\end{align}
and, if  
$h:(0, \infty) \rightarrow [ 0, \infty)$ is differentiable,
then
\begin{align}
\int_0^T s^{-\beta} \(h(s) - h(T)\)^2 ds
    &\le \frac{4}{(1- \beta)^2} \int_0^T s^{2-\beta} h'(s)^2 ds.      
    \label{H3}  
\end{align}   

\end{lemma}
         \begin{proof}  We will derive the inequality  \eref{H2} 
from   the Generalized Hardy inequality, \cite[Equ. (6.1.31)]{SiPt3}. 
  Suppose that $f:(0, \infty) \rightarrow \R$ is locally integrable and vanishes off a bounded interval.
Define
\begin{align}
F(t) = \int_t^\infty s^{-1}f(s) ds, \ \ t >0     \label{H4}
\end{align}   
Let $\alpha >0$. 
The generalized Hardy inequality    \cite[Equ. (6.1.31)]{SiPt3},  with $\theta =0$ and $p =2$,
 asserts that
\begin{align}
         \int_0^\infty t^{2\alpha - 1} F(t)^2 dt \le \alpha^{-2} \int_0^\infty s^{2\alpha -1} f(s)^2 ds,    \label{H5}
         \end{align}
         as one sees from   \cite[Equ. (6.1.29)]{SiPt3}, with $\theta =0$,  because, in the notation of 
         \cite{SiPt3},      we have   $(K_{\alpha, 0}f)(t) = t^\alpha F(t)$.

Now in \eref{H4} let $f(s) = s g(s)$ for $0<s \le T$ and let $f(s) =0$ for $s >T$. 
Then $F(t) = G(t)$ for $0 < t \le T$  and $F(t)=0 $ for $t >T$.
The two integrands in \eref{H5} are therefore zero off the interval $(0,T]$ and  the integrals
 really extend over the interval $(0, T]$.
Put $\alpha = (1-\beta)/2$. Then $\alpha > 0$ and $2\alpha -1 = -\beta$.
 \eref{H5}  now reduces to \eref{H2}.
 
 For the proof of \eref{H3} choose $g(s)= h'(s)$. Then $G(t) = h(T) - h(t)$ and \eref{H3} follows.
   \end{proof}

\bigskip
\noindent
 \begin{proof}[Proof of Theorem \ref{thmibA1}]
 
 \bigskip
 \noindent 
  {\bf Proof of 1.}  
     For $ 0 < s \le r$ we may write  $A(s) - A(r) = - \int_s^r  A'(\sigma) d\sigma$.
  Hence 
 \begin{align*}
 \| A(s) - A(r)\|_6^2 &\le \(\int_s^{r} \|A'(\sigma)\|_6 d\sigma\)^2 \\
& \le \(\int_s^r \s^{(a-2)/2}\{\sigma^{(2-a)/2} \|A'(\sigma)\|_6\} d\s\)^2 \\
&\le \int_s^{r} \s^{a-2} d\s \int_0^{r}  \s^{2-a} \| A'(\s)\|_6^2 d\s \\
&\le (1-a)^{-1}(s^{a-1} - r^{a-1}) C_9(r,\rho_A(r)).
 \end{align*}
 Therefore      $s^{1-a}\| A(s) - A(r)\|_6^2 \le  C_9(r,\rho_A(r))/(1-a)$, 
 which proves \eref{ibA5}.

\bigskip
\noindent
{\bf Proof of  2.} Choose $T = t$ in  \eref{H3} 
 and let $h(s) = \|A(s) - A(t)\|_6$. Then $h(t) =0$ and 
 $|h'(s)| \le \|A'(s)\|_6$.  Choosing $\beta =a$, we have
 $s^{2-\beta}h'(s)^2 \le s^{2-a} \| A'(s)\|_6^2$. 
  \eref{H3} now shows that
  \begin{align}
  \int_0^t s^{-a} \| A(s) - A(t)\|_6^2 ds \le \frac{4}{(1-a)^2} \int_0^t s^{2-a} \|A'(s)\|_6^2 ds. \label{ibA22}
    \end{align}
    From this and \eref{iby9} we see that \eref{ibA6} holds with $C_{22} = 4(1-a)^{-2} C_9$.

 \bigskip
 \noindent
 {\bf Proof of 3.} It follows  from \eref{ibA5} that 
  \begin{align}
s^{(1-a)/2}\| A(s)\|_6 \le \sqrt{ C_{21}(r,\rho_A(r))} + s^{(1-a)/2}\|A(r)\|_6.    \label{ibA23}
\end{align}
Take $r= T$ to conclude that if $\|A(T)\|_6 < \infty$ then $\|A(s)\|_6 < \infty$
 for $ 0 < s \le T$. Now choose a small $r$ and observe that if  $0 <t \le r$ then 
 \beq
a_t = \sup_{0< s \le t} s^{(1-a)/2}\| A(s)\|_6  \le \sqrt{ C_{21}(r,\rho_A(r))}  + t^{(1-a)/2}\|A(r)\|_6. \notag
\eeq
Hence $\lim_{t\downarrow 0}a_t \le  \sqrt{ C_{21}(r,\rho_A(r))} $,  
which is small for small $r$  because $C_{21}$ is a standard dominating function. 
This proves   \eref{ibA5a}.

\bigskip
 \noindent
 {\bf Proof of 4.}  By \eref{ibA6} we have
 \begin{align}
 \( \int_0^t& s^{-a} \| A(s)\|_6^2 ds\)^{1/2} 
              \le \sqrt{C_{22}(t, \rho_A(t))}  +\(\int_0^t s^{-a} \|A(t)\|_6^2 ds\)^{1/2}.        \notag
 \end{align}
 The last term is finite because $\|A(t)\|_6 < \infty$ by \eref{ibA5} and $a < 1$.

\bigskip
\noindent
{\bf Proof of 5.} 
The Yang-Mills heat equation  $A' = -d_A^* B$  together with the Bianchi identity show that 
 $d_A^* A'  = -(d_A^*)^2 B = 0$. Hence  
 \beq
 d^*A'(s) = -[A(s)\lrc A'(s)].         \label{ibA40}
 \eeq  
 In \eref{H3} choose  $\beta = 2a -1$,  
 $T =t$ and $h(s) = \|d^*A(s) - d^*A(t)\|_3$.   
From \eref{ibA40} we find 
$|h'(s)| \le \| d^*A'(s)\|_3 \le \|\, [A(s)\lrc A'(s)] \, \|_3 \le c \|A(s)\|_6 \|A'(s)\|_6$. 
Therefore
\begin{align}
s^{2-\beta} h'(s)^2 &\le c^2 \(s^{1-a}\|A(s)\|_6^2 \) \( s^{2-a} \|A'(s)\|_6^2\)  \notag\\
&\le c^2 a_t^2 \( s^{2-a} \|A'(s)\|_6^2\) \notag
\end{align}
Then \eref{H3},  with $\beta =2a -1$ and $T = t$, gives
\begin{align}
\int_0^t s^{1-2a}\|d^*A(s) - d^*A(t)\|_3^2 ds &\le (1-a)^{-2} \int_0^t a_t^2s^{2-a} \| A'(s)\|_6^2 ds \notag\\
&\le  (1-a)^{-2} a_t^2 C_9(t, \rho_A(t))  \notag
\end{align}
by \eref{iby9}.   This proves \eref{ibA15b}.

\bigskip
\noindent
{\bf Proof of 6.} 
 The derivation of \eref{ibA40} shows that  
     \begin{align}
     0 = d_{A(s)}^* A'(s) = d_\As^* A'(s) + [\alpha(s)\lrc A'(s)].    \notag
     \end{align}
     Since $A'(s) = \alpha'(s)$ we may write the previous identity as
     \begin{align}
     d_\As^* \alpha'(s) = -  [\alpha(s)\lrc A'(s)].                \label{ibA69}
     \end{align}
     Using  $\alpha(T) =0$ we therefore find
     \begin{align}
     d_\As^* \alpha(t) = \int_t^T    [\alpha(s)\lrc A'(s)] ds.        \label{ibA70}
     \end{align}
     Hence
     \begin{align}
     \|d_\As^* \alpha(t)\|_3 &\le c\int_t^T \|\alpha(s)\|_6 \|A'(s)\|_6 ds \notag \\
     &=c \int_t^T s^{-1/2}\Big(s^{-1/4} \|\alpha(s)\|_6 \Big) \Big( s^{3/4}\|A'(s)\|_6\Big) ds \notag\\
     &\le c t^{-1/2} \Big(\int_0^T s^{-1/2} \|\alpha(s)\|_6^2 ds\Big)^{1/2} 
     \Big( \int_0^T s^{3/2}  \|A'(s)\|_6^2 ds\Big)^ {1/2} \notag \\
     &\le c t^{-1/2}     C_{22}(T, \rho_A(T))^{1/2} C_9(T, \rho_A(T))^{1/2}  
     \end{align}
     in view of \eref{ibA6} and \eref{iby9}.      This proves \eref{ibA15c}.

 \bigskip
\noindent
{\bf Proof of 7.}   
   The identities  $B' = d_A A' = dA' + [A\wedge A']$ show that
\begin{align}
dA' = B' - [A\wedge A'].       \label{ibA40b}
\end{align} 
  Let $h(s) = \|d(A(s) - A(t))\|_3$.    
 Then 
 \beq
 |h'(s)| \le \| dA'(s)\|_3 \le \|B'(s)\|_3 + c \|A(s)\|_6 \| A'(s)\|_6.
 \eeq
 Chose $\beta = a - (1/2)$ in \eref{H3} and observe that 
  \begin{align}
(1/2)  s^{2- \beta}|h'(s)|^2 
&\le  s^{5/2 -a}\|B'(s)\|_3^2 + s^{5/2 -a} c^2 \|A(s)\|_6^2 \| A'(s)\|_6^2 \notag \\
\le   s^{5/2 -a}&\|B'(s)\|_3^2  + c^2 s^{a-(1/2)} \(s^{1-a}\|A(s)\|_6^2 \) \(s^{2-a}\| A'(s)\|_6^2\) \notag\\
\le   s^{5/2 -a}&\|B'(s)\|_3^2  + c^2 t^{a-(1/2)} a_t^2  \(s^{2-a}\| A'(s)\|_6^2\).   \notag
  \end{align}
  Now
\begin{align}
  \int_0^t s^{(5/2) -a}  \|B'(s)\|_3^2 ds &\le \int_0^t s^{(2 - a)/2}  \|B'(s)\|_2 s^{(3-a)/2}\| B'(s)\|_6 ds \notag\\
  &\le \(\int_0^t s^{2-a} \|B'(s)\|_2^2 ds\)^{1/2}\( \int_0^t s^{3-a}   \| B'(s)\|_6^2 ds\)^{1/2}   \notag \\
  &\le C_{12}(t, \rho_A(t))^{1/2}  C_{13}(t, \rho_A(t))^{1/2} . 
                               \label{ibA82}
  \end{align}
  Therefore 
  \begin{align*}
  2\int_0^t &s^{(1/2) - a} \| dA(s) - dA(t)\|_3^2 
  \le  2 \frac{4}{((3/2) -a)^2}\int_0^t s^{(5/2) -a} |h'(s)|^2 ds \notag \\
  &\le \frac{16}{(3 -2a)^2} \int_0^t \Big\{s^{5/2 -a}\|B'(s)\|_3^2  
         + c^2 t^{a-(1/2)} a_t^2  \(s^{2-a}\| A'(s)\|_6^2\)\Big\} ds \notag \\
 &\le \frac{16}{(3 -2a)^2}  \Big\{C_{31}(t,\rho_A(t)  + c^2 t^{a-(1/2)} a_t^2 C_9(t, \rho_A(t))\Big\}
  \end{align*}
  wherein we have used \eref{iby9}. This proves \eref{ibA10a}.

 \bigskip
\noindent
{\bf Proof of 8.} 
 Since $dA(t) = B(t) -(1/2)[A(t)\wedge A(t)]$ it suffices to show that
 \begin{align}
 t^{(3-2a)/4} \( \|B(t)\|_3 + \| [ A(t)\wedge A(t)]\, \|_3\) \rightarrow 0, \ \text{as}\ \ t \downarrow 0. \label{ibA85}
 \end{align}
 By interpolation we have
 \begin{align*}
 \( t^{(3-2a)/4} \|B(t)\|_3\)^2 \le \Big(t^{(1-a)/2)} \|B(t)\|_2\Big) \Big(t^{(2-a)/2} \|B(t)\|_6\Big).
 \end{align*}
 The two factors on the right go  to zero as $t\downarrow 0$, 
 the first by \eref{iby1} and the second by \eref{iby3}. 
 Further,
 $t^{1-a} \|A(t)\wedge A(t)\|_3 \le c \Big(t^{(1-a)/2} \|A(t)\|_6\Big)^2 \rightarrow 0$ by \eref{ibA5a}.
 Since $(3-2a)/4 \ge (1-a)$ for $1/2 \le a <1$, the limit \eref{ibA85} holds. 
 Finally, the second limit in  \eref{ibA14}
 holds because $(1/2) \ge (3-2a)/4 $  for $1/2 \le a < 1$.

\bigskip
\noindent 
{\bf Proof of 9.} 
By \eref{iby15} there is a constant $k_t$ for each $t \in (0, T]$ such that 
 $ \|A'(s)\|_\infty \le s^{-3/2}k_t$ for $0 < s \le t$ and $k_t \rightarrow 0$ as $t \downarrow 0$.
 Hence, writing $\alpha(t) = A(t) - A(T)$,  we have
\begin{align}
 \|\alpha(t)\|_\infty &= \| \int_t^T A'(s) ds \|_\infty                     \notag \\
 &\le \int_t^T \|A'(s)\|_\infty ds                                \notag \\
 & \le  \int_t^T s^{-3/2}k_T ds         \notag \\
 & = t^{-1/2}  (1- (t/T)^{1/2}) 2 k_T           \notag                        
 \end{align}
 This proves \eref{ibA72}. Actually $t^{1/2} \|\alpha(t)\|_\infty \rightarrow 0$ as $t\downarrow 0$
 as one sees from  the inequalities 
 $t^{1/2} \|\alpha(t)\|_\infty \le t^{1/2}\int_t^r s^{-3/2} k_r ds + t^{1/2} \int_r^T s^{-3/2} k_T ds$.
 The $\limsup_{t\downarrow 0}$ on the left is then at most $2 k_r$, which is small for small $r >0$.

\bigskip
\noindent
{\bf  Proof of  10.} 
 We will apply the Hardy inequality \eref{H3}  with $T= t$  
  to the function $h(s) = \|d^*A(s) - d^*A(t)\|_2$. 
 We have $|h'(s)| \le \| d^*A'(s)\|_2 = \|\, [A(s)\lrc A'(s)]\, \|_2$ by \eref{ibA40}.   
 Hence  $s^{2-a} h'(s)^2 \le s^{2-a} \|\, [A(s)\lrc A'(s)]\, \|_2^2$. Therefore, by Hardy's inequality \eref{H3}
with $\beta = a$ we find
\begin{align}
\int_0^t s^{-a}\|d^*(A(s) - A(t))\|_2^2 ds 
&\le  \frac{4}{(1-a)^2} \int_0^t  s^{2-a}  \|\, [A(s)\lrc A'(s)]\, \|_2^2 ds\notag\\
&\le    \frac{4}{(1-a)^2} \int_0^t  s^{2-a} c \|A(s)\|_6^2\|A'(s)\|_3^2 ds \notag
\end{align}
The integral can be estimated by
\begin{align}
 \int_0^t  s^{2-a}  \|A(s)\|_6^2\|A'(s)\|_3^2 ds & \le \int_0^t \Big(s^{1-a} \|A(s)\|_6^2 \(s\|A'(s)\|_3^2\) ds \notag\\
 &\le a_t^2 \int_0^t  s\|A'(s)\|_3^2 ds \notag\\
 &\le a_t^2  C_{10}(t, \rho_A(t))                           \label{ibA42}
\end{align}
by \eref{ibA5a} and \eref{iby10}.  This proves  \eref{ibA10b}.

\bigskip
\noindent
 {\bf Proof of 11.}   
 Let $h(s) = \|d(A(s) - A(t))\|_2$. Then $|h'(s)| \le \| dA'(s)\|_2 = \| B'(s) - [A(s)\wedge A'(s)]\, \|_2$
by \eref{ibA40b}.
From Hardy's inequality \eref{H3} with $\beta = a$  and $T =t$ we find
\begin{align}
\int_0^t s^{-a} &\| d(A(s) - A(t)) \|_2^2 ds  \notag\\
&\le \frac{4}{(1-a)^2}\int_0^t s^{2-a}\| B'(s) - [A(s)\wedge A'(s)]\, \|_2^2 ds       \label{ibA43}
\end{align}
We see from \eref{iby12} that $\int_0^t s^{2-a} \| B'(s)\|_2^2 ds < \infty$.
Moreover $\|\, [A(s)\wedge A'(s)]\, \|_2 \le c \|A(s)\|_6 \| A'(s)\|_3$. So the bound \eref{ibA42} shows
that $\int_0^t s^{2-a} \|\, [A(s)\wedge A'(s)]\, \|_2^2  ds < \infty$ also.  Combining these bounds we find
\begin{align}
\int_0^t s^{-a} \| d(A(s) - A(t)) \|_2^2 ds  
      \le \frac{8}{(1-a)^2}\(C_{12}(t, \rho_A(t)) + c^2a_t^2  C_{10}(t, \rho_A(t))\).  \label{ibA45}
\end{align}
The assertion \eref{ibA10c} follows.

\bigskip
\noindent
{\bf Proof of 12.} 
 From \eref{ibA10c} with $t = T$ we have
\begin{align}
\int_0^T s^{-a}\|dA(s) - dA(T)\|_2^2 ds < \infty               \label{ibA46}
\end{align}
By assumption, $A(T) \in H_1$, which is defined in \eref{ymh5}.
Therefore $\|dA(T)\|_2 < \infty$.  Since $a < 1$ we have $\int_0^T s^{-a}\|dA(T)\|_2^2 ds < \infty$.
Hence   \eref{ibA46} implies \eref{ibA10c1}.

\bigskip
\noindent
{\bf Proof of 13.}   
Integrating  \eref{ibA40} over the interval $[s,t]$ we find
\begin{align}
d^* A(s) - d^* A(t) = \int_s^t [ A(\s)\lrc A'(\s)] d\s.     \label{ibA40c}
\end{align}
Hence
\begin{align*}
\| d^* A(s) &- d^*A(t)\|_2 \le   \int_s^t \|\,[ A(\s)\lrc A'(\s)]\, \|_2 d\s  \notag\\
&\le c \int_s^t \|A(\s)\|_6 \| A'(\s)\|_3 d\s.                     \notag
\end{align*}
But
\begin{align}
\(\int_s^t \|A(\s)\|_6 &\| A'(\s)\|_3 d\s \)^2 \notag\\
&=\ \Big\{ \int_s^t   \s^{(a-1)/2}  \(\s^{-a/2}  \|A(\s)\|_6\)   \( \s^{1/2} \| A'(\s)\|_3\) d\s\Big\}^2\notag \\
&\le s^{a-1}  \(\int_s^t \s^{-a} \|A(\s)\|_6^2 d\s\) \(\int_s^t \s\|A'(\s)\|_3^2 d\s\)\notag\\
&\le   s^{a-1}  \(\int_0^t \s^{-a} \|A(\s)\|_6^2 d\s\) \(\int_0^t \s\|A'(\s)\|_3^2 d\s\)\notag \\
&\le s^{a-1}\   \hat a_t\     C_{10}(t, \rho_A(t)  \label{ibA49}
\end{align}
by  \eref{ibA6a} and \eref{iby10}.
Therefore
\begin{align}
\sup_{0< s \le t} s^{1-a}\| d^* A(s)- d^*A(t)\|_2^2 
&\le       c^2 \hat a_t\     C_{10}(t, \rho_A(t)        \label{ibA50}       
\end{align}
 This goes to zero as $t\downarrow 0$. This proves  \eref{ibA10d}.

\bigskip
\noindent
{\bf Proof of 14.} 
 From the identity      \eref{ibA40b}        
 we find 
 \begin{align}
 s^{(1-a)/2} &\|d(A(s) - A(t))\|_2     \notag \\
 &\le s^{(1-a)/2} \|  B(s) - B(t)\|_2  +s^{(1-a)/2} \int_s^t \|\,[A(\s) \wedge A'(\s)]\, \|_2 d\s.  
                      \label{ibA51}
 \end{align}
 Since $ \|\,[A(\s) \wedge A'(\s)]\, \|_2 \le c \|A(\s)\|_6\|A'(\s)\|_3$, \eref{ibA49} shows
 that the second term in \eref{ibA51} is  bounded by $(c  \hat a_t C_{10})^{1/2}$ and therefore goes
 to zero uniformly in $s \le t$ as $t \downarrow 0$. The first term on the right in \eref{ibA51} is at most
 $s^{1-a} \| B(s)\|_2 + t^{1-a} \| B(t)\|_2$ which is bounded by $2 C_1(t, \rho_A(t))$ in
  accordance with   \eref{iby1}. This completes the proof of \eref{ibA10e}.

\bigskip
\noindent
{\bf Proof of 15-16.} 
Over $\R^3$ we have the identity
\begin{align}
\| \n \w\|_2^2 = \| d\w\|_2^2 + \| d^* \w \|_2.    \label{ibA60}
\end{align}
 Consequently \eref{ibA12} follows immediately from \eref{ibA10b} and \eref{ibA10c} while 
  \eref{ibA13a} follows immediately from  \eref{ibA10d} and \eref{ibA10e}.

\bigskip
\noindent
{\bf Proof of 17-18.} 
It follows from \eref{ibA12} with $t = T$ that if  $\|\n A(T)\|_2 < \infty$  then \eref{ibA13b} holds since $ a <1$.

Now take $t = T$ in \eref{ibA13a}. It follows that if $\|\n A(T)\|_2 < \infty$ then
$ \| \n A(r)\|_2 < \infty$ for $ 0< r \le T$. Hence
an argument similar to that in the proof of \eref{ibA5a} shows that \eref{ibA13} follows from \eref{ibA13a}.

This completes the proof of Theorem \ref{thmibA1}.
\end{proof}

\bigskip

\subsection{Estimates for the integral equation}  \label{secests}

Throughout this subsection we will assume that  $A(\cdot)$ is a strong solution to the Yang-Mills 
heat equation over $[0,\infty)$ with finite action. 

\begin{lemma}\label{lemests1}  
For each $T > 0$ there is a constant $\mu_T$, depending only on $T$ and $\rho_A(T)$,     
 such that
\begin{align}
t^{1/2} \|K(t) \w\|_2 &\le \mu_T \|\w\|_{H_1^\As}, \ \ \ 0 <t \le T      \label{eie5}\\
\ \ \  \text{and}\ \ \  \ \mu_T &\rightarrow 0\ \ \text{as}\ \ T \downarrow 0. \label{eie6}
\end{align}
Let $0 < \tau <T$. There is a constant $m_\tau$, depending only on $\tau, T$ and $\rho_A(T)$,  such that
\begin{align}
\|(K(t) - K(r))\w\|_2 \le (t-r)^{3/4} 
    m_\tau  \|\w\|_{H_1^\As},\ \ \ \tau\le r \le t \le T.      \label{eie7}
\end{align}

\end{lemma}
        \begin{proof}   
  We need to prove  bounds of the form \eref{eie5}  and \eref{eie7} for
  each of the four operators that  appear in \eref{ps13}. Taking the terms in the multiplication
   operator $M(t)$ first we have  
  \begin{align}
  t^{1/2}\|\, [ \alpha_j(t), [\alpha_j(t), \w]]\, \|_2 &\le c^2 t^{1/2}\|\alpha_j(t)\|_6^2 \|\w\|_6 \notag\\
  &\le c^2 C_{21}(T, \rho_A(T))    \|\w\|_6    \label{eie9}
  \end{align}
  by \eref{ibA5}.     Furthermore, in view of \eref{iby5} and \eref{ibA5} with $a =1/2$ we have
  \begin{align}
  \|\, [ \alpha_j(t), &[\alpha_j(t), \w]] -  [ \alpha_j(r), [\alpha_j(r), \w]] \, \|_2        \notag \\
   & \le c^2 \|\alpha_j(t) - \alpha_j(r)\|_6 (\|\alpha_j(t)\|_6+\|\alpha_j(r)\|_6)\|\w\|_6       \notag\\
  &\le 2c^2\Big( \|A_j(t) - A_j(r) \|_6 max_{\tau \le s \le T} \ \|\alpha_j(s)\|_6\Big) \| \w\|_6  \notag\\
  &\le 2c^2\Big(\int_r^t \|A'(s)\|_6 ds  \Big( \tau^{-1/2} C_{21}(T, \rho_A(T))\Big)^{1/2}\| \w\|_6 \notag\\
  &\le 2c^2\Big(\int_r^t s^{-5/4} C_5(T, \rho_A(T))^{1/2} 
                \Big( \tau^{-1/2} C_{21}(T, \rho_A(T))\Big)^{1/2}\| \w\|_6   \notag\\
  &\le 2c^2 \tau^{-5/4} (t-r)  C_5(T, \rho_A(T))^{1/2}
                  \Big( \tau^{-1/2} C_{21}(T, \rho_A(T))\Big)^{1/2}\| \w\|_6 \notag\\
  &= (t-r) \tau^{-3/2} C_{40}(T, \rho_A(T))  \|\w\|_6.
  \end{align}   
Concerning the second term in \eref{ps11} we have  
  \begin{align}
  t^{1/2}\|\, [div_\As \alpha(t), \w]\, \|_2 &\le c t^{1/2} \|div_\As \alpha(t)\|_3 \|\w\|_6  \notag\\
  &    \le cC_{24}(T, \rho_A(T)) \|\w\|_6    
  \end{align}
  by \eref{ibA15c}. Furthermore, using \eref{ibA69}, \eref{ibA5} and \eref{iby5}, we find
  \begin{align*}
  \|\, [div_\As (\alpha(t)&-\alpha(r)), \w]\, \|_2 =  \|\,\int_r^t [div_\As \alpha'(s), \w]ds \, \|_2\notag\\
  &\le c \int_r^t \|div_\As \alpha'(s)\|_3 ds\ \|\w\|_6               \notag\\
  &\le c^2 \int_r^t \|\alpha(s)\|_6 \|A'(s)\|_6ds    \ \|\w\|_6 \notag\\
  &\le c^2 \int_r^t s^{-1/4} s^{-5/4} ds \ (C_{21} C_5)^{1/2}\ \|\w\|_6 \notag\\
  &\le c^2 (t-r) \tau^{-3/2}   \ (C_{21} C_5)^{1/2}  \ \|\w\|_6.
 \end{align*}
   The third term in \eref{ps11} is easily estimated by
  \begin{align*}
  t^{1/2} \|\, [\w \lrc B(t)]\, \|_2 &\le c t^{1/2} \|B(t)\|_3 \|\w\|_6 \notag \\
  &\le  c  C_6(T, \rho_A(T))^{1/2}  \|\w\|_6  
  \end{align*}
  by \eref{iby6} with $a = 1/2$. Furthermore  
  \beq
  \|\, [\w \lrc (B(t)- B(r))]\, \|_2\le c\|B(t) - B(r)\|_3 \|\w\|_6. \notag
  \eeq
  We will show that there is a standard dominating function $C_{42}$ such that
  \begin{align}
  \|B(t) - B(r)\|_3 \le (t-r)^{3/4} \tau^{-5/4} C_{42}(T, \rho_A(T)),\ \ 0 <\tau \le r \le t \le T. \label{eie40}
  \end{align} 
  For the proof of \eref{eie40} apply H\"older's inequality to find
  \begin{align*}
  \|B(t) - B(r)\|_3 &\le \int_r^t \|B'(s)\|_3 ds \\
  &\le \Big( \int_r^t 1^{4/3} ds \Big)^{3/4} \(\int_r^t \|B'(s)\|_3^4ds \)^{1/4} \\
  & \le (t-r)^{3/4}   \(\int_r^t \|B'(s)\|_2^2 \|B'(s)\|_6^2 ds \)^{1/4} \\
  &\le (t-r)^{3/4}   \(C_{4} \int_r^t s^{-5/2}\|B'(s)\|_6^2 ds \)^{1/4} \\
  &\le (t-r)^{3/4}   \(C_{4}  \tau^{-5} \int_r^t s^{5/2}\|B'(s)\|_6^2 ds \)^{1/4} \\
  &\le (t-r)^{3/4}   \(C_{4}  \tau^{-5} C_{13}(T, \rho_A(T)) \)^{1/4}
  \end{align*}
  by  \eref{iby4} with $a = 1/2$ and $t =T$ and by \eref{iby13}.

  Thus the multiplication operator $M(t)$  satisfies
 \begin{align*}
 t^{1/2}\|M(t)\w\|_2   &\le q_T \|\w\|_6, \ \ 0<t \le T\ \ \ \text{and}\\
 \| (M(t) - M(r))\w \|_2 &\le (t-r)^{3/4} q_{\tau, T} \|\w\|_6,\ \ 0 < \tau \le r \le t \le T
 \end{align*}
 for some constants $q_T$ and $q_{\tau, T}$ which are 
 majorized by dominating functions of $T$ and $\rho_A(T)$ for each $\tau >0$.

  The differential operator term in \eref{ps13} can be dominated as follows.
  \begin{align*}
  t^{1/2}\|\, [\alpha(t)\cdot \n^\As \w]\, \|_2 &\le c t^{1/2}\|\alpha(t)\|_\infty \|\n^\As \w\|_2  \notag\\
  &\le  c C_{25}(T) \| \w\|_{H_1^\As}
  \end{align*}
  by \eref{ibA72}.  Furthermore, in view of \eref{iby15} we have
  \begin{align}
  \|\, [(\alpha(t)- \alpha(r))\cdot \n^\As \w]\, \|_2
  &\le c \|\alpha(t)- \alpha(r)\|_\infty \   \|\n^\As \w\|_2 \notag\\
  &\le c \int_r^t \|A'(s)\|_\infty ds \  \|\n^\As \w\|_2 \notag\\
  &\le c \int_r^t s^{-3/2} ds  \cdot C_{15}^{1/2}  \  \|\n^\As \w\|_2 \notag\\
  &\le c (t-r) \tau^{-3/2} \cdot  C_{15}^{1/2}   \  \|\n^\As \w\|_2. \notag
  \end{align}
  This completes the proof of the  lemma.   
     \end{proof}

     \begin{remark}{\rm 
     Three of the terms in the operator $K(t)$  have been shown above to be 
     H\"older continuous of order one and the fourth one of order $3/4$. But higher order 
     initial behavior estimates developed in \cite{CG3} show that $\|B'(t)\|_6$ is bounded 
     on $[\tau, T]$, from which it would follow that all four terms are H\"older continuous of order one.
     However we will not need this improvement in this paper.
     }
     \end{remark}

                \begin{lemma}\label{freeprop}   $($Free propagation$)$
          Let  $ 0 \le b <1$ and suppose that  $w_0 \in H_b$. 
Then, for some constants $c_b$ and $\gamma_b$ there holds
\begin{align}
e^{2t} c_b^2 \|w_0\|_{H_b^\As}^2 &\ge t^{1-b} 
\|e^{t\Delta_\As} w_0\|_{H_1^\As}^2 \rightarrow 0\ \text{as}\ \ t\downarrow 0\ \ \ \text{and} \label{ST449a} \\
\int_0^T t^{-b} \|e^{t\Delta_\As} w_0\|_{H_1^\As}^2 dt 
         &\le e^{2T} \gamma_b^2 \|w_0\|_{H_b^\As}^2.              \label{ST450a} 
\end{align}
In particular, the function $t\mapsto e^{t\Delta_\As} w_0$ lies in $\Q_T^{(b)}$.
\end{lemma}
            \begin{proof}
The proof of \eref{ST449a} and \eref{ST450a} relies only on the spectral
  theorem and  a computation that may be found in \cite[Lemma 3.4] {G70}. 
  The continuity of    $t\mapsto e^{t\Delta_\As} w_0$ on $[0,T]$ into $H_b^\As$ and on $(0,T]$ into $H_1^\As$
  is clear.  The condition \eref{vst350b} follows from  \eref{ST449a}.
\end{proof}

        \begin{remark}{\rm  
If $L$ is a non-negative self-adjoint operator
on a Hilbert space and $D = L^{1/2}$ then
\beq
\| D^\alpha e^{-tL}\| \le c_\alpha t^{-\alpha/2},\ \ \  t > 0, \ \ \ \alpha \ge0,     \label{hk0}
\eeq
for some constant $c_\alpha$, as follows from the spectral theorem  and the inequality  
$\sup_{\lambda>0} \lambda^{\alpha/2} e^{-t\lambda} =t^{-\alpha/2} \sup_{\sigma >0} \sigma^{\alpha/2}e^{-\sigma}$.
Here $\lambda \ge 0$ is a spectral parameter for $L$. The case of interest for us
 will be  $L = 1 -\Delta_\As$ acting on $L^2(M; \Lambda^1\otimes \kf)$.
}
\end{remark}

\begin{remark}\label{remcv}{\rm  
The following   identity, which arises frequently,  is listed here for convenience.
Let  $\mu$ and $\nu$ be real numbers with  $ \mu <1$ and  $\nu <1$. Then
\begin{align}
\frac{1}{t} \int_0^t(t-s)^{-\mu} s^{-\nu} ds 
                        &= t^{-\mu -\nu } C_{\mu, \nu}\                            \label{rec519}  
\end{align} 
for some finite constant $C_{\mu,\nu}$. This follows from the substitution $s = t\sigma$.
}
\end{remark}

\begin{lemma} \label{lemestb}
 Suppose that $ 0 < b <1$ and $w \in \Q_T^{(b)}$. Let
\beq
(Yw)(t)  = \int_0^t e^{(t-s) \Delta_\As} K(s) w(s) ds, \ \ 0 \le t \le T                          \label{vst500}
\eeq
If\,  $0\le r \le 1$ then 
\begin{align}
\| (Yw)(t) \|_{H_r^\As} \le c_r \int_0^t (t-s)^{-r/2} s^{-1/2} \| w(s)\|_{H_1^\As} ds\ \mu_t,  \label{vst501}
\end{align}
and
\begin{align}
\| (Yw)(t) \|_{H_r^\As} \le c_{r,b} t^{(b-r)/2} |w|_T\ \mu_T,\   0\le t \le T .                  \label{vst510}
\end{align}
\end{lemma}
       \begin{proof} By \eref{hk0}, \eref{eie5}, \eref{vst352b} and\eref{rec519}  we have
       \begin{align}
 \| (Yw)(t) \|_{H_r^\As} &=\int_0^t \|D^r e^{-(t-s)\Delta_\As} K(s)w(s)\|_2 ds         \notag\\
 &\le \int_0^t   \|D^r e^{-(t-s)\Delta_\As}\|_{2\rightarrow 2} \|K(s) w(s)\|_2 ds      \notag \\
 &\le \int_0^t  c_r (t-s)^{-r/2} s^{-1/2} \|w(s)\|_{H_1^\As}  ds\ \mu_t                \\
 &\le \int_0^t c_r  (t-s)^{-r/2} s^{-1/2} s^{(b-1)/2} ds \ |w|_t \ \mu_t                \notag\\
 &= \int_0^t c_r  (t-s)^{-r/2} s^{(b/2) -1}  ds \ |w|_t \ \mu_t                              \notag \\
 &= c_rC_{r/2, 1-(b/2)} t^{(b-r)/2} \  |w|_t\  \mu_t .
 \end{align}
 This proves both \eref{vst501} and \eref{vst510}. The condition $b >0$ is needed in the fourth line.
 \end{proof}

\subsection{Existence and uniqueness of mild solutions} \label{seceum}

\begin{proof}[Proof of Theorem \ref{thmmild1b}]
Define
\beq
(Zw)(t) = e^{t\Delta_\As}w_0 + (Yw)(t),     \ \ w \in \Q_T^{(b)} .  \label{vst580b}
\eeq
We will show that, for sufficiently small $T$,  $Z$ is a contraction mapping on a 
closed subset of the Banach space $\Q_T^{(b)}$ invariant under $Z$. 
Take $r =b$ in \eref{vst510} and then $r=1$ to find
\begin{align}
\|(Yw)(t)\|_{H_b^\As} &\le  c_{b,b} \mu_{T} |w|_T \\
\|(Yw)(t)\|_{H_1^\As} &\le  c_{1,b} \mu_{T} |w|_T t^{(b-1)/2}.
\end{align} 
Therefore
\begin{align}
\sup_{ 0 < t \le T} \| (Yw)(t)\|_{H_b^\As} &\le  c_{b,b} \mu_T |w|_T \ \ \text{and} \\
\sup_{ 0 < t \le T} t^{(1-b)/2} \| (Yw)(t)\|_{H_1^\As} &\le c_{1,b} \mu_T |w|_T.
\end{align}
Hence, in view of the definition \eref{vst353b}, we find
\begin{align}
\|Yw\|_{\Q_T^{(b)}} \le c_5 \mu_T |w|_T,
\end{align}
where $c_5 = c_{b,b} + c_{1,b}$.
We may choose $T>0$ so small that $c_5 \mu_T \le 1/2$.  Then
\begin{align}
\|Yw\|_{\Q_T^{(b)}} &\le (1/2) |w|_T     
\le (1/2)  \|w\|_{\Q_T^{(b)}} .    \label{vst585}
\end{align}
$Yw$ is  easily seen to  have the appropriate continuity properties to lie in $\Q_T^{(b)}$.
 $Y$ is  therefore a contraction in the Banach space $\Q_T^{(b)}$ with contraction constant $1/2$.

Concerning the freely propagated term in \eref{vst580b}, Lemma \ref{freeprop} 
 shows that it lies in $\Q_T^{(b)}$. Moreover the inequality 
 $\|e^{t\Delta_\As} w_0\|_{H_b^\As} \le \| w_0\|_{H_b^\As}$, together with 
\eref{ST449a}, shows that  $t^{(1-b)/2} \| e^{t\Delta_\As} w_0\|_{H_1^\As} \le e^T c_b \|w_0\|_{H_b^\As}$
for $0 < t \le T$. Hence $\|e^{(\cdot)\Delta_\As} w_0\|_{ \Q_T^{(b)}} \le(1 + e^T c_b) \|w_0\|_{H_b^\As}$.
Let $c_b'(T) = 1 + e^T c_b$. Then the operator $Z$ has the bounds 
\begin{align}
\|Zw\|_{ \Q_T^{(b)}}  &\le c_b'(T) \|w_0\|_{H_b^\As} +  (1/2) \|w\|_{ \Q_T^{(b)}} ,\ \ \  \text{and} \label{vst586}\\
\|Zw_1 - Zw_2\|_{ \Q_T^{(b)}}  &\le    (1/2) \|w_1-w_2\|_{ \Q_T^{(b)}}
\end{align}
when  $w_1(0) = w_2(0) = w_0$ because the freely propagated terms in $Zw_j$ are the same
and therefore cancel.

Let $W = \{w \in \Q_T^{(b)}: w(0) = w_0\}$. 
This is a closed subset of $ \Q_T^{(b)}$ because  of the presence of the second term
 in the norm, defined in  \eref{vst353b}. It is invariant under $Z$ because 
 $(Zw)(0) = w(0) = w_0.$
Thus $Z$ is a contraction on $W$ and therefore has a unique fixed point in $W$. 
By \eref{vst586} a fixed point under $Z$ satisfies
 $\| w\|_{Q_T^{(b)}} = \| Z w\|_{Q_T^{(b)}} \le c_b'(T) \| w_0 \|_{H_b^\As} + (1/2) \|w\|_{Q_T^{(b)}}$,
 from which  \eref{vst360} follows by subtraction.
 This completes the proof of  Theorem \ref{thmmild1b}.  
 \end{proof}

\subsection{Mild solutions are strong solutions} \label{secmildstr} 
  
 \begin{remark}\label{strat6} {\rm (Strategy)    
  Typically, a solution to the integral equation \eref{vst31}
will be a strong solution if the integrand $K(s) w(s)$ is H\"older continuous as a function
of $s$ into a suitable Banach space of functions on $M$. 
In our circumstances the coefficient operator $K(s)$ has a singularity at $s=0$.
For $s>0$ it is more regular but by no means smooth. It will be necessary to use
detailed information developed in Section \ref{secibA},
 concerning the behavior of the
 connection form $A(s)$ near and away from $s =0$.
 We will show first that any mild solution
 $w$ is H\"older continuous away from $t=0$ as a function into $H_1^\As$. 
   The main theorem of this section asserts that, for any mild solution $w$, 
  the function $K(s)w(s)$ is a H\"older continuous
 function into $L^2(M)$    
  on any finite interval $[\tau,T]$ when $\tau >0$.
  We will then  use this
   to show that the  solution to \eref{vst31} 
is actually a strong solution for $t> 0$. We will prove this for $w_0 \in H_b(M)$ whenever $0 < b <1$.
}
\end{remark}

\begin{theorem}\label{thmstrong}  
Suppose that $A(\cdot)$ is a strong solution to the Yang-Mills heat
 equation over $[0, \infty)$  of finite  action,
 Let $w$ be a solution to the integral equation
 \eref{vst31} lying in
  $\Q_T^{(b)}$ for some $b \in (0,1)$. Then $w$ is a strong solution  to \eref{vst20} over
 $(0,T]$.   
 
  \end{theorem}
 The proof depends on the following lemmas.

\begin{lemma}\label{lemST8b} 
 Let $ 0 < \alpha < 1$ and $0 < T <\infty$. Let $-L$ be a non negative self-adjoint operator on a Hilbert space
 and let $D = (1-L)^{1/2}$.
 There is a constant $e_{T, \alpha}$ such that      
  \begin{align}
     \| D(e^{\epsilon L} -1) e^{\delta L} \|
 &\le \epsilon^\alpha \delta ^{-\frac{1}{2} - \alpha} e_{T,\alpha}    
      \ \    \forall \epsilon > 0\ \  \text{and}\ \ \forall \delta \in (0, T].            \label{vst118}   
           \end{align}
\end{lemma}
\begin{proof}With the help of the operator inequality
\begin{align}
\| D(e^{\epsilon L} -1) e^{\delta L} \| \le \|D^{-2\alpha}(e^{\epsilon L} -1)\|\, \| D^{1+2\alpha} e^{\delta L}\| \notag 
\end{align}
it suffices to make estimates of the two norms. By the spectral theorem, $\|D^{-2\alpha}(e^{\epsilon L} -1)\|$
is at most the supremum over $x \in [0, \infty)$ of 
$ (1+x)^{-\alpha} (1 - e^{-\epsilon x}) = (1 +\epsilon^{-1} y)^{-\alpha}(1 - e^{-y})
=\epsilon^{\alpha} (\epsilon + y)^{-\alpha} (1- e^{-y}) \le \epsilon^\alpha  y^{-\alpha} (1- e^{-y})
  \le \epsilon^{\alpha} c_\alpha$,  
  wherein we have put   $y = \epsilon x$. The second norm, writing $c = (1/2) +\alpha$, 
   is at most the supremum over $ x \in [0,\infty)$   of  
  $(1+x)^c e^{-\delta x} = (1+ \delta^{-1} y)^c e^{-y} 
   = \delta^{-c} (\delta + y)^c e^{-y}\le  \delta^{-c} (T + y)^c e^{-y}\le \delta^{-c}  \hat e_{T,\alpha}$
    for $0 < \delta \le T$,
    wherein we have put $y = \delta x$.
\end{proof}

\begin{lemma} \label{lemH}$($H\"older continuity of $\rho$$)$  
Let $0 < \tau <T < \infty$. Suppose that  $0 < b <1$ and that $w \in \Q_T^{(b)}$.       
Let $0 < \alpha < 1/2$.  
Define
\beq
\rho(t) =(Yw)(t) = \int_0^t e^{(t-s)\Delta_\As} K(s) w(s) ds.              \label{vst65}
\eeq
There is a constant  $c_5$ depending only on $\alpha, \tau$  and $T$ such that
\beq
\|\rho(t) - \rho(r)\|_{H_1^\As} \le c_5 (t-r)^\alpha  \mu_T  |w|_T, \ \ \text{for}\ \ \tau \le r <t <T,  \label{vst66}
\eeq
where $\mu_T$ is defined in Lemma \ref{lemests1}.
\end{lemma}
   \begin{proof}
 Choosing $r$ and $t$ as in \eref{vst66},  
 we may write 
 \begin{align*}
   \rho(t) - \rho(r) &= \int_0^r \Big( e^{(t-\sigma)\Delta_\As} - e^{(r-\sigma)\Delta_\As} \Big)
     K(\sigma)w(\sigma)d\sigma  \\
      &   + \int_r^t e^{(t-\sigma)\Delta_\As} K(\sigma)w(\sigma)d\sigma.
   \end{align*}
 Therefore 
   \begin{align}
  \| &\rho(t) - \rho(r)\|_{H_1^\As}                                                                 
 \le \int_0^r \| \Big( e^{(t-\sigma)\Delta_\As} - e^{(r-\sigma)\Delta_\As}\Big) 
                                       K(\sigma)w(\sigma)\|_{H_1^\As}d\sigma   \notag \\
                &\ \ \ \ \ \ \qquad  \qquad  \qquad \qquad 
      + \int_r^t\| e^{(t-\sigma)\Delta_\As} K(\sigma)w(\sigma)\|_{H_1^\As}d\sigma                      \notag\\
& \le  \int_0^r \Big\| \Big( e^{(t-\sigma)\Delta_\As} - e^{(r-\sigma)\Delta_\As}\Big)\Big\|_{2\rightarrow H_1^\As}
                     \|K(\sigma) w(\sigma)\|_{2}  d\sigma                   \notag\\
&\ \ \ \ \ \ \qquad  \qquad  \qquad \qquad  
    +  \int_r^t \Big\|e^{(t-\sigma)\Delta_\As}\Big\|_{2\rightarrow H_1^\As}
                                                             \|K(\sigma) w(\sigma)\|_{2} d\sigma.   \notag\\
&\le  (t-r)^\alpha \int_0^r  (r-\sigma)^{-\frac{1}{2} - \alpha}\  
                                           \|K(\sigma) w(\sigma)\|_{2} d\sigma\   \cdot e_{T, \alpha} \notag\\
&\qquad\qquad  +    \int_r^t (t-\sigma)^{-1/2}   
                               \|K(\sigma) w(\sigma)\|_{2} d\sigma  \   \cdot   c_1 .                   \label{vst70}
  \end{align}
  The bound in the first line in \eref{vst70} comes from \eref{vst118} with 
  $\delta = r-\sigma$ and $\epsilon = t-r$, while the bound
   in the second line      comes from      the spectral theory bound 
   \eref{hk0}.      

 By \eref{eie5} and \eref{vst352b}  we have
\beq
\|K(s) w(s)\|_2 \le   s^{-1/2} \mu_T \|w(s)\|_{H_1^\As} \le \mu_T s^{-1/2} s^{(b-1)/2} |w|_T.   
 \label{vst76}
\eeq
Insert the bound \eref{vst76} into \eref{vst70} to find
\begin{align}
  \| \rho(t) - \rho(r)\|_{H_1^\As} & \le  \Big\{(t-r)^{\alpha}  \int_0^r 
               (r-\sigma)^{- \frac{1}{2} -\alpha}    \sigma^{(b/2) -1}
               d\sigma\ \cdot  e_{T, \alpha}                                      \label{vst77}\\
&\ \ \ \ \ +\int_r^t    (t-\sigma)^{-1/2}\sigma^{(b/2) -1} 
d\sigma \cdot c_1\Big\}  \mu_T |w|_T.                                     \label{vst78}
         \end{align}
          The  integral in line \eref{vst77}, which  is finite because $\alpha < 1/2$ and $b >0$,  is at most 
          $r^{\frac{b-1}{2} - \alpha}\cdot Const. \le \tau^{\frac{b-1}{2} - \alpha}\cdot Const.$  by \eref{rec519}. 
The integral in line \eref{vst78} is at most
\begin{align}
\tau^{(b/2) -1} \int_r^t (t-\sigma)^{-1/2} d\sigma = 2\tau^{(b/2) -1} (t-r)^{1/2}.
\end{align}
Since $(t-r)^{1/2} \le (t-r)^\alpha \cdot constant$  
  on $[\tau, T]$ 
the assertion \eref{vst66} follows.
\end{proof}

\begin{lemma}\label{lemholdw} $($H\"older continuity of $w(\cdot)$$)$ Suppose that $w$ is
 a solution to the integral equation \eref{vst31} lying in $\Q_T^{(b)}$ for some $b \in (0, 1)$. Let
$0<\alpha <1/2$   and let $0< \tau < T < \infty$. Then  there is a constant $c_6$, depending only
 on $\alpha, \tau, T, A, w(0)$     
   and $| w|_T$  such that
\beq
\| w(t) - w(r)\|_{H_1^\As} \le c_6 (t-r)^{\alpha}\ \ \ \text{for}\ \ \  \tau \le r \le t \le T     \label{vst79}
\eeq
\end{lemma}
         \begin{proof} For any function $w_0 \in H_b^\As$
         the function $t\mapsto e^{t\Delta_\As} w_0$ is differentiable on
          the interval $[\tau,\infty)$ into $H_1^\As$
         and therefore locally H\"older of order $\alpha$. 
         Since the second term on the right in \eref{vst31} 
         has been shown in Lemma \ref{lemH}  to be H\"older continuous of order $\alpha$ 
         on the interval $[\tau,T]$ the lemma follows.      
\end{proof}

\begin{lemma}\label{lemholdKw} $($H\"older continuity of $K(\cdot) w(\cdot)$$)$
 Let $w(\cdot)$ be a solution
of the integral equation \eref{vst31} lying in $\Q_T^{(b)}$ for some $b \in (0, 1)$. 
Let $\tau >0$ and let $0 < \alpha < 1/2$.   
$K(s)w(s)$ is H\"older continuous on $[\tau, T]$ of order $\alpha$ as a function into $L^2$ 
\end{lemma}
     \begin{proof} If $\tau \le r < t \le T$ then, in view of \eref{eie7}  
     \eref{eie5}  and \eref{vst79}
      we have
      \begin{align*}
  \|K(t)w(t) &- K(r)w(r)\|_{2} \\
    & \le \| (K(t) - K(r))w(t)\|_{2} + \|K(r)(w(t) - w(r))\|_{2} \\
    &\le (t-r)^{3/4} \|w(t)\|_{H_1^\As}\ m_\tau  +  \mu_T \tau^{-1/2} 
     \|w(t) -w(r)\|_{H_1^\As}\\
    & \le(t-r)^{3/4} \tau^{-1/2} |w|_T\ m_\tau   +\mu_T \tau^{-1/2}(t-r)^\alpha  c_6.               
        \end{align*}
        wherein we have used $t^{(b-1)/2} \le \tau^{-1/2}$ for $t \ge \tau$ in the last line.   
  \end{proof}

\bigskip
\noindent  
  \begin{proof} [Proof of Theorems \ref{thmstrong} and \ref{thmwe}] 
  We can apply Theorem 11.44 in \cite{RR} over the interval $[\tau,T]$ for
   any $\tau >0$   because we now know, in view of Lemma \ref{lemholdKw},
    that the forcing function $K(s)w(s)$ is H\"older continuous on this interval
     as a function into $L^{2}(M;\L^1\otimes \kf)$. 
     The strong time derivative $w'(s)$ of the function $[\tau,T] \ni s \mapsto w(s) \in L^2(M)$ therefore exists,
     $w(s)$ is in the domain of $\Delta_\As$, and both $w'(s)$ and $\Delta_\As w(s)$ are H\"older continuous
     of order $\alpha$ on $[\tau, T]$ into $L^2(M;\L^1\otimes \kf)$. 
  Moreover the equation \eref{vst20}    holds for each   $t \in [\tau,T]$. 
       This proves Theorem \ref{thmstrong}.
       
       For the uniqueness of strong solutions asserted in Theorem \ref{thmwe} observe that the hypotheses of  Theorem \ref{thmwe} imply 
       that $w(\cdot)$ lies in $\Q_T^{(b)}$ for any $T>0$. We can now apply
       an argument similar to that used in the proof of \cite[Theorem 3.30]{G70} to conclude
       that $w(\cdot)$ satisfies the integral equation \eref{vst31}.
       Uniqueness now follows from Theorem \ref{thmmild1b}.
       
       To extend $w(\cdot)$ to a solution over all of $[0, \infty)$ observe that
       \begin{align}
      \rho_A(t_0, t_0 +t) := (1/2) \int_{t_0}^{t_0 + t} (s-t_0)^{-1/2} \| B(s)\|_2^2 ds \le \rho_A(t)
       \end{align}
        for all $t_0 \ge 0$ and all $t \ge 0$ because $\|B(s)\|_2^2$ is nonincreasing. Therefore if one starts 
        the existence theorem at some time $t_0 \ge 0$ then the short time $T$ needed to make
        $c_5 \mu_T \le 1/2$ in the proof of contractivity of $Y$ in Section \ref{seceum} can be chosen
        independently of $t_0$ because $\mu_T$ depends monotonically only on $T$ and $\rho_A(t_0, t_0+T)$.
        Having proven existence up to  time $t_0 \ge T$ one can therefore continue the solution 
        up to time $t_0 +(T/2)$ by applying the short time existence theorem to $w(t_0 -(T/2))$. 
        This completes the proof of Theorem \ref{thmwe}, Parts a) and b).
  \end{proof}

\section{Finite $b$-action for the augmented equation}   \label{secboa}

\bigskip
For $0 \le b <1$ either of the following two conditions gives a measure of the 
 singular behavior of $w(t)$ near $t =0$. 
\begin{align}
\|w(t)\|_{H_1^\As}^2 &= o(t^{b-1})\ \ \ \ \text{as} \ \ t\downarrow 0.  \label{ce390}\\
\int_0^T t^{-b} \|w(t)\|_{H_1^\As}^2 dt &< \infty \ \ \text{for some}\ \ T \in (0, \infty).             \label{ce391}
\end{align}
Both conditions  in \eref{ce390} and \eref{ce391} are gauge invariant. 
Neither one implies the other. 
 The existence theorem, Theorem \ref{thmmild1b}, shows that the solution to the augmented 
 variational equation  \eref{av1} satisfies \eref{ce390} if $b \in (0,1)$ 
 and $w_0 \in H_b^\As$.
In this section we will prove that  if $b \in [1/2, 1)$  then the solution also satisfies \eref{ce391}.

\begin{theorem}\label{thmact1} 
 Assume that $A(\cdot)$ is a strong solution of the Yang-Mills heat equation over $[0,\infty)$ of finite action.
Suppose that $1/2 \le b <1$, $0< T <\infty$ and that $w_0 \in H_b^\As$, where $\As = A(T)$ as in
    Section \ref{secEUaug}.
If $w(\cdot)$ 
is  the mild solution of the augmented variational equation \eref{av1} with initial value $w_0$ 
and lying in $\Q_T^{(b)}$ then  for sufficiently small $T$ there holds
\beq
\ \ \ \qquad \qquad \int_0^T s^{-b} \| w(s)\|_{H_1^\As}^2 ds  \le \gamma_T \|w_0\|_{H_b^\As}^2< \infty
                                        \qquad \qquad\qquad       \label{ce401}
\eeq
for some constant $\gamma_T$ depending only on $T$ and $\rho_A(T)$.    
\end{theorem}

We are going to use the following abstract action bound lemma from \cite{G70}.

\begin{lemma} \label{thmactint} Let $L$ be a non-negative
  self-adjoint operator on a Hilbert space $H$. Suppose that $\alpha, \mu, b$ are real numbers
  such that
  \begin{align}
 &0 \le \alpha \le 1,                                                      \label{ce300}\\          
& 0\le  \mu \le b < 1.                                                    \label{ce302}               \\ 
 &\delta \equiv 1-\alpha - \mu \ge0,                              \label{ce301}
    \end{align}
  Then there is a constant $C_{\alpha,\mu}$, depending only on $\alpha$ and $\mu$,
   such that for any measurable function $g:(0,T)\rightarrow H$ there holds
  \begin{align}
 \int_0^T t^{-b} \Big\|\int_0^t s^{-\mu} L^\alpha e^{(t-s)L} g(s) ds\Big\|^2 dt 
 \le T^{2\delta}\int_0^T s^{-b} \| g(s)\|^2 ds \cdot  C_{\alpha,\mu} \label{ce303}
 \end{align} 
 when the right side is finite.
 Moreover $C_{\alpha, 0} \le 1$.
  \end{lemma}  
       \begin{proof} See \cite[Theorem 3.19]{G70} . 
  \end{proof}

\bigskip
\noindent 
\begin{proof}[Proof of Theorem \ref{thmact1}]
Define
\beq
g(s) = s^{1/2}  \Big( K(s)w(s)\Big),  0 < s \le T.
\eeq
Then
\begin{align}
\|g(s)\|_2 \le \|w(s)\|_{H_1^\As}\ \mu_T
\end{align}
by \eref{eie5}.
 Let us write $L = 1-\Delta_\As$ and $D = L^{1/2}$.
With $Yw$ defined as in \eref{vst500} we have 
\begin{align}
\|(Yw)(t)\|_{H_1^\As} &= \Big\|\int_0^t e^{(t-s)\Delta_\As} K(s)w(s) ds\Big\|_{H_1^\As}          \notag \\
&= \Big\|\int_0^t e^{(t-s)\Delta_\As} s^{-1/2} g(s)ds\Big\|_{H_1^\As}                \notag\\
&=\Big\|\int_0^t D e^{(t-s)(1-L)} s^{-1/2} g(s)ds \Big\|_2                   \notag\\
&= \Big\|e^t \int_0^t s^{-1/2}L^{1/2} e^{-(t-s)L} e^{-s}g(s)ds \Big\|_2. 
\end{align}   
We can apply 
Lemma \ref{thmactint} 
 with $H = L^2(M; \L^1\otimes \kf)$, $\alpha = 1/2$ and $\mu = 1/2$. 
 Then  \eref{ce303} holds with $\delta = 0$. Hence
\begin{align}
\int_0^T t^{-b} \|(Yw)(t)\|_{H_1^\As}^2 dt 
&=\int_0^T t^{-b} e^{2t} \Big\|\int_0^t s^{-1/2}L^{1/2} e^{-(t-s)L} e^{-s}g(s)ds \Big\|_2^2 dt  \notag\\
 &\le  e^{2T}\int_0^T s^{-b} \|e^{-s} g(s)\|_2^2 ds\cdot C_{\alpha, \mu}     \notag\\
 &\le e^{2T}\mu_T^2   \int_0^T s^{-b} \|w(s)\|_{H_1^\As}^2 ds\  C_{\alpha, \mu}      \label{ce412}
\end{align}
for $ 1/2  \le b < 1$.  
 Choose $T$ sufficiently
small so that  $e^{2T}\mu_T^2 C_{\alpha,\mu} \le 1/4$ and 
also $c_5 \mu_T \le 1/2$, as in Section \ref{seceum}. Then
\begin{align}
\Big(\int_0^T t^{-b} \| (Yw)(t)\|_{H_1^\As}^2 dt\Big)^{1/2}
\le \frac{1}{2} \Big(\int_0^T s^{-b} \|w(s)\|_{H_1^\As}^2 ds\Big)^{1/2}   \label{ce413}
\end{align}
and also \eref{vst585} holds.

We can now adapt the final step of the proof of Theorem \ref{thmmild1b} for
 our present purpose by simply
changing the norm on the space $\Q_T^{(b)}$ defined in  Notation \ref{notvst7b} thus: 
Replace the norm \eref{vst353b} by 
\beq
\|w\|_{\hat\Q_T^{(b)}} =  \Big(\int_0^T s^{-b} \|w(s)\|_{H_1^\As}^2 ds\Big)^{1/2} + \|w\|_{\Q_T^{(b)}}
\eeq
and strengthen  the condition \eref{vst350b} by requiring also $\|w\|_{\hat\Q_T^{(b)}}<\infty$.
 Then, by \eref{ce413} and \eref{vst585},  $Y$ is a contraction
in the resulting Banach space, $\hat\Q_T^{(b)}$ with contraction constant $1/2$.  
 Moreover, if $w_0 \in H_b^\As$ then
 the freely propagated term $e^{t\Delta_\As} w_0$ in the integral equation \eref{vst31}
  lies in this space by  \eref{ST450a}. Hence the integral equation \eref{vst31} has a unique
  solution in $\hat \Q_T^{(b)}$.  Since $\hat\Q_T^{(b)} \subset \Q_T^{(b)}$ the unique solution
  in  $\hat \Q_T^{(b)}$ is the same as the unique solution in  $ \Q_T^{(b)}$. The inequality
  \eref{ce401} now follows in the same way as \eref{vst360}.
   \end{proof}

\section{Initial behavior of solutions to the augmented variational equation} \label{secibw}

\subsection{Pointwise and integral identities} \label{secids}

 For a solution $w(\cdot)$ to the augmented variational equation \eref{av1}
on some interval there are two quantities  whose behavior near $t =0$ will largely determine
 the short time behavior of the solution to the variational equation \eref{ve} itself. Define 
\begin{align}
\psi(s) &= d_{A(s)}^* w(s)\ \ \ \text{and}                     \label{ib10psi} \\
\zeta(s) &= d_A^*d_A w(s) + [w(s)\lrc B(s)] \label{ib549}
\end{align}
$\psi(s)$ measures the deviation of $w(s)$ from horizontal at $A(s)$. The augmented variational equation
may be written
\begin{align}
-w'(s) = \zeta(s) + d_A\psi(s)    \label{ib10av}.
\end{align}

 \begin{lemma}\label{lempi2} $($Pointwise identities$)$  
  If $w$ is a solution to the augmented variational  equation  \eref{av1} on some interval then
  \begin{align}
  \text{$($Order$)$}\qquad \qquad \qquad  &       \notag\\
 (1)\qquad\ \ \ \ \ \ \  \, d_A^* \zeta(s) &= [ w\lrc A']                                                         \label{pi11a} \\ 
 (1) \qquad \ \ \,   (d/ds) \psi(s) &= - d_A^*d_A \psi + 2 [A'\lrc w].   \label{pi12a} \\                                                                                                       
 (2) \qquad\ \ \ \  \  -w''(s) &= \Big(d_A^* d_A + d_A d_A^*\Big) w'   \label{pi13}\\
+\Big\{ d_A^*[A' \wedge w&] + [ A' \lrc d_A w]  + d_A [A' \lrc w] + [A', d_A^* w] \Big\} 
                                              + [w\lrc B]'  .   \notag
\end{align}
\end{lemma}
 \begin{proof} 
 Using the identity 
 $d_A^*[\w \lrc B] = [ d_A \w  \lrc B] - [\w \lrc d_A^* B] 
 =[ d_A \w  \lrc B] + [\w \lrc A']$,
we may apply $d_A^*$ to $\zeta$  to find 
$d_A^* \zeta = [B\lrc d_A w] + [ d_A w  \lrc B] + [w \lrc A'] = [w \lrc A']$, which is \eref{pi11a}.

For the proof of \eref{pi12a} differentiate the definition of $\psi$ to find
\begin{align*}
(d/ds) \psi(s) &= (d/ds) (d_{A(s)}^* w(s)) \\
 &= d_A^* w' + [A'\lrc w] \\
&=-d_A^*\(\zeta (s) + d_A \psi\)  + [A'\lrc w]\\
&= -[w\lrc A'] - d_A^* d_A \psi  + [A'\lrc w],
\end{align*}
which proves \eref{pi12a}.

Differentiate \eref{av1} 
 with respect to $s$  to find \eref{pi13}.
  \end{proof}

 \begin{lemma}\label{lemintid2} $($Integral identities$)$ 
 Denote by $L_A$ the gauge covariant Hodge Laplacian  given by
 \beq
-L_A  = d_A^*d_A +d_A d_A^*,                 \label{intid40}
\eeq
$($not to be confused with the Bochner Laplacian given by \eref{ps9}.$)$ 
  If $w$ is a strong solution to the
  augmented variational  equation \eref{av1} on some interval then     
  \begin{align}   
(Order)& \notag\\    
 (0)  \ \ \ & \frac{d}{ds} \|  w(s) \|_2^2 
                    +2\Big\{ \| d_A  w(s) \| _2^2   + \| d_A^*  w(s) \| _2^2\Big\} \notag \\
         &     \ \ \ \ \ \ \ \ \                       = -2(B(s), [w(s)\wedge w(s)] ),   \label{intid35} \\
 (1)  \ \ \       &\frac{d}{ds}\Big\{ \| d_{A(s)} w(s) \|_2^2 + \| d_{A(s)}^* w(s) \|_2^2\Big\}
  + \| w'(s)\|_2^2 + \| L_{A(s)} 
            w(s)\|_2^2                                                     \notag\\
        &\ \ \ \ \ \ \ \ \ = 2\Big\{ ([A'\wedge w], d_A w) + ( [A'\lrc w], d_A^* w)\Big\} 
                                                                                + \|w\lrc B\|_2^2                  \label{intid36} \\ 
(2)  \ \ \   &\frac{d}{ds} \| w'(s) \|_2^2 
+2\Big\{ \|d_{A(s)} w'(s) \|_2^2  + \|d_{A(s)}^* w'(s) \|_2^2  \Big\}                 \notag\\
&\ \ \ \ \ \ \ \ \ = -2\Big\{ ([A'\wedge w],d_A w') + ([A'\lrc w],d_A^* w')     \notag \\
&\ \ \ \ \ \ \ \ \ \ \ \ \ \ \ +\Big([A'\lrc d_A w] +[A', d_A^* w], w'\Big) + ([w\lrc B]', w')\Big\}. \label{intid38}
       \end{align}    
\end{lemma}     
                    \begin{proof} 
  From \eref{av1} we find
  \begin{align*}
  (1/2) (d/ds) \| w(s) \|_2^2 
  & = ( (d/ds) w(s),  w(s))\\
  & = (  -d_A^* d_A w - d_Ad_A^*w - [w\lrc B],  w)\\
   & = - \| d_A w\|_2^2 - \| d_A^* w\|^2  -( B, [w\wedge w]),
   \end{align*}
   which is \eref{intid35}.

   For ease in reading define $g(s) = [w(s)\lrc B(s)]$. Then we may write
                    \eref{av1} as   $w' =L_A w -g$.   
   For the proof of \eref{intid36} observe first that 
 \begin{align}
   (w', L_A w)   &= (w', w' + g) = \| w'\|_2^2 + (w', g),\ \ \ \  \text{while also}  \label{intid20}\\   
   (w', L_A w) &= ( L_A w - g, L_A w) \notag\\
    &= \| L_A w\|_2^2 - (g, L_A w)\notag\\
    &= \| L_A w\|_2^2 - (g, w' +g)\notag \\
    &=  \| L_A w\|_2^2 - \| g\|_2^2 - (g,w').   \label{intid21}
    \end{align}
 Adding \eref{intid20} to \eref{intid21} gives
 \beq
 2(w', L_Aw) =   \| w'\|_2^2 +  \| L_A w\|_2^2   - \| g\|_2^2.     \label{intid22}
 \eeq     
    Hence
 \begin{align}
(1/2)& \frac{d}{ds}\Big\{ \| d_{A(s)} w(s) \|_2^2 + \| d_{A(s)}^* w(s) \|_2^2\Big\}  \notag\\
&= \Big\{\Big(\frac{d}{ds} (d_A w), d_Aw\Big)+ \Big(\frac{d}{ds}(d_A^* w), d_A^* w\Big)\Big\} \notag\\
&=\{([A' \wedge w], d_A w) + ([A'\lrc w], d_A^*w)\} +(w' ,d_A^* d_Aw)  + (w', d_A d_A^* w)\notag\\
&=\{([A' \wedge w], d_A w) + ([A'\lrc w], d_A^*w)\} - (w', L_A w).     \label{intid24}
\end{align}    
Replace the last term in \eref{intid24} by \eref{intid22} to find \eref{intid36}.
    
To prove the second order identity \eref{intid38}  use \eref{pi13}  to see that
\begin{align*}
(1/2)(d/ds) \| w'(s)\|_2^2 &= (w'', w') \\
= -\Big( (d_A^* d_A + &d_Ad_A^*) w', w'\Big) -(g',w')  \\
             - \Big( \Big\{d_A^* [A' \wedge &w] + [A' \lrc d_Aw] 
                         + d_A [A' \lrc w] + [A', d_A^* w] \Big\}, w'\Big).
\end{align*}
Hence
\begin{align*}
(1/2)&(d/ds) \| w'(s)\|_2^2 + \| d_A w'\|_2^2 + \| d_A^* w'\|_2^2 \\
&=-\Big\{ ( [A' \wedge w], d_A w') + ([A'\lrc w], d_A^* w')  \\
 &\ \ \ \ \ \ \         +\Big( [A' \lrc d_Aw] + [A', d_A^* w], w'\Big)\Big\} - (g',w'),
 \end{align*}
 which is \eref{intid38}.
\end{proof}

\bigskip
We will use these identities to derive differential inequalities and then, from these, derive 
 information about initial behavior  of solutions with the help of the following lemma.
 
 \begin{lemma}\label{lem50}
  Suppose that $f, g, h$ are nonnegative continuous  
  functions on $(0, t]$ and that $f$ is differentiable. Suppose also  that 
 \beq
 (d/ds) f(s) + g(s) \le h(s)     ,\ \ \ 0<s \le t                                          \label{vs90}
 \eeq
 Let $-\infty < b <1$ and assume that 
 \begin{align}
 \int_0^t s^{-b} f(s) ds < \infty.                                                          \label{vs90f}
 \end{align}
 Then 
 \beq
 t^{1-b} f(t) +\int_0^t s^{(1-b)}  g(s)ds \le \int_0^t s^{(1-b)}  h(s) ds 
                   +(1-b) \int_0^t s^{-b} f(s)ds     .                                           \label{vs91}
 \eeq
  If equality holds in \eref{vs90} then equality holds in \eref{vs91}.
 \end{lemma}
  \begin{proof} See  \cite[Lemma 4.8]{G70} for  a proof.
  \end{proof}

\begin{remark}\label{GFS}{\rm  (Gaffney-Friedrichs-Sobolev inequality) 
The main technique  in the next few subsections will be based on the  
Gaffney-Friedrichs  
inequality, which                     
  asserts, for our  convex subset of $\R^3$,  that for any integer $p \ge 1$ and  any  $\kf$ valued p-form $\w$ (satisfying appropriate boundary conditions)
there holds
\begin{align}
(1/2) \|\w\|_{H_1^A}^2  \le \Big\{ \|d_A^* \w\|_2^2 + \| d_A \w\|_2^2 
                                   + \lambda(B) \|\w\|_2^2\Big\}            \label{gaf49}
  \end{align}
  for any $\kf$ valued connection form $A \in W_1(M;\L^1\otimes \kf)$ with curvature $B$.
Here    
we have written  
\beq
\lambda(B) =1 + \gamma \|B\|_2^4,           \label{gaf70} 
\eeq
where $\gamma \equiv  (27/4)\kappa^6 c^4$ is a constant depending only on a 
 Sobolev constant  $\kappa$ for $M$ and the commutator bound 
  $c\equiv \sup\{\|ad\ x\|_{\kf\rightarrow \kf}: \|x\|_{\kf} \le 1\}$.
  The $H_1^A$ norm is defined in Notation \ref{notgisn}.
 
 Usually we will use the Sobolev bound that follows from this:
\beq
\|\w\|_6^2  \le \kappa^2\Big\{ \|d_A^* \w\|_2^2 + \| d_A \w\|_2^2 
                                   + \lambda(B) \|\w\|_2^2\Big\}.            \label{gaf50}
\eeq
These inequalities allow us to make good use of the Bianchi identity, which usually  
simplifies one of the terms on the right side of \eref{gaf49} and \eref{gaf50}. 
See \cite[Theorem 2.17, Remark 2.18 and  Equ.(4.31)]{CG1} for the derivation of
 these  inequalities.    
 }
 \end{remark}

\subsection{Initial behavior of $w$,  order 1}     \label{secibw1}

In Section \ref{secEUaug} we proved existence and uniqueness of strong solutions to the augmented variational equation for initial value in $H_b^\As$ with  $0< b <1$.     
We want to derive more detailed information about the short time behavior of 
 derivatives of the solution. All of our bounds on derivatives will be dominated by the following
 gauge invariant functional of the solution.
         \begin{definition}\label{defbac}{\rm 
         Let $0 \le b <1$. The {\it b-action} of a function 
$w:[0,\infty)  
 \rightarrow \{\kf\ \text{valued 1-forms on}\ M\}$ up to time $t$  is
 \begin{align}
  \nn w\nn_t^2=  
  \begin{cases}
   \int_0^t s^{-b}  \|\n^{A(s)} w(s)\|_2^2\,ds \qquad\qquad\qquad\ \ \text{if}\ M = \R^3 \\
   \int_0^t s^{-b} \Big( \|\n^{A(s)} w(s)\|_2^2 + \| w(s)\|_2^2 \Big)ds\ \ \    \text{if}\ M \ \text{is bounded.} \label{w3b}
   \end{cases}
\end{align}
A strong solution $w$ to the augmented variational equation has {\it finite b-action} if 
\begin{align}
\nn w\nn_t < \infty \ \text{for all}\  t >0.     \label{w3d} 
\end{align}
}
\end{definition}
In previous sections we have used the Sobolev norms 
given by $\|\w\|_{H_1^\As}^2 = \| \n^{A(T)} \w\|_2^2 + \| \w\|_2^2$  rather than
 the varying norms used in the integrands in \eref{w3b}. 
 The notion  
 ``finite strong  b-action'' was defined in Definition \ref{defb-act} by the condition 
 \begin{align}
\int_0^\tau s^{-b} \| w(s)\|_{H_1^\As}^2 ds < \infty\ \ \text{for some}\ \tau >0.     \label{w3c}
\end{align}
This differs from the notion of finite b-action given in Definition \ref{defbac} in two ways: 
Most importantly, the additive $L^2$ norm,
which is present in the integrand in \eref{w3c},   
 is absent
from \eref{w3b} when $M = \R^3$. Secondly, there is the distinction between use of 
 $A(s)$
versus fixed $A(T)$. This is not a significant distinction because these norms are equivalent, uniformly
for $0 \le s \le \tau$,
  by virtue of Lemma \ref{lemeqSob} and our standing assumption \eref{vst356}.
    When $M \ne \R^3$ the presence
   of the term $\|w(s)\|_2^2$  is essential for use in Sobolev inequalities for 
    dominating the $L^6$ norm  because 
   none   of our boundary conditions   requires $w$ to be zero on $\p M$. 
   In this case \eref{w3c} is equivalent to \eref{w3d}. 
   When $M = \R^3$, however,
  this added term is not needed for bounding $L^6$ norms and \eref{w3c} is strictly
   stronger than \eref{w3d} when $M =\R^3$. 
 We will use the action norm  \eref{w3b}  extensively to bound $L^6$ norms and no other  $L^p$ norms.
  It will be used  in \cite{G72} as a gauge invariant Riemannian metric on a space of
  solutions to the  Yang-Mills heat equation.

            Since a strong solution to the augmented variational equation is a continuous function on $(0, \infty)$
into $H_1^\As$, it follows that $\nn w\nn_t < \infty$ for all $t>0$ if   $\nn w\nn_t < \infty$ for some $t >0$.
  It was shown in \eref{ce401} that \eref{w3c} holds for any mild solution to the augmented variational
equation  lying in $\Q_T^{(b)}$, at least when $1/2 \le b <1$. In particular $\nn w\nn _t <\infty$ also,
for all $t >0$ if its initial value lies in $H_b^{A(T)}$.

              In this section   we are going to let $b \in [0, 1)$ and take as a hypothesis that
  our  solution $w$ has finite b-action   in the sense  of \eref{w3d}.     
   Whether $M$ is bounded or not we have the easily verified bounds
  \begin{align}
  \int_0^t s^{-b}\Big(\|d_{A(s)} w(s)\|_2^2 + \| d_{A(s)}^* w(s)\|_2^2\Big) ds &\le 4 \nn w\nn_t^2  \label{ib5a}\\
  \int_0^t s^{-b} \|w(s)\|_6^2 ds &\le \kappa_6^2 \nn w\nn_t^2.  \label{ib6a}
  \end{align}

  Our goal in this section is to establish bounds on the initial behavior of $w$
 and its derivatives
entirely in terms of the action $\nn w\nn_t$. The artificial decomposition \eref{vst20} 
and the associated estimates
will not be used in this section or any further  in this paper.

       \begin{theorem}\label{ibord1} $($Initial behavior of $w$, order 1$)$. 
Let  $0 \le b <1$.  Suppose that $A(\cdot)$ is 
a strong solution to the Yang-Mills heat equation over $[0, \infty)$ of finite action. 
Let $w(\cdot)$ be a strong solution to the augmented variational equation \eref{av1},
 not necessarily  lying in $Q_T^{(b)}$, but with finite b-action in the sense  of \eref{w3d}.
Let $\psi(s) = d_{A(s)}^* w(s)$ and  $\zeta(s) = d_A^*d_A w(s) + [w(s)\lrc B(s)]$ as in 
\eref{ib10psi} and \eref{ib549}.
      Then  there are standard dominating  functions $C_j$ such that
\begin{align}
t^{1-b}\Big\{& \| d_{A(t)} w(t)\|_2^2 +\| d_{A(t)}^* w(t)\|_2^2 \Big\}  
         + \int_0^t s^{1-b} \Big\{ \|w'(s) \|_2^2 + \| L_{A(s)} w(s)\|_2^2 \Big\} ds  \notag\\
&\ \ \ \ \ \ \le  C_{87}(t, \rho_A(t))\ \nn w\nn_t^2\ \ \qquad\qquad  \text{and}         \label{ib10b}
\end{align}
\begin{align}
&\int_0^t s^{1-b}\Big\{ \|d_A^* d_A w(s)\|_2^2 + \| d_A w(s)\|_6^2 + \| d_A \psi(s)\|_2^2    \notag\\
&\qquad\qquad  + \| \psi(s)\|_6^2 + \|d\psi(s)\|_2^2+\|\zeta(s)\|_2^2 \Big\} ds  
     \le C_{88}(t, \rho_A(t))\ \nn w\nn_t^2,  \label{ib10ba} 
\end{align}
where $L_{A(s)}$ is the Hodge Laplacian, defined in \eref{intid40}.
Moreover the following interpolation bounds hold.
\begin{align}
\int_0^t s^{(1/2) -b} \|\psi(s)\|_3^2 ds &< \infty \ \qquad\qquad  \text{if}\ \ \  0 \le b <1 
                \ \ \text{and} \label{ib10c}\\
\int_0^t \|\psi(s)\|_3 ds &= O(t^{(2b +1)/4}) \ \ \text{if}\ \ \ 0 \le b <1. \label{ib10d}
\end{align}
\end{theorem}
The proof depends on the following lemmas.

\begin{lemma}\label{propdi1}$($Differential inequality, order 1$)$  
Suppose that $A(\cdot)$ 
is a strong solution to the Yang-Mills heat equation over $(0,  \infty)$  
 with finite action
 and that $w$ is a  strong solution to \eref{av1} over $(0, \infty)$.  
Then
\begin{align}
&\frac{d}{ds} \Big\{ \| d_{A(s)} w(s) \|_2^2 + \| d_{A(s)}^* w(s) \|_2^2\Big\} \notag
  + \| w'(s)\|_2^2 + \| L_{A(s)} w(s)\|_2^2 \\
 &\le 2c  \|A'(s)\|_3 \|w(s)\|_6 \Big(\| d_A w(s)\|_2 + \| d_A^* w(s)\|_2\Big) 
 +   c^2 \|w(s)\|_6^2 \| B(s)\|_3^2  .                                                  \label{di1}                                
  \end{align}
  \end{lemma}
         \begin{proof}
         It suffices to show that the right side of \eref{intid36} is bounded by the right side 
         of \eref{di1}. But
         \begin{align*}
      &2\Big\{ ([A'\wedge w], d_A w) + ( [A'\lrc w], d_A^* w)\Big\}  + \|w\lrc B\|_2^2 \\ 
  &\ \ \  \le  2c \Big\{ \|A'(s)\|_3 \|w(s)\|_6 \| d_A w(s)\|_2 
                          +   \|A'(s)\|_3 \|w(s)\|_6 \| d_A^* w(s)\|_2 \Big\} \\
     &\ \ \ +c^2 \|w(s)\|_6^2 \| B(s)\|_3^2,
      \end{align*}
which is \eref{di1}.      
      \end{proof}

The term $\| L_{A(s)} w(s)\|_2^2$ in line \eref{ib10b} contains second derivatives of $w$. We wish
to use these second derivatives to estimate $L^6$ norms of the first derivatives of $w$.
 However the cross terms
in the expansion of $\|L_{A(s)} w(s)\|_2^2$ will have to be separated out first and controlled 
before we can use the Gaffney-Friedrichs-Sobolev  inequality \eref{gaf50}. 
The next lemma is aimed at this.

\begin{lemma} \label{lemcross} $($Cross terms$)$
Suppose that $A$ is a solution to the Yang-Mills heat equation over $(0, \infty )$ 
 with finite action.
Let $0 \le b <1$. If  $w$ is a function 
$($not necessarily  a solution$)$ with finite b-action then there exists a standard
 dominating  function $C_{93}$  such that 
\begin{align}
\int_0^t s^{1-b}\Big(\|d_A^*&d_A w(s)\|_2^2  + \| d_A d_A^* w(s)\|_2^2\Big) ds  \notag\\
&\le   2\int_0^t s^{1-b}\| L_{A(s)} w(s)\|_2^2 ds  
 + C_{93}(t ,\rho_A(t)) \nn w\nn_t^2 \ .  
                                                       \label{ib2b}
\end{align}
\end{lemma}

The proof depends on the following lemma.
\begin{lemma} \label{crossterms} $($Cross term inequality$)$
Let $0 \le b <1$ and let  $ s >0$.  
 Suppose that $w(s)$ lies in the domains of 
both $d_{A(s)}^* d_{A(s)}$ and $d_{A(s)}d_{A(s)}^*$. 
Let
\beq
U(s)= 2 ( d_A w(s), [B(s), d_A^* w(s)])        .                         \label{ib4b}
\eeq
Then 
\beq
\|d_A^* d_A w(s)\|_2^2 + \|d_A d_A^* w(s)\|_2^2 
            = \| L_{A(s)} w(s)\|_2^2 - U(s) .
                                                    \label{ib5b}
\eeq
Moreover 
\begin{align}
s^{1-b}& 
    |U(s)| \le (1/2)s^{1-b} \|d_A^*d_A w(s)\|_2^2                                 \label{ib8b.0}\\
&
+c^2 \Big\{s^{3/2} \| B(s)\|_6^2 \Big\} \Big( s^{-b} \| d_A^* w(s) \|_2^2 \Big) 
+(\kappa^2/2) \Big( s^{-b} \| d_A w(s)\|_2^2\Big)                                          \label{ib8b.1}\\
 &+ (1/2)c^2 s^{1-b}\|w\|_6^2 \| B\|_3^2  
        + \Big\{s\lambda(B(s))/2\Big\}\Big(s^{-b} \| d_A w(s) \|_2^2\Big).              \label{ib9b} 
\end{align}
\end{lemma}      
         \begin{proof} Expand $\| L_A w\|_2^2$ to find
\begin{align*}
\| L_A \w \|_2^2 &= (d_A^* d_A \w + d_A d_A^* \w,  d_A^* d_A \w + d_A d_A^* \w) \\
&=  \|d_A^* d_A \w\|_2 +   \|d_A d_A^* \w\|_2^2 + 2(d_A^*d_A\w, d_Ad_A^* \w).
\end{align*}
The last term is $2(d_A \w, d_A^2 d_A^*\w)$, which is $2(d_A \w, [B(s),d_A^*\w])$, by the Bianchi identity. 
This gives   \eref{ib5b} in view of the definition of $U(s)$ in \eref{ib4b}. 
By H\"older's inequality we now find
 \begin{align}
 s^{1-b} |U(s)| 
      &=2s^{1-b}  \Big|( d_A w(s), [ B(s), d_A^* w(s)])\Big|  \notag\\
    & \le  2cs^{1-b} \| B(s)\|_6 \|d_A^* w(s)\|_2 \|d_A w(s)\|_3             \notag \\
    &= 2c \Big(s^{3/4} \| B(s)\|_6\Big) \Big(s^{-b/2} \|d_A^* w(s)\|_2 \Big) 
    \Big(s^{(1/4) - (b/2)} \|d_A w(s)\|_3\Big)                         \notag\\
    &\le     c^2 \Big\{s^{3/2}\| B(s)\|_6^2\Big\} \Big( s^{-b} \|d_A^* w(s)\|_2^2\Big)  
                   +  s^{(1/2) - b}\|d_A w(s)\|_3 ^2.     \label{ib6b}
    \end{align}
  The first of these two terms is the first term in line \eref{ib8b.1}. The second term in line
   \eref{ib6b}can be dominated  by interpolation  between $L^2$ and $L^6$ thus:
    \begin{align}
   s^{(1/2) - b}& \| d_A w(s)\|_3^2 
   \le  \Big(\kappa s^{-b/2} \| d_Aw(s)\|_2\Big) 
                          \Big(s^{(1-b)/2} \kappa^{-1}\| d_A w(s)\|_6\Big)                    \notag \\
    &\le (1/2) \kappa^2 s^{-b} \| d_A w(s)\|_2^2 
                         + (1/2) s^{1-b} \kappa^{-2} \| d_A w(s)\|_6^2.                     \label{ib7b}
    \end{align}
    The first term in line \eref{ib7b} is the second term in line \eref{ib8b.1}. 
We can dominate the second term in line \eref{ib7b} by applying the
 Gaffney-Friedrichs-Sobolev inequality \eref{gaf50}
to the 2-form $\w = d_Aw(s)$.
We find
\begin{align*}
& (1/2)s^{1-b} \kappa^{-2}\| d_A w(s) \|_6^2   \notag\\
&\ \ \ \ \le (1/2)s^{1-b} \Big(\| d_A^* d_A w(s)\|_2^2    +  \| d_A d_A w(s)\|_2^2  
+ \lambda(B(s)) \| d_A w(s) \|_2^2\Big)                            \notag\\ 
  &\ \ \ \ \le    (1/2) s^{1-b} \| d_A^* d_A w(s)\|_2^2   
  +(1/2)s^{1-b}\|\,[B(s)\wedge w(s)]\,\|_2^2   \\
  &\qquad\qquad      +(1/2) \Big\{s\lambda(B(s))\Big\}\Big(s^{-b} \| d_A w(s) \|_2^2\Big).\\
  &\le  (1/2) s^{1-b} \| d_A^* d_A w(s)\|_2^2   
  +(1/2)s^{1-b}c^2 \|w(s)\|_6^2 \| B(s)\|_3^2   \\
  &\qquad\qquad      +(1/2) \Big\{s\lambda(B(s))\Big\}\Big(s^{-b} \| d_A w(s) \|_2^2\Big).
\end{align*}
The three terms on the right are the terms that appear in lines \eref{ib8b.0} and \eref{ib9b}.
This completes the proof of Lemma \ref{crossterms}.
\end{proof}

\bigskip
\noindent
\begin{proof}[Proof of Lemma \ref{lemcross}]
From \eref{ib5b} we see that
\begin{align}
&\int_0^t s^{1-b}\Big(\|d_A^*d_A w(s)\|_2^2  + \| d_A d_A^* w(s)\|_2^2\Big) ds  \notag\\
&\ \ \ \ \ \ \le   \int_0^t s^{1-b}\| L_{A(s)} w(s)\|_2^2 ds   +\int_0^t s^{1-b}|U(s)|ds              \notag\\ 
&\ \ \ \ \ \ \le  \int_0^t s^{1-b}\| L_{A(s)} w(s)\|_2^2 ds 
+ (1/2)\int_0^t s^{1-b} \|d_A^*d_A w(s)\|_2^2ds            \label{ib20b} \\
&+\int_0^t \Big\{c^2 \Big(s^{3/2} \| B(s)\|_6^2 \Big) \Big( s^{-b} \| d_A^* w(s) \|_2^2 \Big) 
           +(\kappa^2/2) \Big( s^{-b} \| d_A w(s)\|_2^2\Big)                                          \notag\\
            &+ (1/2)c^2 s^{1-b}\|w(s)\|_6^2 \| B(s)\|_3^2
               + \Big(s\lambda(B(s))/2\Big)\Big(s^{-b} \| d_A w(s) \|_2^2\Big)   \Big\}ds.   \label{ib21b}
\end{align}
The second term in line \eref{ib20b} cancels with half of one term on the left.  It suffices
to show,  therefore, that the integral of each of the four terms in the last two lines
can be dominated by an expression of the form $C(t,\rho_A(t)) \nn w\nn_t^2$.
 These four integrals  add to at most
\begin{align*}
&   c^2 (\sup_{0< s \le t}s^{3/2} \|B(s)\|_6^2)\int_0^t s^{-b} \| d_A^* w(s) \|_2^2ds +
(\kappa^2/2)\int_0^t s^{-b} \| d_A w(s)\|_2^2ds \\
& \qquad\qquad + (c^2/2)(\sup_{0< s \le t}s\|B(s)\|_3^2) \int_0^t s^{-b} \|w(s)\|_6^2 ds  \\
&\qquad\qquad + (\sup_{0 <s \le t} s\lambda(B(s))/2)\int_0^t s^{-b} \| d_A w(s) \|_2^2 ds.
\end{align*}
All three suprema are bounded by standard dominating functions of $t, \rho_A(t)$ in accordance
with Lemma \ref{lemiby} with $a = 1/2$. All four integral  factors  are dominated by $\nn w \nn_t^2$
by \eref{ib5a} and \eref{ib6a}. 
This concludes the proof of \eref{ib2b}.
\end{proof}

\bigskip
\noindent
\begin{proof}[Proof of Theorem \ref{ibord1}]
For the proof of \eref{ib10b} we need only apply  Lemma \ref{lem50}   with 
 $f, g, h$ chosen to match up with the differential inequality \eref{di1}. 
 Thus we take 
$f(s) =  \| d_{A(s)} w(s)\|_2^2 +\| d_{A(s)}^* w(s)\|_2^2$, 
take $g(s) = \|w'(s) \|_2^2 + \| L_{A(s)} w(s)\|_2^2$
 and take $h(s)$ to be the entire right hand side of \eref{di1}.
 We find from \eref{vs91} that  
\begin{align}
t^{1-b}\Big\{& \| d_{A(t)} w(t)\|_2^2 +\| d_{A(t)}^* w(t)\|_2^2 \Big\}  
         + \int_0^t s^{1-b} \Big\{ \|w'(s) \|_2^2 + \| L_{A(s)} w(s)\|_2^2 \Big\} ds \notag\\
&\le \int_0^t  s^{1-b}
\Big\{ 2c  \|A'(s)\|_3 \|w(s)\|_6 \Big(\| d_A w(s)\|_2 + \| d_A^* w(s)\|_2\Big) \label{ib16.1}\\
 &\qquad \qquad \qquad+   c^2 \|w(s)\|_6^2 \| B(s)\|_3^2 \Big\}ds \label{ib16.2} \\
 &\ +(1-b)\int_0^t s^{-b}\Big\{ \| d_{A(s)} w(s) \|_2^2 
                    + \| d_{A(s)}^* w(s) \|_2^2\Big\} ds.                                  \label{ib16b}
\end{align}
The integrals in lines \eref{ib16.1} and \eref{ib16.2} add to at most
\begin{align}
&(2c)(\sup_{0< s \le t} s \|A'(s)\|_3)\Big(\int_0^t s^{-b} \|w(s)\|_6^2 ds \Big)^{1/2} \cdot      \notag\\
&\qquad\qquad\qquad\qquad\qquad\Big(\int_0^t s^{-b} \Big(\| d_A w(s)\|_2 
                                     + \| d_A^* w(s)\|_2\Big)^2 ds \Big)^{1/2}                               \notag\\
&+ c^2 \sup_{0< s\le t} s \|B(s)\|_3^2 ) \int_0^t s^{-b} \|w(s)\|_6^2 ds.                      \label{ib16c}
\end{align}
The two suprema in \eref{ib16c} are bounded  by standard dominating functions, in accordance
 with Lemma \ref{lemiby}, while the integral factors  are 
 dominated by $\nn w \nn_t^2$ by \eref{ib5a} and \eref{ib6a}. 
 The integral in \eref{ib16b} is also dominated by $\nn w \nn_t^2$.
 This completes  the proof of \eref{ib10b}.

   For the proof of \eref{ib10ba} observe that the inequality \eref{ib2b} combined with 
   \eref{ib10b} shows that the integral of the first  and third  terms in \eref{ib10ba}
    is finite and in fact    dominated by an $A$ dependent multiple of $\nn w\nn_t$.  
  
   The second  term in line
\eref{ib10ba} is integrable by the GFS inequality \eref{gaf50}  because
\begin{align}
\kappa^{-2} \|d_A w(s)\|_6^2 \le \|d_A^* d_A w(s)\|_2^2 + \|(d_A)^2 w(s)\|_2^2 + \lambda(B(s))
\|d_A w(s)\|_2^2,                
\end{align}
which implies
\begin{align}
\kappa^{-2} \int_0^t s^{1-b} \|d_A w(s)\|_6^2 ds 
&\le \int_0^t  \Big\{s^{1-b}\|d_A^* d_A w(s) \|_2^2 + s^{1-b}\|\, [ B(s) \wedge w(s)]\, \|_2^2  \notag\\
&+ s\lambda(B(s))  \Big(s^{-b} \| d_A w(s)\|_2^2\Big) \Big\} ds .      \label{ib18c}
\end{align}
The first term on the right in \eref{ib18c} is integrable since it is equal to the first term on the left in \eref{ib10ba}, whose integrability has already been proven.
The second term  in \eref{ib18c}is at most 
 $c^2\Big(s\|B(s)\|_3^2\Big) \Big(s^{-b} \|w(s)\|_6^2\Big)$,  which, 
 in view of     \eref{iby6} (with $a = 1/2$)  
  is a bounded function times an integrable function, as is the  third term also.

     The fourth term in \eref{ib10ba} is integrable by an application of the ordinary
  Sobolev inequality.    
    Indeed, since  $d_A^* w(s)$ is a 0-form
   Sobolev's inequality  shows that   
 $ \kappa^{-2}\| d_A^* w(s) \|_6^2  \le  \| d_A d_A^* w(s)\|_2^2   + \| d_A^* w(s)\|_2^2$,
from which the integrability of $s^{1-b} \| \psi(s)\|_6^2$   
  follows because, upon multiplication by $s^{1-b}$, 
 the first term on the right is integrable and the second term is bounded, by \eref{ib10b}, 
  and therefore integrable. 

      The fifth term in \eref{ib10ba} differs only slightly from the third term because
       $\|d\psi(s)\|_2 \le \|d_A\psi(s)\|_2 + \| \, [ A(s)\wedge \psi(s)]\,\|_2$.
       But 
    \begin{align}
       \int_0^t s^{1-b}  \| \, &[ A(s)\wedge \psi(s)]\,\|_2^2 ds 
       \le  c^2\int_0^t   \|  A(s)\|_6^2\  s^{1-b} \|\psi(s)\|_3^2 ds  \notag \\
 &\le c^2 \sup_{0 < s \le t} (s^{1/2} \|A(s)\|_6^2)\int_0^t s^{(1/2)- b}\|\psi(s)\|_3^2 ds, \label{ib18e}
       \end{align}   
 while
 \begin{align}      
   \int_0^t s^{(1/2)- b}&\|\psi(s)\|_3^2 ds 
   \le \int_0^t(s^{-b/2} \|\psi(s)\|_2) ( s^{(1-b)/2} \| \psi(s)\|_6)ds   \notag\\
   &\le \Big(\int_0^t s^{-b} \| \psi(s)\|_2^2 ds \Big)^{1/2} 
   \Big(\int_0^t s^{1-b} \| \psi(s)\|_6^2 ds \Big)^{1/2}  .     \label{ib18f}
 \end{align}      
 The first factor in \eref{ib18f} is finite because $w$ has finite b-action. 
 The second factor already appears as the fourth term on the left in  \eref{ib10ba} and is therefore finite.
 The supremum in line \eref{ib18e} is finite by virtue of \eref{ibA5a}. 
 
 Concerning the sixth term in \eref{ib10ba}, the augmented variational equation \eref{av1} shows that 
 $\zeta(s) = - w'(s) - d_A \psi(s)$. But $\int_0^t s^{1-b}\(\|w'(s)\|_2^2 + \|d_A \psi(s)\|_2^2\) ds < \infty$
  by \eref{ib10b} and \eref{ib10ba} (third term). 
  This completes the proof of \eref{ib10ba}.

 The inequality  \eref{ib10c}  follows from \eref{ib18f}. Finally, the Schwarz inequality shows that
 \begin{align*}
 \int_0^t \|\psi(s)\|_3 ds &\le \(\int_0^t s^{b-(1/2)} ds\)^{1/2} \(\int_0^ts^{(1/2)-b} \|\psi(s)\|_3^2 ds\)^{1/2} \\
 &= o(t^{(b +(1/2))/2}),
 \end{align*}
 which is \eref{ib10d}. 
This completes the proof of Theorem \ref{ibord1}.
 \end{proof}

\subsection{Initial behavior of $w$, order 2}     \label{secib2}

        \begin{theorem}\label{ibord2b} $($Initial behavior, order 2$)$. 
 Suppose that $A(\cdot)$ is 
a strong solution to the Yang-Mills heat equation over $[0, \infty)$ with  
 finite action. Let $0 \le b <1$.
Let $w(\cdot)$ be a strong solution to the augmented variational equation \eref{av1} along
$A(\cdot)$ with finite b-action in the sense  of \eref{w3d}.  
Then there are standard dominating functions $C_j$ such that, for  $0 < t < \infty$, 
there holds
     \begin{align}
t^{2-b} \| w'(t)\|_2^2 
+\int_0^t &s^{2-b} \Big\{\|d_{A(s)} w'(s) \|_2^2  + \|d_{A(s)}^* w'(s) \|_2^2  \Big\}ds   \label{ib540b} \\  
              & \le  \nn w\nn_t^2\ C_{83}(t, \rho_A(t))      \ \ \ \ \qquad\qquad \text{and}                \notag\\           
\int_0^t s^{2-b} \| w'(s)\|_6^2 ds &\le  \nn w\nn_t^2\ C_{91}(t, \rho_A(t)).      \label{ib542b}
\end{align}
Define $\psi$ and $\zeta$ as in \eref{ib10psi} and \eref{ib549}.
The following integral bounds on the third order derivatives of $w$ hold.
        \begin{align}
\int_0^T s^{2-b}\( \|d_A^*d_A\psi(s)\|_2^2     
+ \|d_A^*\zeta(s)\|_2^2 + \|d_A\zeta(s)\|_2^2 \)ds < \infty.        \label{ib542c}
\end{align}
Moreover 
\begin{align}
\int_0^t s^{2-b}  \|d_A \psi(s)\|_6^2  ds  &< \infty, \ \ \ 0\le b <1, \ \   \label{ib543} \\
\int_0^t s^{2-b}  \| d_A^* d_A w(s)\|_6^2ds &< \infty, \ \ \ 0 \le b <1,   \label{ib545}       \\
\int_0^t s^{2-b} \|\zeta(s)\|_6^2 ds &< \infty, \ \ \  0 \le b <1.                  \label{ib545z}                         
\end{align}

 The following interpolation consequences hold. 
\begin{align} 
&\int_0^T \|\psi(s)\|_q^2 ds < \infty,  
           \ \ \ \ \ \ \ q^{-1} = (1/2) - (b/3),\ \ 0 \le b < 1.                     \label{ib546}\\
&\int_0^T s^{1/2} \Big(\| d_A \psi(s)\|_r^2  +  \| d_A^* d_A w(s)\|_r^2  
+ \|\zeta(s)\|_r^2 + \|d\psi(s)\|_r^2\Big) ds < \infty, \label{ib547}\\
 &\qquad\qquad\qquad \qquad \qquad\ \ \ \ r^{-1} = (2/3) - (b/3), \ \ (1/2) \le b < 1. \notag \\
 &\int_0^T s^{(3/2) -b}\( \| d_{A} \psi(s)\|_3^2 + \| d_A^* d_A w(s)\|_3^2  \notag \\
  &\qquad\qquad\ \ \  + \| \zeta(s)\|_3^2 + \|d\psi(s)\|_3^2 \)ds < \infty, \ \ 0 \le b <1.    \label{ib548} \\
  & \int_0^T \|  d_{A(s)} \psi(s)\|_\rho ds < \infty\ \ \text{if}\ \ 2 \le \rho < 3,\ \ 1/2 \le b < 1 \label{ib715}
 \end{align}
 
 \end{theorem}

\begin{remark}{\rm We will show in the next section by different methods that
$\int_0^t s^{(3/2) - b}  \| \psi(s)\|_\infty^2 ds < \infty$, which, interestingly, holds even though
 \eref{ib548} just
 barely fails to give this inequality because Sobolev just barely fails to give control
  of $\|\psi(s)\|_\infty$ by $\| d_A\psi(s)\|_3$.
}
\end{remark}

The proof of Theorem \ref{ibord2b} depends on the following lemmas.

             \begin{lemma}\label{diffinb} $($Differential inequality$)$
             Suppose that $w(\cdot)$ is a strong solution to the
augmented variational equation \eref{av1} on the interval $(0,T]$.
Then   
\begin{align}
\frac{d}{ds} \| w'(s) \|_2^2 
+&\Big\{ \|d_{A(s)} w'(s) \|_2^2  + \|d_{A(s)}^* w'(s) \|_2^2  \Big\}   \notag\\
\le  &\ c_1\, \|A'(s)\|_3^2 \Big(\|w(s)\|_6^2 
                               + 2\kappa^2  ( \|d_Aw\|_2^2 + \|d_A^* w\|_2^2 ) \Big) \notag\\
+&\ c_2\, \| B(s)\|_2^4 \|w'(s)\|_2^2 + c_3\, s^{-1} \|w'(s)\|_2^2  + c_4\, \|w'(s)\|_2^2 \notag\\
+ &\ c_5\,  s^{1/2} \|B'(s)\|_2^2 \| w(s)\|_6^2     \label{ib550b}
\end{align}
for some constants $c_j$ that depend only on a Sobolev constant $\kappa$ and
 the commutator bound $c$.
\end{lemma}
     \begin{proof}  
     Our strategy will be to bound the terms on the right side of the integral identity \eref{intid38}.
In our use of H\"older's inequality we will be forced to use 
$\| w'(s)\|_6$    as a frequent factor. By the Gaffney-Friedrichs-Sobolev inequality
this has the same degree of singularity (as $s \downarrow 0$) as some of the terms
on the left side of \eref{intid38}. We will arrange the estimates in such a way as to allow
cancellation of some of these singular terms.

\bigskip   
  We are going to use H\"older's inequality repeatedly to show that the right side of 
  the identity \eref{intid38}   is at most
  \begin{align}
  &2c^2 \|A'\|_3^2\|w\|_6^2   + (1/2)  \Big( \|d_A w'\|_2^2 +  \|d_A^*w'\|_2^2\Big)  \label{ib401} \\
+ &4c^2\kappa^2\|A'\|_3^2 \Big( \|d_Aw\|_2^2 + \|d_A^* w\|_2^2 \Big) 
                                                            + (1/4) \kappa^{-2} \|w'\|_6^2 \label{ib402}\\
 +& (1/4)(12c\kappa^{3/2})^4 \|B\|_2^4 \|w'\|_2^2   + (1/8) \kappa^{-2} \|w'\|_6^2  \label{ib403}\\
 +&  8(c\kappa)^2 s^{1/2} \| B'\|_2^2 \| w\|_6^2 +
                   2^{-5} \kappa^{-2}s^{-1} \|w'\|_2^2 +  (1/8) \kappa^{-2}\|w'\|_6^2  .     \label{ib404}
\end{align}
Here, as below, we have suppressed the argument $s$ in all functions.
  
        To bound the first two terms on the right side of \eref{intid38} observe that
                  \begin{align}
 2|([A'\wedge w],&d_A w') + ([A'\lrc w],d_A^* w')|                     \notag \\
 &\le 2\Big( \|\, [A' \wedge w]\,\|_2 \|d_A w'\|_2 +\|\, [A'\lrc w]\,\|_2 \| d_A^* w'\|_2\Big) \notag\\
 & \le 2c  \|A'\|_3  \|w\|_6\Big( \|d_A w'\|_2 + \| d_A^* w'\|_2\Big)       \notag\\
 &\le (1/2) \Big(2c \|A'\|_3\|w\|_6\Big)^2 
           + (1/2)  \Big( \|d_A w'\|_2^2 +  \|d_A^*w'\|_2^2\Big).                   \label{ib551b}
 \end{align}  
 This contributes the line \eref{ib401} in our bound of the right side of \eref{intid38}.
 
      To bound  the second pair of terms on the right side on \eref{intid38},
       we use  $2ab \le (2\kappa a)^2 +(b/2\kappa)^2$ twice, with $\kappa$ equal to the
        Sobolev constant in \eref{gaf50}, to find
 \begin{align}
 2| ( [ A'\lrc d_A  w] &+    [ A', d_A^* w], w')| 
 \le 2c\|A'\|_3 \big( \|d_A w\|_2 + \| d_A^* w \|_2 \Big) \|w'\|_6    \notag\\
 &\le  (2c\kappa\|A'\|_3)^2 \Big( \|d_Aw\|_2^2 + \|d_A^* w\|_2^2 \Big) 
 + (1/4) \kappa^{-2} \|w'\|_6^2.                                                             \label{ib552b}
\end{align} 
This contributes the line \eref{ib402} in our bound of the right side of \eref{intid38}.

The last term in \eref{intid38} is $2 ([w'\lrc B] + [w\lrc B'] , w')$.  These two terms must 
 be estimated in different ways.               
To estimate $2([w'\lrc B], w')$  we can use the interpolation 
           $\|f\|_4 \le \|f\|_2^{1/4} \|f\|_6^{3/4}$ to find 
 \begin{align}
2|( &[w'\lrc B], w')| \le 2c \|B\|_2 \|w'\|_4^2                                        \notag\\
&\le \Big(12c\kappa^{3/2} \|B\|_2 \|w'\|_2^{1/2}\Big) 
                       \Big( (1/6)\kappa^{-3/2} \|w'\|_6^{3/2}\Big)                  \notag\\
&\le  (1/4) \Big(12c\kappa^{3/2} \|B\|_2 \|w'\|_2^{1/2}\Big)^4 
               + (3/4) \Big((1/6)\kappa^{-3/2} \|w'\|_6^{3/2}\Big)^{4/3}       \notag\\
&\le (1/4)(12c\kappa^{3/2})^4 \|B\|_2^4 \|w'\|_2^2   + (1/8) \kappa^{-2} \|w'\|_6^2  \label{ib554b}
\end{align}
because $(3/4) (1/6)^{4/3} \le 1/8$. This contributes the line \eref{ib403} in
 our bound of the right side of \eref{intid38}.

Our estimate of the final term  $2( [w\lrc B'], w')$ appearing in \eref{intid38}
 will have  an explicit $s$ dependence.  We have
\begin{align}
2|(& [w\lrc B'], w')| \le 2c  \| B'\|_2\| w\|_6  \|w'\|_3                                     \notag\\
&= \Big(4c\kappa s^{1/4}\|B'\|_2\|w\|_6\Big)     
                \Big((1/2) \kappa^{-1} s^{-1/4}\| w'\|_3\Big)                             \notag\\
&\le (1/2)   \Big(4c\kappa s^{1/4} \|B'\|_2\|w\|_6 \Big)^2 
         +  (1/2) \Big((1/2) \kappa^{-1} s^{-1/4}\| w'\|_3\Big)^2                              \notag\\
& = 8(c\kappa)^2 s^{1/2} \| B'\|_2^2 \| w\|_6^2      
                              + (1/8) \kappa^{-2} s^{-1/2}\|w'\|_3^2                            \notag \\   
  & \le  8(c\kappa)^2 s^{1/2} \| B'\|_2^2 \| w\|_6^2 
  +  (1/8) \kappa^{-2} \{(1/4)s^{-1} \|w'\|_2^2 + \|w'\|_6^2\} .   \label{ib553a}
\end{align}
In the last line we have used $s^{-1/2} \|w'\|_3^2 \le (s^{-1/2} \|w'\|_2) \| w'\|_6
 \le (1/4)s^{-1}\|w'\|_2^2 + \|w'\|_6^2$.
 
      This completes the proof that the right hand side of \eref{intid38} is dominated by the sum
      of the four lines \eref{ib401} -\eref{ib404}.

         Notice that the last term in each of the three lines \eref{ib402}- \eref{ib404} is a multiple
  of $\|w'\|_6^2$ and  they add to $(1/2) \kappa^{-2} \|w'(s)\|_6^2$.
   From the Gaffney-Friedrichs-Sobolev inequality  \eref{gaf50} we find  
$$
(1/2)\kappa^{-2} \|w'\|_6^2 \le (1/2)\Big( \|d_A w'\|_2^2 +  \|d_A^*w'\|_2^2\Big) 
        +(1/2)\lambda(B) \|w'\|_2^2.
$$
Therefore, adding the last term in line \eref{ib401}  to the last terms in the next three lines,
 we find  that the sum of the last terms  in all four lines is at most
\beq
\Big( \|d_A w'\|_2^2 +  \|d_A^*w'\|_2^2\Big) 
        +(1/2)\lambda(B) \|w'\|_2^2.
\eeq
The term  $\Big( \|d_A w'\|_2^2 +  \|d_A^*w'\|_2^2\Big)$ appears on the left side of
\eref{intid38} with a factor of 2. We can therefore cancel this term with half of its multiple
on the left to find that the left side of \eref{ib550b} is bounded by the remaining terms
in the lines \eref{ib401} -\eref{ib404} plus  $(1/2)\lambda(B) \|w'\|_2^2$.
These add to the right side of \eref{ib550b}. 
        This completes the proof of \eref{ib550b}.
 \end{proof} 
 \bigskip

\begin{lemma}\label{lemrecinterp} $($Interpolation bounds$)$ Suppose that $f:[0,t)\rightarrow $
 \{functions on $M$\}.  

\noindent 
Assume that $1/2 \le b \le 3/2$. Let $r^{-1} = (2/3) - (b/3)$. Then $2 \le r \le 6$ and 

\noindent

\begin{align}
\int_0^t s^{b} \|f(s)\|_r^2 ds &\le \Big(\int_0^t s^{1/2} \|f(s)\|_2^2 ds\Big)^{(3/2) -b} 
            \Big( \int_0^t s^{3/2} \| f(s)\|_6^2 ds\Big)^{b - (1/2)}, \label{ib214}   \\                                               
\int_0^t s^{1/2} \|f(s)\|_r^2 ds &\le  \Big(\int_0^t s^{1-b} \|f(s)\|_2^2 ds\Big)^{(3/2) -b} 
            \Big( \int_0^t s^{2-b} \| f(s)\|_6^2 ds\Big)^{b - (1/2)}   ,           \label{ib215} \\
 \int_0^t s^{-1/2} \|f(s)\|_r^2 ds &\le  \Big(\int_0^t s^{-b} \|f(s)\|_2^2 ds\Big)^{(3/2) -b} 
            \Big( \int_0^t s^{1-b} \| f(s)\|_6^2 ds\Big)^{b - (1/2)}   .           \label{ib216}
\end{align} 
Assume that $0\le b \le 1$. Let $q^{-1} = (1/2) - (b/3)$. Then $2\le q  \le 6$ and 
\begin{align}
\int_0^T \|f(s)\|_q^2 ds &\le \Big(\int_0^T s^{-b} \|f(s)\|_2^2 ds \Big)^{1-b}
\Big(\int_0^T  s^{1-b} \|f(s)\|_6^2 ds \Big)^{b} .                      \label{ib217}
\end{align}
Moreover, if $\rho \ge 2$ and $ 1/\rho > (1/2) - (b/3)$ then
\begin{align}
\int_0^T \|f(s)\|_\rho ds 
\le C_{b, \rho}   \Big(\int_0^t s^{1-b} \|f(s)\|_2^2 ds\Big)^{\alpha/2}  
            \Big( \int_0^t s^{2-b} \| f(s)\|_6^2 ds\Big)^{\beta/2},               \label{ib218}
       \end{align}
       with a finite constant $C_{b, \rho}$ and non-negative constants $\alpha, \beta$.

\bigskip       
\noindent
Assume that $ - \infty < b < 2$. Then  
\begin{align}
\int_0^T s^{(3/2)- b} \|f(s)\|_3^2 ds \le \Big(\int_0^T s^{1-b} \|f(s)\|_2^2 ds \Big)^{1/2}
\Big(\int_0^T s^{2-b} \| f(s)\|_6^2ds \Big)^{1/2}. \label{ib219}
\end{align}
\end{lemma}
\begin{proof}All of these inequalities  are consequences of the interpolation inequality
\begin{align}
&\|f(s)\|_\rho^2  \le \Big(\|f(s)\|_2^2\Big)^\alpha \Big(\| f(s)\|_6^2\Big)^{\beta},  \label{ib220} \\
           &\ \ \ \ \ \ \ \ \alpha = (3/\rho) - (1/2),\ \ \ \ \  \beta = (3/2) - (3/\rho),            \notag
\end{align}
which is valid for $2 \le \rho \le 6$.
\eref{ib220} implies that, for any real number $\gamma$,
\begin{align}
s^{\gamma + \beta} \|f(s)\|_\rho^2 
\le  \Big(s^{\gamma}\|f(s)\|_2^2\Big)^\alpha 
                     \Big(s^{1+\gamma}\| f(s)\|_6^2\Big)^{\beta}.                               \label{ib221}
\end{align}
Integrate this inequality over $(0,t)$ and use Holder's inequality to find
\begin{align}
\int_0^t s^{\gamma + \beta} \|f(s)\|_\rho^2 ds
\le  \Big(\int_0^t s^{\gamma}\|f(s)\|_2^2ds\Big)^\alpha 
               \Big(\int_0^ts^{1+\gamma}\| f(s)\|_6^2ds\Big)^{\beta}.                            \label{ib222}
\end{align}
All four  
 inequalities in the statement of the lemma now result from proper choice of $\gamma$. Thus:

To prove \eref{ib214}, \eref{ib215}  and \eref{ib216} take $\rho =r$ and $b = 2 -3r^{-1}$ to find
 $\beta = b -(1/2)$ in  all three cases.
Choose $\gamma = 1/2$ to find
  $\gamma + \beta = b$, from which \eref{ib214} follows.
 Choose $\gamma = 1-b$ to find $\gamma + \beta =1/2$, from which \eref{ib215} follows.
 And choose $\gamma = -b$ to find  $\gamma + \beta =-1/2$,  from which \eref{ib216} follows.

        To prove    \eref{ib217}  replace $\rho $ by $q$ in \eref{ib220} - \eref{ib222} and
 take $b =(3/2) - 3q^{-1}$    to find  $\beta = b$.    
Choose  $\gamma = -b$ to find $\gamma + \beta =0$, from which \eref{ib217} follows.

         For the proof of \eref{ib218} observe that by \eref{ib220} we have
\begin{align}
&\int_0^\tau \|f(s)\|_\rho ds \le \int_0^\tau s^{((b-1)\alpha +(b-2)\beta)/2} \(s^{(1-b)/2} \|f(s)\|_2\)^\alpha 
             \(s^{(2-b)/2} \|f(s)\|_6\)^\beta ds                                 \notag\\
  &\le \(\int_0^\tau s^{b-1-\beta} ds \)^{1/2}\(\int_0^\tau s^{1-b} \|f(s)\|_2^2ds\)^{\alpha/2}
         \(\int_0^\tau s^{2-b} \|f(s)\|_6^2ds\)^{\beta/2},        \label{rec715} 
             \end{align}
             wherein we have used   H\"older's inequality with the three powers $2, (2/\alpha), (2/\beta)$.
             But $ b -\beta = b -\{(3/2) - (3/\rho)\}  = 3\{(1/\rho) - [(1/2) - (b/3)]\}$. Hence the first integral in
             \eref{rec715} is finite if and only if $ (1/\rho) - [(1/2) - (b/3)] >0$. This proves \eref{ib218}.

               To prove \eref{ib219} choose $\rho =3$, giving $\alpha = \beta =  1/2$, and 
  choose $\gamma =1-b $ in \eref{ib222}.
\end{proof}

\bigskip
\noindent
             \begin{proof}[Proof of Theorem \ref{ibord2b}] 
Starting with the differential inequality   \eref{ib550b},   we may apply
Lemma \ref{lem50}, choosing  
$f(s) = \|w'(s)\|_2^2$,  $g(s) =  \|d_{A(s)} w'(s) \|_2^2  + \|d_{A(s)}^* w'(s) \|_2^2 $
 and  $h(s)$ equal  to the entire right  hand side of  \eref{ib550b}. 
Replace  the number denoted $b$ in  Lemma \ref{lem50}  by $b-1$, with our
 present meaning of $b$.  
   Then \eref{vs91} shows that 
 
\begin{align}
t^{2-b}& \| w'(t)\|_2^2 
+\int_0^t s^{2-b} \Big\{ \|d_{A(s)} w'(s) \|_2^2  + \|d_{A(s)}^* w'(s) \|_2^2  \Big\}ds     \notag\\
&\le (2-b) \int_0^t s^{1-b} \|w'(s)\|_2^2ds      \label{ib561b}\\
+&\int_0^t  s^{2-b} \Big\{\ c_1\, \|A'(s)\|_3^2 \Big(\|w(s)\|_6^2 
                               + 2\kappa^2  ( \|d_Aw\|_2^2 + \|d_A^* w\|_2^2 ) \Big)       \label{ib562b}\\
&\ \ \ \ + c_2\, \| B(s)\|_2^4 \|w'(s)\|_2^2 + c_3\, s^{-1} \|w'(s)\|_2^2  
                                 + c_4\, \|w'(s)\|_2^2                                                                \label{ib563b}\\
 &\ \ \  \ +\ c_5\,  s^{1/2} \|B'(s)\|_2^2 \| w(s)\|_6^2  \Big\} ds   .                 \label{ib564b}
 \end{align}
The line \eref{ib561b} is finite by \eref{ib10b} and is bounded by an $A$ 
dependent multiple of $\nn w\nn_t^2$.  This justifies use of Lemma \ref{lem50}.  

 We need to show now that the integrals in lines \eref{ib562b} through \eref{ib564b}
  are all finite and dominated by an $A$ dependent multiple of $\nn w\nn_t^2$.  
 The sum of these lines is  bounded by
 \begin{align}
 &c_1 \Big(\sup_{0 < s \le t} s^2 \|A'(s)\|_3^2\Big)\int_0^t  s^{-b}\Big(\|w(s)\|_6^2 
                               + 2\kappa^2  ( \|d_Aw\|_2^2 + \|d_A^* w\|_2^2 ) \Big)  ds \label{ib565.1}\\
  & +c_2 \Big(\sup_{0 < s \le t} s \|B(s)\|_2^4\Big)\int_0^t s^{1-b}  \|w'(s)\|_2^2 ds
  + c_3 \int_0^t s^{1-b} \| w'(s)\|_2^2 ds  \label{ib565.2}\\
  &+c_4 t\int_0^t s^{1-b} \| w'(s)\|_2^2 ds
  + c_5 \sup_{0< s \le t}\Big( s^{5/2} \| B'(s)\|_2^2\Big)
                                           \int_0^t s^{-b} \| w(s)\|_6^ 2 ds. \label{ib565.3}
 \end{align}

The three suprema in these three lines are all dominated by a standard bounding
function of $t, \rho_A(t)$, in accordance with Lemma \ref{lemiby}, with $a =1/2$.

    The integral in line \eref{ib565.1}  is at most $5\kappa^2 \nn w\nn_t^2$
     by \eref{ib5a} and \eref{ib6a}. Therefore 
    line \eref{ib565.1} is dominated by an $A$ dependent multiple of $\nn w\nn_t^2$,
    as required for \eref{ib540b}.

    In line \eref{ib565.2}    
   both integrals are appropriately  dominated in accordance with
     the first order estimate \eref{ib10b}. 
   The first term in line \eref{ib565.3} is also dominated in accordance with \eref{ib10b}.
   The second integral in that line is at most      $\kappa^2 \nn w\nn_t^2$. 
    This completes the proof of \eref{ib540b}.

          For the proof of \eref{ib542b} we need only apply the Gaffney-Friedrichs-Sobolev inequality
   \eref{gaf50}, which shows that 
   $\kappa^{-2} \| w'(s)\|_6^2 \le \| d_{A(s)} w'(s)\|_2^2 +\| d_{A(s)}^* w'(s)\|_2^2  
   + (1 + \gamma \|B(s)\|_2^4)  \|w'(s)\|_2^2$. Upon multiplication by $s^{2-b}$ the
    inequality \eref{ib540b} shows that the first two terms are integrable over $(0, t)$.
     The last term is $s(1 + \gamma \|B(s)\|_2^4)$  times  $s^{1-b} \|w'(s)\|_2^2$, which is the product of a bounded factor, in accordance with Lemma \ref{lemiby}
  and an integrable factor, in accordance with  \eref{ib10b}. This proves \eref{ib542b}.

       The  remaining inequalities in the theorem will be derived from  \eref{ib540b} and \eref{ib542b}
        with the help of the GFS inequality and interpolation. We need the following identities. As in \eref{ib549}
       we write $\zeta(s) = d_A^* d_A w(s) + [w(s)\lrc B(s)]$.  And as in \eref{ib10av} we have  
     $ w'(s) = -\zeta(s) - d_A \psi(s)$. 
             Using the identity $d_A^* \zeta(s) = [w(s) \lrc A'(s)]$ from   \eref{pi11a}, and the Bianchi identity,
              we may apply  $d_A^*$ and $d_A$ to this equation to find
              \begin{align}
              -d_A^*d_A \psi &=  d_A^* w'   + [w\lrc A']\ \ \ \ \text{and}      \label{ib591}\\
              -d_A\zeta &= d_A w'  + [B,\psi].  \label{ib592}
              \end{align}               
 We assert that   
 \begin{align}
 \int_0^T s^{2-b}\( \|\, [w(s)\lrc A'(s)]\, \|_2^2    +\|\, [B(s), \psi(s)]\, \|_2^2 \)ds < \infty.    \label{ib593}
 \end{align}  
Indeed
\begin{align} 
\int_0^T &s^{2-b}\(  \|\, [w(s)\lrc A'(s)]\,\|_2^2  + \|\, [B(s), \psi(s)]\,\|_2^2\) ds  \notag\\
&\le c^2\int_0^T s^{2-b}\(  \|A'(s)\|_3^2 \|w(s)\|_6^2 +  \|B(s)\|_3^2 \|\psi(s)\|_6^2\)ds  \notag\\
&\le c^2 \Big(\sup_{0 < s\le T}s^2\|A'(s)\|_3^2 \Big)\int_0^T s^{-b} \| w(s)\|_6^2 ds   \notag\\
&+ c^2\Big(\sup_{0 < s \le T} s \|B(s)\|_3^2\Big) \int_0^T s^{1-b} \| \psi(s)\|_6^2 ds. \notag
\end{align}
The two suprema are finite by \eref{iby7} and \eref{iby6} (with a = 1/2), respectively.
The two integrals are finite by   the finite action assumption (\eref{w3b}  is finite) and    \eref{ib10ba}, respectively. This proves \eref{ib593}.   
But the left hand side of \eref{ib542c} is equal to
\begin{align}
\int_0^T  s^{2-b}\( \|d_A^* w'   + [w\lrc A']\, \|_2^2 + \|\, [w\lrc A']\,\|_2^2 + \| d_A w'  + [B,\psi]\, \|_2^2 \)ds.
\end{align}
It follows from \eref{ib540b} and \eref{ib593} that this integral is finite. This proves \eref{ib542c}.

Concerning the $L^6$ bounds \eref{ib543} - \eref{ib545z} observe first that
   \begin{align}
   &\int_0^t s^{2-b} \lambda(B(s)) \(\| d_{A(s)} \psi(s)\|_2^2 + \|\zeta(s)\|_2^2\) ds  \notag\\
  &\le \Big(\sup_{0 < s \le t} s\lambda(B(s)) \Big) 
                    \int_0^t s^{1-b} \(\| d_{A(s)} \psi(s)\|_2^2 + \|\zeta(s)\|_2^2\)  ds\notag\\
     &< \infty                     
                      \label{ib570}
   \end{align}
   by  \eref{iby1} and \eref{ib10ba}.

The $L^6$ bound \eref{ib545z} for $\zeta$ now follows from the GFS inequality \eref{gaf50} together
 with \eref{ib542c} and \eref{ib570}. The $L^6$ bound \eref{ib543} follows from 
 the GFS inequality \eref{gaf50}, \eref{ib542c} and the additional equality 
    \begin{align}
   \int_0^t s^{2-b} \| d_A d_A \psi(s)\|_2^2 ds &=\int_0^t s^{2-b} \|\, [B(s), \psi(s)]\,\|_2^2 ds,  
     \label{ib568}
   \end{align}
   which is finite by \eref{ib593}.
   
   The inequality \eref{ib545} can be deduced more easily from  what has already been proven
        than from another  application of the Gaffney-Friedrichs-Sobolev inequality.  
        We have $d_A^* d_Aw = \zeta - [w\lrc B]$. Since $\zeta$ satisfies the inequality \eref{ib545z} we need
        only show that $[w\lrc B]$ does also. But
          \begin{align}
        \int_0^t s^{2-b}\| w(s) \lrc B(s) \|_6^2 ds  &\le
        c^2\int_0^t s^{2-b} \| w(s)\|_6^2 \| B(s)\|_\infty^2 ds \\
        &\le   c^2 \Big(\sup_{0 < s \le t} s^2\|B(s)\|_\infty^2 \Big)
        \int_0^t s^{-b} \| w(s)\|_6^2 ds \\
        &<\infty
            \end{align}
        because the supremum is finite by \eref{iby16} and $w$ has finite b-action.
            This completes the proof of \eref{ib545}.

        For the proof of the interpolation inequalities \eref{ib546} choose 
        $f(s,x) = | \psi(s,x)|$ in \eref{ib217} to find
        \beq
        \int_0^T \|\psi(s)\|_q^2 ds \le \Big(\int_0^T s^{-b} \|\psi(s)\|_2^2 ds\Big)^{1-b}
        \Big(\int_0^T s^{1-b} \| \psi(s)\|_6^2 ds \Big)^b.
        \eeq
         The first factor is finite because $w$ has
        finite $b$ action. The second factor is finite by \eref{ib10ba}. This proves \eref{ib546}.
        
        To prove \eref{ib547}  Put $f(s, x) = |d_A\psi(s, x)|$ in \eref{ib215} to find 
        \begin{align*}
        \int_0^T s^{1/2} &\|d_A \psi(s)\|_r^2 ds \\
        & \le \Big(\int_0^T s^{1-b} \|d_A \psi(s)\|_2^2 ds\Big)^{(3/2)-b}
       \Big(\int_0^T s^{2-b} \|d_A \psi(s)\|_6^2 ds\Big)^{b - (1/2)} .
       \end{align*}
       The first factor is finite by \eref{ib10ba}. The second factor is finite by \eref{ib543}.
            
            The same argument applies to the second and third terms in \eref{ib547}
          because each of them satisfy the same $L^2$ initial behavior bounds \eref{ib10b} 
           as $d_A\psi$ and the same $L^6$ initial behavior bounds \eref{ib543}-\eref{ib545z}. 
          
           Concerning the fourth term in \eref{ib547}, observe that with 
        $q^{-1} = r^{-1} - (1/6) = (1/2) - (b/3)$, we have
        \begin{align}
        \int_0^T s^{1/2} \|\, [ A(s),\psi(s)]\, \|_r^2 ds&\le c^2 \int_0^T s^{1/2} \| A(s)\|_6^2 \|\psi(s)\|_q^2 ds \notag\\
        &\le c^2 \(\sup_{0 < s \le T}s^{1/2} \| A(s)\|_6^2\)\int_0^T \|\psi(s)\|_q^2 ds, \notag
        \end{align}
        which is finite by \eref{ibA5a} and \eref{ib546}. Therefore the fourth term in \eref{ib547} 
        differs from the first term  by a finite amount.  This completes the proof of \eref{ib547}.
                
               To prove \eref{ib548} choose $f(s,x) = |d_{A(s)} \psi(s,x)|$ in \eref{ib219}. Then \eref{ib219}
        together with \eref{ib10ba} and \eref{ib543} show that 
        $\int_0^T s^{(3/2) - b} \| d_A\psi(s)\|_3^2 ds < \infty$. 
        The same argument applies to the second and third terms in \eref{ib548}.

               Concerning the fourth term in \eref{ib548},  
 since $\| d\psi(s)\|_3 \le \| d_A \psi(s)\|_3 + \|\, [A(s), \psi(s)]\,\|_3$, it suffices 
 to observe that 
  \begin{align*}
  \int_0^T s^{(3/2) - b} \|\, &[A(s), \psi(s)]\,\|_3^2 ds \le c^2 \int_0^T   s^{(3/2) - b} \|A(s)\|_6^2 \|\psi(s)\|_6^2ds \\
  &\le c^2 \(\sup_{0< s \le T} s^{1/2} \|A(s)\|_6^2\) \int_0^T s^{1-b} \|\psi(s)\|_6^2 ds \\
  &<\infty
  \end{align*}     
  by \eref{ibA5a} and \eref{ib10ba}. This completes the proof of \eref{ib548}.  
  
  In regard to \eref{ib715},  choose  $f(s) = d_{A(s)} \psi(s)$ in \eref{ib218} to find that for $0 < b \le 1$
  and $\rho \ge 2$  we have 
\begin{align}
\int_0^\tau \|  d_{A(s)} \psi(s)\|_\rho ds < \infty\ \ \text{if $1/\rho > (1/2) - (b/3)$},   \label{rec715a}
\end{align}
 in view of \eref{ib10ba} and \eref{ib543}. In particular if  $ 1/2 \le b <1$ then the restriction on $\rho$
 is satisfied if $ 2 \le \rho < 3$. This proves \eref{ib715}. 
        \end{proof}

\subsection{High $L^p$ bounds for $\psi$}           \label{sechighp}

Our energy methods typically establish $L^p$ bounds on functions for $ 2 \le p \le 6$.
For larger values of $p$ we will use the following Neumann domination method.

\begin{theorem}\label{thmpw1}  
 Let  $1/2 \le a <1$ and $0 < b <1$.  
Suppose that $A(\cdot)$ is a strong solution to the Yang-Mills
 heat equation over $M$ with finite a-action and satisfying the Neumann boundary conditions
 $A(t)_{norm}=0$, $B(t)_{norm} =0$  for $t >0$  in case $M \ne \R^3$.   
 Let $w$    be a strong solution
 to the augmented variational equation \eref{av1}  with  finite b-action in the sense of \eref{w3d}  
 and satisfying Neumann 
 boundary conditions  \eref{vst151n}    
 in case $M \ne \R^3$. 
 Define $\psi(t) = d_A^* w(t)$ for $t >0$  as in \eref{ib10psi}. Then there are   standard dominating 
  functions $C_p$ such that
\begin{align}
\int_0^T t^{(3/2) - b -(3/p)} \|\psi(t)\|_p^2 dt 
 \le C_p(T, \rho_A(T) )  \nn w\nn_T^2     <\infty,  \ \ \ 6 \le p \le  \infty  \label{ib200p}
\end{align}
for all $T < \infty$.
\end{theorem}

\begin{remark}{\rm Dirichlet boundary conditions are noticeably absent in the allowed hypothesis.
This arises from the failure of Lemma \ref{lemnda2} in this case.
}
\end{remark}

The proof depends on the following five lemmas. The first lemma is taken from \cite{G70}.

     \begin{lemma} \label{lemnda1}$($Neumann domination with averaging$)$ Let $0 < T < \infty$.
Suppose that $M\subseteq \R^N$ is the closure of an open set with smooth boundary and that
  $A:(0, T] \rightarrow C^1(M; \L^1\otimes \kf)$ is a time dependent  1- form on $M$ which is continuous in the time variable.
 Let $0 \le p \le n$. Let $\w: (0,T) \rightarrow C^2( M;\L^p \otimes \kf)$ be a time dependent, $\kf$ valued, p-form on $M$ which is continuously differentiable in the time variable and satisfies the equation
 \begin{align}
 \w'(s,x)  =\sum_{j=1}^N(\n_j^{A(s)} )^2 \w(s,x) + h(s, x),     \label{nda5.1}
 \end{align}
 where $h \in C((0, T]\times M; \L^p \otimes \kf)$. Assume also that if $M \ne \R^N$ then
 \begin{align}
 \n_{n} |\w(s,x)|^2 \le 0,\ \ \ \ \ 0< s <T,  \ \ \  x\in \p M,              \label{nda6.1}
 \end{align}
 where   $n$ is the outward drawn unit normal.
 Denote by $\Delta_N$ the Laplacian on real valued functions over $\R^N$ if $M = \R^N$
  or the Neumann Laplacian  on real valued functions if  $M \ne \R^N$.
 Then 
 \begin{align}
 |\w(t,x)| &\le t^{-1} \int_0^t e^{(t-s)\Delta_N} |\w(s, \cdot)| ds\ (x)    \notag\\
 &\qquad\qquad +  t^{-1}\int_0^t  e^{(t-s)\Delta_N} s|h(s, \cdot)| ds\ (x) .         \label{nda7.1}
 \end{align} 
 \end{lemma}
    \begin{proof} See \cite[Proposition 4.21]{G70}. 
    \end{proof}

 \begin{lemma}\label{lemnda2} $($Normal derivative$)$  Suppose that $M \ne \R^3$
 and  that $A(\cdot)$ is
  a strong solution to the Yang-Mills
 heat equation over $(0, \infty)$  
 satisfying Neumann boundary conditions.  
 Let $w$  be a strong solution  to \eref{av1} satisfying Neumann  
 boundary conditions  \eref{vst151n}.  
 Define  $\psi(s) = d_{A(s)}^* w(s)$ for $s >0$
 as in \eref{ib10psi}.
  Then
 \beq
 \n_n |\psi(s, x)|^2 = 0\ \ \  \text{for all}\ \ s>0  
  \ \ \text{and}\ \ x \in \p M.    \label{nda15}
 \eeq
 \end{lemma}
        \begin{proof} We are assuming now that $A(t)_{norm}=0$ and $B(t)_{norm} =0$ for all $t >0$
        and that  
  \begin{align}
 w(t)_{norm} &= 0\ \ \text{on}\  \p M\ \  \text{for all}\ \ t>0 \ \text{and}  \label{nda16} \\
 (d_{A(t)}w(t))_{norm} &= 0 \ \ \text{on}\  \p M\ \  \text{for all}\ \ t>0.  \label{nda17}
 \end{align}
 We will show  that these four conditions imply that the solution $w(\cdot)$ to the augmented variational equation
 \eref{av1} satisfies
 \begin{align}
 (d_{A(t)}\psi(t))_{norm} = 0\ \ \ \text{for all} \ \ \  t >0.          \label{nda18}
 \end{align}
 We may write \eref{av1} as   $d_A\psi(t) = -\Big\{ w'(t)  +d_A^* d_A w + [ w\lrc B]\Big\}$. 
  It suffices to show that each of the three terms in braces has normal component zero.
  Differentiating \eref{nda16} with respect to $t$ we see that  $  w'(t)_{norm} = 0$. 
 Since $B(t)_{norm} =0$, we also have  $[ w(t) \lrc B(t)]_{norm} =0$ (for any $w$.)    
  It follows from \eref{nda17} and  \cite[ Equ. (3.20)]{CG1} that  $(d_A^*d_A w(t))_{norm} =0$.
 This proves \eref{nda18}.
        
        It follows now from \eref{nda18} that 
 \begin{align*}
 \n_n | \psi(t, x)|^2 &= \<(d_A \psi(t, x))_{norm}, \psi(t, x)\> + \< \psi(t, x), (d_A \psi(t, x))_{norm}\>\\
 &=0.
 \end{align*}    
 \end{proof}

 \begin{lemma} \label{lemnda3}$($Pointwise bound$)$ If $M=\R^3$, or if $M\ne \R^3$ but
 Neumann boundary conditions hold for $A$ and $w$, then
 \begin{align}
 |\psi(t, x)|&\le  t^{-1} \int_0^t e^{(t-s)\Delta_N} |\psi(s, \cdot)| ds\ (x)   \notag\\
 & + t^{-1}\int_0^t e^{(t-s)\Delta_N} 2 s|\, [A'(s) \lrc w(s)]\, | ds \ (x) .       \label{nda20}
 \end{align}
 \end{lemma}
 \begin{proof} 
 Take $\w(s) = \psi(s)$ in \eref{nda5.1}. Then \eref{pi12a} shows that \eref{nda5.1} holds with  
  $h(s) =2[A'(s)\lrc w(s)]$.
  Lemma \ref{lemnda2}
  shows that  the condition \eref{nda6.1} holds when $\w = \psi$. 
  We may therefore apply Lemma \ref{lemnda1}
   to see that \eref{nda20} holds. 
 \end{proof}

 \begin{lemma} \label{lemnda4} $($Weighted estimate$)$  
 There is a standard dominating function  $C_{50}$ 
   such that     
 \begin{align}
\int_0^T s^{(5/2)- a -b} &\|\, [A'(s) \lrc w(s)]\, \|_2^2 
 \le  C_{50}(T, \rho_A(T)) \nn w\nn_T^2  \ \  \forall\ T \ge 0         .     \label{nda25}
\end{align}
\end{lemma} 
         \begin{proof}
\begin{align*}
\int_0^T s^{(5/2)- a -b} &\|\, [A'(s) \lrc w(s)]\, \|_2^2 ds
 \le c^2 \int_0^T s^{(5/2)-a -b} \|A'(s)\|_3^2 \| w(s)\|_6^2 ds \\
 &\le c^2 \sup_{0 < s \le T}\Big(s^{(5/2)-a} \|A'(s)\|_3^2\Big) \int_0^T s^{-b}  \| w(s)\|_6^2 ds.
\end{align*} 
The supremum is finite  by \eref{iby7}.
The last integral is at most $\kappa_6^2\, \nn w\nn_T^2$ by \eref{ib6a}.
\end{proof}

\begin{lemma}\label{lemu1}$($A convolution inequality$)$
          Let $0 \le c <1$ and $0 < T < \infty$.  Suppose that  $\alpha$ and $\beta$ are
          non-negative functions  on $(0, T]$ such that
\begin{align}
\alpha(t) \le (1/t)&\int_0^t (t-s)^{-c}  \beta(s) ds\ \ \ \text{for}\ \ \ 0 < t \le T .  \label{u31}
\end{align}
Then for any real number $b_1 < 2c+1$ there holds
\begin{align}
\int_0^T t^{b_1} \alpha(t)^2 dt \le \gamma \int_0^T s^{b_1 -2c} \beta(s)^2 ds                 \label{u32}
\end{align}
for some constant $\gamma$ depending only on $b_1$ and $c$.
\end{lemma}
     \begin{proof} This is \cite[Lemma 4.24]{G70}. 
     \end{proof}

\bigskip
\noindent
     \begin{proof}[Proof of Theorem \ref{thmpw1}]
     From \eref{nda20} we find 
     \begin{align}
     \|\psi(t)\|_p &\le t^{-1}\int_0^t \|e^{(t-s) \Delta_N} \Big( |\psi(s, \cdot)| 
                       + 2s |\, [A'(s, \cdot) \lrc w(s, \cdot)]\, | \Big) \|_p ds      \notag\\
  & \le   t^{-1}\int_0^t \|e^{(t-s) \Delta_N}\|_{2\rightarrow p} \Big( \|\psi(s)\|_2 
                       + 2s \| \, [A'(s) \lrc w(s)]\, \|_2 \Big) ds    \notag\\
  & \le   t^{-1}\int_0^t c_1(t-s)^{-(3/4) + (3/2p)} \Big( \|\psi(s)\|_2 
                       + 2s \| \, [A'(s) \lrc w(s)]\, \|_2 \Big) ds.  \label{nda30}
     \end{align}
 In Lemma \ref{lemu1}    choose  $c = (3/4) -(3/2p)$ and choose 
     $b_1 = (3/2)-b -(3/p)$. Then $b_1 - 2c =  -b <0$. Take $\alpha(t) = \| \psi(t)\|_p$ and
     $\beta(s) = c_1\Big( \|\psi(s)\|_2  + 2s \| \, [A'(s) \lrc w(s)]\, \|_2 \Big)$ in Lemma \ref{lemu1}. 
     Then  \eref{u32}  and\eref{nda30}       show that
     \begin{align}
     \int_0^T t^{(3/2) - b - (3/p)} \|\psi(t)\|_p^2 dt \le \gamma \int_0^T s^{-b} \beta(s)^2 ds.
     \end{align}
     But, by \eref{ib5a},  $\int_0^T s^{-b}\| \psi(s)\|_2^2 ds \le 4\nn w \nn_T^2 < \infty$ because $w$ has finite b-action. 
     Moreover, since $A$ has finite a-action and therefore finite (1/2) action, we can
      put $a = 1/2$ in \eref{nda25} to conclude that
     $\int_0^T s^{-b} \beta(s)^2 ds < \infty$. This completes the proof of \eref{ib200p}.
\end{proof}

\section{Recovery of $v$ from $w$: the differential equation}     \label{secrec}

\subsection{$v$ and $v_\tau$ satisfy the variational equation}

We intend to recover a solution to the variational equation \eref{ve}  from a  solution
 $w$ to the augmented variational equation \eref{av1}. 
 The two  functions 
 $v(t)$, given in   \eref{rec1} and $v_\tau(t)$ given in \eref{rec2} will be shown to be solutions
  to the variational equation.  
 They differ just by the choice of the lower limit $\tau$ in \eref{rec2}. The behavior   
   of these two functions differ considerably, even for fixed $t >0$,   
   because in \eref{rec1} the integrand comes close
   to the singular point at $s = 0$ whereas in \eref{rec2} it    does not.

Let $\psi(s) = d_{A(s)}^* w(s)$ again, as in \eref{ib10psi}, 
and let $\tau \ge 0$. 
Define  
\begin{align}
v_{\tau}(t) := w(t) + d_{A(t)}\int_\tau^t \psi(s)ds,  \ \ \ 0 < t <  \infty,            \label{rec23} 
\end{align}
and let
 \begin{align}
 \alpha =\int_0^\tau \psi(s) ds, \ \  \tau > 0. 
  \label{rec26}
 \end{align}
 Clearly   
\beq
v_0(t) = v_\tau(t) + d_{A(t)} \alpha\ \ \text{for }\ \ \ \tau >0.     
 \label{rec27} 
\eeq
 The solution $v(t)$ defined in \eref{rec1} coincides with $v_0(t)$. We will use the notation $v_0(t)$ in this 
 section so as to be able to treat the cases $\tau=0$ and $\tau >0$ as simultaneously as possible.
 We are going to show that      
 
 \noindent
\ \ \ \ a)  $v_\tau(t)$ is an almost strong solution if $\tau =0$. Moreover it has the correct initial value $v_0$, 
          
\noindent
\ \ \ \,b) $v_\tau(t)$ is a strong solution if $\tau >0$.  It differs from the almost strong solution $v_0(\cdot)$
 by the  vertical solution $d_{A(t)} \alpha$. 
 We will see in Section \ref{secinit}  
 that Lemma \ref{lemvert1} is applicable.

The second term in \eref{rec23}  is clearly vertical at $A(t)$ for each $t \in (0,  \infty)$.  
It is the vertical correction to $w$ needed to convert the solution, $w$,  of \eref{av1} to a solution of \eref{ve}.
But   the second term  
 is not a vertical solution itself because the scalar factor $\int_\tau^t \psi(s) ds$ 
 is not independent of $t$.

In this subsection we are going to show,  at an algebraic level, that 
$v_\tau(\cdot)$ satisfies the variational equation over $(0,  \infty)$ for all $\tau \in [0,\infty)$.
In Section \ref{secsas} we will show that 
$v_\tau(\cdot)$ is an almost strong solution if $\tau =0$ and a strong solution if $\tau  >0$.
In Section \ref{secinit} we will show that both solutions converge to their correct initial values in the $L^2$
sense.
In Section \ref{secivas}  
 we will show that $v_0(t)$ converges to 
its initial value $v_0$ in  the sense of $H_b^{\As}$, as asserted in  Theorem \ref{thmveu1e}.

\begin{theorem} \label{thmrec5}  
 $($$v$ and $v_\tau$ solve the variational equation$)$ 
Suppose that $w(s)$ is a solution to the augmented variational equation \eref{av1} on $(0,  \infty)$.
Let $\psi(s) = d_{A(s)}^* w(s)$ again  as in \eref{ib10psi}. 
Fix $\tau \ge 0$ and define  
$v_\tau$  by \eref{rec23}. Then
\beq
-v_\tau'(t) = d_{A(t)}^* d_{A(t)} v_\tau(t) + [ v_\tau(t)\lrc B(t)], \ \ \ 0 < t <  \infty. 
         \label{rec711}
\eeq
Let $\zeta(s) = d_{A(s)}^* d_{A(s)} w(s) +[w(s)\lrc B(s)]$ as in \eref{ib549}. 
Then  $v_\tau$ can also be represented as  
\beq
v_\tau(t) = w(\tau) - \int_\tau^t \Big( \zeta(s) +[A(s) -A(t), \psi(s)] \Big) ds, \ \ \   \ 0 < t <  \infty.   \label{rec750}
\eeq
\end{theorem} 
      \begin{proof}    
   Let 
   \beq
   \eta(t) = \int_{\tau}^t \psi(s) ds, \ \ \ 0 < t <  \infty. \label{rec758}  
   \eeq	
   Then      
   $v_\tau(t) = w(t) + d_{A(t)}\eta(t)$  by \eref{rec23}. 
   We have   
 \begin{align*}
 \frac{d}{dt} d_{A(t)}\eta(t)  &= [A', \eta] +d_A \eta' \\
 &=  - [d_A^* B, \eta]  + d_Ad_A^* w
 \end{align*}
 because $A' = - d_A^*B$ by the Yang-Mills heat equation.
 Hence, suppressing $t$ in places, we find 
 \begin{align}
 -v_\tau'(t) &= - w'(t) - \frac{d}{dt} d_{A(t)}\eta(t)                                              \notag\\
 &= \Big\{ (d_A^*d_A + d_A d_A^*)w + [w\lrc B]\Big\} - d_A d_A^* w + [d_A^* B, \eta]\notag\\
 & = d_A^*d_A  w +[w\lrc B] +[d_A^* B, \eta]       \notag\\
 &= d_A^*d_A(v_\tau  - d_A\eta)  + [(v_\tau-d_A \eta)\lrc B] + [d_A^* B, \eta]    \notag\\
 & = d_A^* d_A v_\tau +[v_\tau \lrc B] + \Big\{ -d_A^*(d_A)^2\eta - [d_A\eta \lrc B] + [d_A^* B, \eta]  \Big\}.
                                                                                                 \label{rec759}
 \end{align}
 By the Bianchi identity we have 
 $d_A^*(d_A)^2 \eta = d_A^*[B,\eta]
  =[d_A^* B, \eta] - [d_A \eta\lrc B]$.
 The expression in braces  in line \eref{rec759} is  therefore zero.
 This proves \eref{rec711}. 

From \eref{ib10av} we see that   $ d_A \psi = -\{ w' + \zeta\}$.
 Hence 
 \begin{align}
 d\int_\tau^t  \psi(s) ds &=\int_\tau^t \Big(d_{A(s)}\psi(s)   -[A(s), \psi(s)]\Big) ds  \notag\\
&= -\int_\tau^t \Big( \Big\{w'(s) + \zeta(s) \Big\} +[A(s), \psi(s)]\)ds \notag\\
 &=w(\tau) - w(t) -\int_\tau^t \(\zeta(s) +[A(s), \psi(s)]\) ds.    \notag 
 \end{align}
 Thus
 \begin{align}
 d\int_\tau^t  \psi(s) ds & =w(\tau) - w(t) -\int_\tau^t \(\zeta(s) +[A(s), \psi(s)]\) ds\ \ \ \ \ \text{and} \label{rec767}  \\
 d_{A(t)}\int_\tau^t  \psi(s) ds  &=w(\tau) - w(t) -\int_\tau^t \(\zeta(s) +[A(s)-A(t), \psi(s)]\) ds. \label{rec766}
 \end{align}
 Add $w(t)$ to  \eref{rec766} to find \eref{rec750}, given the definition \eref{rec23}. 
This completes the proof of the theorem.
\end{proof}

\begin{remark}{\rm The two representations of $v_\tau(t)$ given in \eref{rec23} and \eref{rec750} differ
in the following crucial way. The second derivatives of $w(s)$ in the integrand in \eref{rec750}
are $d_{A(s)}^* d_{A(s)} w(s)$ whereas the second derivatives of $w$ that appear in \eref{rec23}
(after moving $d_{A(t)}$ under the integrand and shifting time parameter to $s$) are 
$d_{A(s)} d_{A(s)}^* w(s)$. In order to compute $H_1$ norms we will have to use the Gaffney-Friedrichs
inequality, which requires computing exterior derivatives of $v_\tau(t)$ and its coderivatives. 
For computing exterior derivatives
the representation \eref{rec23} will allow us to use the Bianchi identity, while \eref{rec750} will allow
us to compute coderivatives using the adjoint Bianchi identity. 
}
\end{remark}

\subsection{Strong solutions vs. almost strong solutions} \label{secsas}        
  \begin{theorem}\label{thmrec10}     
 Let $ 0 < b <1$.  Choose $T \in (0, \infty)$ and let $\As = A(T)$. Suppose that $w$  is the solution to the augmented variational equation \eref{av1}
 constructed in Theorem \ref{thmwe}.   
 Define $v_\tau(t)$ by \eref{rec23} again.
   Then   
   \begin{align}
   v_\tau(t) &\in  H_1^\As \ \ \ \forall\  t \in (0,  \infty)\ \ \ \text{if}\ \ \tau >0.  \label{rec777}\\
   d_{A(t)} v_\tau(t) &\in H_1^\As \ \ \ \forall\  t \in (0, \infty)\ \ \ \text{if}\ \ \tau \ge 0.   \label{rec776}
   \end{align}
   In particular, $v_\tau(t)$ is a strong solution to the variational equation 
   over $(0,\infty)$ if $\tau >0$ and is
   an almost strong solution if $\tau =0$.  
 \end{theorem}

The proof depends on the following  lemma.

\begin{lemma}\label{lemrec5} 
 Assume that $0 < b <1$ and $0 \le t < \infty$   
 Define again $\eta(t)$  by \eref{rec758}. Then
\begin{align}
\|\eta(t)\|_2 &< \infty\ \ \forall\  \tau \ge 0  \ \text{and}\ \forall\  t\ge 0                    \label{rec790} \\
\| d_{A(t)} \eta(t)\|_2 &< \infty\ \ \forall\  \tau \ge 0 \ \ \text{if}\  \  t >0                \label{rec791}\\
\|d_\As  d_{A(t)} \eta(t)\|_2 &< \infty \ \ \forall \  \tau \ge 0 \ \ \text{if}\ \  t >0       \label{rec792}\\
\|\, [ B(t), \eta(t)]\, \|_2 & < \infty       \ \ \forall \  \tau \ge 0 \ \ \text{if}\  \ t >0      \label{rec795.0}\\
\|d_\As [ B(t), \eta(t)] \ \|_2   &< \infty \ \ \forall \   \tau \ge 0  \ \ \text{if}\  \ t >0    \label{rec795} \\
\|d_\As^* [ B(t), \eta(t)] \ \|_2  & < \infty \ \ \forall\   \tau \ge 0  \ \ \text{if}\  \ t >0   \label{rec796} \\
\| d_\As^*\int_\tau^t \zeta(s) ds \|_2 &< \infty\ \ \forall\  \tau > 0  \ \ \text{if}\  \ t >0   \label{rec793}\\
\| d_\As^*\int_\tau^t [A(t) - A(s), \psi(s)]ds \|_2 &< \infty\ \ \forall\  \tau > 0 \ \ \text{if}\  \ t >0 \label{rec794}
\end{align}
\end{lemma} 
                \begin{proof} We will write integrals over $[\tau, t]$ or $[t, \tau]$ as if $ t \ge \tau$ with
 no loss of generality. Let $t_1 = max(t, \tau)$.  The proof of \eref{rec790} follows from the inequalities 
 \begin{align}
 \|\eta(t)\|_2 &\le \int_0^{t_1}\|\psi(s)\|_2 ds  \le \Big(\int_0^{t_1} s^{b} ds \Big)^{1/2}\Big(\int_0^{t_1} s^{-b}
      \|\psi(s)\|_2^2 ds\Big)^{1/2}   \notag\\
      &\le t_1^{(b+1)/2} \(\int_0^{t_1}s^{-b} \|d_{A(s)}^* w(s)\|_2^2 ds\)^{1/2} 
                             \le t_1^{(b+1)/2}2\nn w\nn_{t_1} <\infty                                 \notag
 \end{align} 
  for all  $t\ge0$ and $ \tau \ge 0$ by \eref{ib5a}. 
To prove \eref{rec791} and the remaining inequalities we  take $t >0$. Then     
     \begin{align*}
&\| d_{A(t)} \eta(t)\|_2 \le \int_\tau^t \|d_{A(t)} \psi(s)\|_2ds \\
&\le \int_\tau^t \| d_{A(s)}\psi(s)\|_2 ds  + \int_\tau^t \|A(t) - A(s)\|_6 \|\psi(s)\|_3 ds. \\
&\le \int_0^{t_1} \| d_{A(s)}\psi(s)\|_2 ds +\(\int_\tau^t s^{-a}\|A(t) - A(s)\|_6^2 ds\)^{1/2} 
                   \(\int_0^{t_1} s^a\|\psi(s)\|_3^2 ds\)^{1/2}  \\
                   &<\infty\ \    \text{for all} \ \ t>0\  \text{and}\ \tau \ge 0\  \text{if}\ 0 < b <1\ 
\end{align*}
 because the first term is finite by \eref{ib10ba} (since $b >0$), while the second term
 is a product  of an integral over $[\tau, t]$  (or $[t, \tau]$), which is finite when $\tau >0$ because it  
  excludes a neighborhood of $ s =0$,  and is finite for $\tau =0$ by 
\eref{ibA6}.   The second factor is finite by \eref{ib10c} because $ a \ge 1/2 \ge (1/2) - b$ for all $b \in [0, 1)$.

{\bf Proof of \eref{rec792}. } 
For the proof of \eref{rec792} we have
\begin{align*}
d_\As d_{A(t)} \eta(t) &= d_{A(t)}d_{A(t)}\eta(t) +[(A(T) - A(t))\wedge   d_{A(t)} \eta(t)] \\
&= [B(t), \eta(t)] + [(A(T) - A(t))\wedge   d_{A(t)} \eta(t)]. 
\end{align*}
Therefore
\begin{align}
\|d_\As d_{A(t)} \eta(t)\|_2 \le c \|B(t)\|_\infty \|\eta(t)\|_2 + c \|A(T) -A(t)\|_\infty \|  d_{A(t)} \eta(t)\|_2. \notag
\end{align}
The two  $L^\infty$ 
 norms are finite by  \eref{iby16} and \eref{ibA72} 
 respectively.  \eref{rec792} now follows from \eref{rec790} and \eref{rec791}.

        {\bf Proof of  \eref{rec795.0}, \eref{rec795} and \eref{rec796}.}  
    Since  $\| B(t)\|_\infty < \infty$ by \eref{iby16},  the inequality \eref{rec795.0} follows from
    the inequality $\| [ B(t), \eta(t)]\, \|_2 \le c \|B(t)\|_\infty \| \eta(t)\|_2$ and from \eref{rec790}.
        The identities 
\begin{align}
d_\As &[B(t), \eta(t)]  =d_{A(t)} [B(t), \eta(t)] + [(A(T) - A(t)) \wedge [B(t), \eta(t)]]    \notag\\
 &= -[d_{A(t)} \eta(t) \lrc B(t)] + [(A(T) - A(t)) \wedge [B(t), \eta(t)]],                         \label{rec795.1} \\
d_\As^* &[B(t), \eta(t)] = d_{A(t)}^* [B(t), \eta(t)] +[(A(T) - A(t)) \lrc [B(t), \eta(t)] ]  \notag\\
&= -[A'(t), \eta(t)] - [d_{A(t)} \eta(t) \lrc B(t)]  +     [(A(T) - A(t)) \lrc [B(t), \eta(t)] ]         \label{rec796.1}
\end{align}
are similar  and have similar bounds. Thus
\begin{align}
\|\, [d_{A(t)} \eta(t) \lrc B(t)]\, \|_2  \le c \|B(t)\|_\infty \|d_{A(t)} \eta(t)\|_2,  \notag
\end{align}
which is finite by \eref{iby16} and \eref{rec791}. Moreover 
\begin{align}
\|A'(t)\|_\infty < \infty\ \ \ \text{and}\ \ \ \|  A(T) - A(t) \|_\infty < \infty     \notag
\end{align}
by  \eref{iby15}   and \eref{ibA72}. Therefore the remaining three terms in the lines \eref{rec795.1}
and \eref{rec796.1} have finite $L^2$ norms by \eref{rec790} and \eref{rec795.0}.

            {\bf Proof of \eref{rec793}.}  
To prove \eref{rec793} we may write
       \begin{align}
&\| d_\As^*\int_\tau^t \zeta(s) ds \|_2 
      =\| \int_\tau^t \(d_{A(s)}^*\zeta(s)  +[(A(T)- A(s))\lrc \zeta(s)]\)ds \|_2 \notag\\
&\le \int_\tau^t \|d_{A(s)}^*\zeta(s) \|_2ds  +c \int_\tau^t \| A(T) - A(s)\|_6 \|\zeta(s)\|_3    ds  \notag\\
&\le  \int_\tau^t \|d_{A(s)}^*\zeta(s) \|_2ds        \label{rec798} \\
   &+ c\(\int_0^{t_1} s^{-a}\| A(T) - A(s)\|_6^2 ds \)^{1/2} \(\int_\tau^t s^a\|\zeta(s)\|_3^2 ds \)^{1/2}  \label{rec799}
\end{align}
The  integral over $(0,t_1]$ is finite by  \eref{ibA6}. 
Since the interval $[\tau, t]$ is bounded away from zero the integrals over $[\tau, t]$ in lines \eref{rec798}
and \eref{rec799} are finite by  \eref{ib542c} and \eref{ib548}, respectively. This proves \eref{rec793}.

{\bf Proof of \eref{rec794}.}  
We have the identities 
\begin{align}
d_\As^*&\int_\tau^t [A(t) - A(s), \psi(s)]ds  =\int_\tau^t  d_\As^* [A(t) - A(s), \psi(s)]ds  \notag\\
&=\int_\tau^t \([  d_\As^*(A(t) - A(s)), \psi(s)] +[(A(t) - A(s))\lrc d_\As \psi(s)] \) ds. \notag 
\end{align}
Therefore
\begin{align}
\|d_\As^*&\int_\tau^t [A(t) - A(s), \psi(s)]ds\|_2 \le c\int_\tau^t   \(\| d_\As^*(A(t) - A(s))\|_3 \|\psi(s)\|_6 \notag\\
&\ \ \ \ \ +\|A(t) - A(s)\|_3 \| d_\As \psi(s)\|_6 \) ds \notag\\
&\le c \(\int_\tau^t s^{b-1} \| d_\As^*(A(t) - A(s))\|_3^2 ds\)^{1/2} 
                                \(\int_0^{t_1} s^{1-b} \|\psi(s)\|_6^2 ds \)^{1/2}\notag\\
&\ \ \ \ \  + c\(\int_\tau^t s^{b-2}\|A(t) - A(s)\|_3^2ds\)^{1/2}  \(\int_\tau^t s^{2-b} \| d_\As \psi(s)\|_6^2 ds\)^{1/2} \notag
\end{align}
 The $A$ integrals are finite because the interval $[\tau, t]$ excludes a neighborhood of zero.
The first $\psi$ integral is  finite by \eref{ib10ba}. The second $\psi$ integral is finite because it differs
from \eref{ib543}  by at most $\int_\tau^t s^{2-b}\|A(T) - A(s)\|_\infty^2 \| \psi(s)\|_6^2 ds$, while   
 $\|A(T) - A(s)\|_\infty$ is bounded over $[\tau, t]$ in accordance with \eref{ibA72}.
\end{proof}

\bigskip
\noindent
\begin{proof}[Proof of Theorem \ref{thmrec10}] Let $\tau >0$.
In order to prove that $v_\tau(t) \in H_1^\As$ for $t >0$ it suffices, by the Gaffney-Friedrichs inequality,
 to show that $v_\tau(t), d_\As v_\tau(t)$ and $d_\As^* v_\tau(t)$ are all in $L^2(M)$ for each $t >0$. Now 
 \eref{rec23} and \eref{rec750} show that  
\begin{align}
 d_\As v_\tau(t) &=  d_\As w(t) + d_\As d_{A(t)} \int_\tau^t \psi(s) ds,              \label{rec784}\\
 d_\As^* v_\tau(t) &=d_\As^* w(\tau) - d_\As^*\int_\tau^t \(\zeta(s) +[A(s)- A(t),\psi(s)] \) ds. \label{rec785}
 \end{align}
 Since   $w(t)$ and $w(\tau)$ are both in $H_1^{\As}$ for $t > 0$  and $\tau > 0$ we need only
  address the second term in each line. But 
   the second term in line \eref{rec784} is in $L^2(M)$ by \eref{rec792} and
  the second term in line \eref{rec785} is in $L^2(M)$ by \eref{rec793} and \eref{rec794}.
  $v_\tau(t)$ itself is in $L^2(M)$ by  \eref{rec23} and \eref{rec791}. Hence $v_\tau(t) \in H_1^\As$
   for each $t >0$ when  $\tau >0$. In order to show that it is  strong solution to the 
    variational    equation  \eref{ve}, we only need to show    
   that $d_{A(t)}v_\tau(t)$ is in $H_1^\As$ for each $t >0$ since Theorem \ref{thmrec5} already shows that
   it satisfies the variational equation informally. From \eref{rec23} and the Bianchi identity we see that
   \begin{align}
   d_{A(t)} v_\tau(t) = d_{A(t)} w(t) + [B(t), \eta(t)].
   \end{align}
   The first term is in $H_1^\As$  because $w$ is a strong solution by Theorem \ref{thmwe}. 
    The second term is in $H_1^\As$ by the Gaffney-Friedrichs inequality
    in view of \eref{rec795.0}, \eref{rec795} and \eref{rec796}.
          Therefore $v_\tau(\cdot)$ is a strong solution when $\tau >0$. But the last 
   argument shows that $d_{A(t)} v_\tau(t) \in H_1^{\As}$ for any $\tau \ge 0$ because 
     \eref{rec795.0},      \eref{rec795} and \eref{rec796}  all hold for  any $\tau \ge 0$. 
           Therefore $v_\tau(t)$ is an almost strong solution    even for $\tau =0$.
 This completes the proof of Theorem \ref{thmrec10}.  
\end{proof}

\section{Recovery of $v$ from $w$: initial value}  \label{secinitv}

 We have shown in Section  \ref{secrec}  that the functions $v$ and $v_\tau$,
defined in \eref{rec1} and \eref{rec2}, are respectively almost strong
and strong solutions to the variational equation over $(0, \infty)$. In Section \ref{secinit} 
we will show  that both take on their correct initial values in the sense
 of $L^\rho(M; \L^1\otimes \kf)$ convergence for $2\le \rho <3$. We will show in Section \ref{secivas} that 
 the almost strong solution attains its initial value in the stronger
  sense of $H_b^\As$ convergence.     And in Section
 \ref{secrecab} we will show that the strong solution has finite b-action. For the latter two results we
 will have to use some non-gauge invariant techniques along with the non-gauge invariant 
 hypothesis that $\| A(s)\|_3 < \infty$ for some $s >0$.

\subsection{Initial values in the $L^\rho$ sense}   \label{secinit}

It will be convenient to write \eref{rec3} in the form
 \begin{align}
 \alpha_\tau = \int_0^\tau \psi(s) ds, \label{rec3b}
 \end{align}
 where $\psi(s) = d_{A(s)}^*w(s)$ since we will make extensive use of the initial behavior 
  of $\psi(s)$ and some of its derivatives that has been established in Section \ref{secibw}.

       \begin{theorem} \label{strinit} Assume that $1/2 \le a < 1$ and $1/2 \le b < 1$. 
Suppose that $ 2 \le \rho < 3$ and let $ \tau >0$. 
 Denote by $w$ the strong solution to the augmented variational equation constructed
  in Theorem \ref{thmwe}. Define $v(t)$ and $v_\tau(t)$ by \eref{rec1} and \eref{rec2} respectively. Then 
\begin{align}
\| v_\tau(t) - \(v_0 - d_{A_0} \alpha_\tau\)\|_\rho \rightarrow  &0
                                     \ \ \ \text{as}\ \ \ t \downarrow 0 \ \ \ \text{and} \label{rec976} \\
 \sup_{0\le t \le 1} \| v_\tau(t) - v(t)\|_\rho \rightarrow &0 \ \ \ \text{as}\ \ \ \tau  \downarrow 0. \label{rec977}
\end{align}
\end{theorem}

Since $v_\tau$ is a strong solution to the variational equation, by Theorem \ref{thmrec10}, 
it is continuous on $(0, \infty)$ into $H_1^\As$ and therefore also into $H_b^\As$ for $0 \le b \le 1$. 
Theorem \ref{strinit} is properly concerned, therefore, only with the behavior of $v_\tau$ at $t =0$.

\begin{lemma}\label{lemvert6}   $($Continuity of the vertical correction$)$ 
 Let $1/2 \le a < 1$ and $1/2\le b<1$. 
 Let $0 < T < \infty$ and let $\As = A(T)$.
 Suppose that $2 \le \rho <3$.  
 Then  
 \begin{align}
 &\ \alpha_\tau \in  L^p(M; \kf)\ \ \text{for}\ \ 2 \le p <\infty,                     \label{rec713d} \\
 &\ \alpha_\tau \in H_1^\As(M; \kf),      \label{rec713e} \\
 &\sup_{0 < t \le 1} \|d_{A(t)}\alpha_\tau \|_\rho \rightarrow 0\ \ 
                \text{as}\ \ \tau \downarrow 0.                                             \label{rec713c} \\
  &\ t,\tau \mapsto d_{A(t)} \alpha_\tau\in  L^\rho(M; \L^1\otimes \kf)  \ \ 
                     \text{is continuous on}\  [0,\infty)^2.                     \label{rec713ab} 
 \end{align}
 \end{lemma} 
            \begin{proof}             
 From  \eref{ib715} we see that for $1/2 \le b <1$ we have 
    \begin{align}
   \int_0^\tau \|d_{A(s)} \psi(s)\|_\rho ds  
           &< \infty  \ \ \text{if}\ \ 2\le \rho<3\ \ \ \text{and}  \label{rec715b} \\
   \int_0^\tau \|d_{A(s)} \psi(s)\|_\rho  ds  
          &\rightarrow 0\ \  \text{as}\ \  \tau\downarrow 0 \ \ \text{if}\ \ 2\le \rho<3.  \label{rec715c}
   \end{align}
   For any number $p \in [6,\infty)$ there is a number $\rho \in [2,3)$ such that $p^{-1} = \rho^{-1} - (1/3)$.
   There is therefore a Sobolev constant $\kappa_p$ such that 
   $ \|\psi(s)\|_p \le \kappa_p\(\|d_{A(s)} \psi(s)\|_\rho + \|\psi(s)\|_2\)$. 
      (Actually, the term $\|\psi(s)\|_2$ is not needed if $M = \R^3$.) 
   Hence, in view of \eref{rec715b}, we have  
   \beq
   \int_0^\tau \|\psi(s)\|_p ds <\infty                         \label{rec715e}
   \eeq
       since $\int_0^\tau\|\psi(s)\|_2 ds < \infty$,  as shown in \eref{rec790}. 
   Since \eref{rec715e} holds also for $p = 2$, it holds for  all $p \in [2, \infty)$.      
   This proves \eref{rec713d}.    
   
   For all  $\tau \ge 0$    
   and $ t \ge 0$  we can write 
\beq
d_{A(t)}\alpha_\tau = \int_0^\tau d_{A(s)} \psi(s) ds + \int_0^\tau [A(t) - A(s), \psi(s)] ds.  \label{rec714}
\eeq
Hence 
   \begin{align}
   &\|d_{A(t)} \alpha_\tau \|_\rho  \le \int_0^\tau  \|d_{A(s)} \psi(s)\|_\rho\ ds 
                                                                                    + \| \int_0^\tau [A(t) - A(s), \psi(s)] ds\|_\rho   \notag\\
    &\ \ \ \  \le   \int_0^\tau  \|d_{A(s)} \psi(s)\|_\rho\ ds 
                   + c \sup_{0 < s \le \tau}\| A(t) - A(s)\|_3\int_0^\tau \|\psi(s)\|_p ds,            \label{rec716}
 \end{align}
 where $\rho^{-1} = 3^{-1} + p^{-1}$. The first integral is finite by \eref{rec715b}. 
 Since  $\rho <3$ we have  $p < \infty$ and therefore the integral 
 in \eref{rec716}   is finite  by \eref{rec715e}.  
 The supremum in line \eref{rec716} is finite by \eref{vst356}.
 
             It follows from  
       \eref{rec715b} and \eref{rec716} that  
\begin{align}
d_{A(t)}\alpha_\tau &\in L^\rho( M),\ \  t \ge 0, \ \ \tau \ge 0.    
    \label{rec713a}
\end{align} 
In particular, for $t = T$ and $\rho =2$ we can conclude that \eref{rec713e} holds.
The last estimate in line \eref{rec716} also shows that 
$\sup_{0\le t \le 1}\| \int_0^\tau [A(t) - A(s), \psi(s)] ds\|_\rho \rightarrow 0$ 
 as $\tau \downarrow 0$, which together with   \eref{rec715c} shows that \eref{rec713c} holds.

 It remains to prove the joint continuity \eref{rec713ab}. For $0 \le \tau_0 \le \tau$ we have
 \begin{align}
 &\|d_{A(t)} \alpha_\tau - d_{A(t_0)} \alpha_{\tau_0} \|_\rho
 \le \|(d_{A(t)} -d_{A(t_0)}) \alpha_\tau\|_\rho + \|d_{A(t_0)}(\alpha_\tau - \alpha_{\tau_0})\|_\rho \notag\\
 &\le c \| A(t) - A(t_0)\|_3 \|\alpha_\tau\|_p 
 +\Big\|\int_{\tau_0}^\tau \(d_{A(s)} \psi(s) +[ A(t_0) - A(s)), \psi(s)]\) ds\Big\|_{\rho} \notag \\
&\le c \| A(t) - A(t_0)\|_3  \int_0^\tau \| \psi(s)\|_p ds  +\int_{\tau_0}^\tau  \|d_{A(s)} \psi(s)\|_\rho  \notag\\
& \ \ \ \   +c\sup_{t_0 \le s \le t} \|A(t_0 ) - A(s)\|_3 \int_{\tau_0}^\tau \|\psi(s)\|_p ds.  \notag
 \end{align}
 All three terms go to zero as $|t-t_0| + |\tau - \tau_0| \rightarrow 0$. This concludes the proof of the lemma.
  \end{proof}

\begin{remark}{\rm  (Larger $\rho$ from larger $b$)
The restriction on $\rho$ specified in \eref{rec715a} allows larger $\rho$  for larger $b$. 
In order for this to yield larger $\rho$ in \eref{rec713c} and \eref{rec713ab}  
when $b > 1/2$  it seems
unavoidable  to assume also that $a > 1/2$, so as to allow larger $\rho$ on the left side of \eref{rec716}. 
 For example if $a > 1/2$ then one can use $\|A(t) - A(s)\|_q$ 
 in \eref{rec716} for some $q >3$,  allowing a finite value of $p$ even if $\rho \ge 3$. We won't pursue
 the arithmetic needed for this because we don't foresee a need for this  extension.
 It might be of some interest to note that the vital condition \eref{rec715e}, which we have derived 
 from a Sobolev inequality, also follows from the high $L^p$ bound \eref{ib200p} for $b \ge 1/2$ 
 because  $(3/2) - b - (3/p) < 1$.
 But the use of \eref{ib200p}  disallows Dirichlet  boundary conditions. 
 Perhaps of some ultimate importance is the fact that for $\rho=2$ we have $p =6$ in \eref{rec716}
 and in this case we  
 can use  the simple energy bound \eref{ib10ba}, which already implies that
 $\int_0^\tau \|\psi(s)\|_6 ds < \infty$ for $b >0$.
}
\end{remark}

\bigskip
\noindent
\begin{proof}[Proof of Theorem \ref{strinit}] By \eref{rec1} we have $v(t) = w(t) + d_{A(t)}\alpha_t$.
Both terms are continuous functions of $t \in [0,\infty)$  into $L^\rho(M, \L^1\otimes \kf)$, the first because
$H_b^\As \hookrightarrow L^2 \cap L^3$ is continuous,  and the second  by  \eref{rec713ab}. Thus $v(t)$ is a continuous function into $L^\rho(M; \L^1\otimes \kf)$. Since, by \eref{ve21}, 
 $v_\tau(t) = v(t) - d_{A(t)}\alpha_\tau$, $v_\tau$ is also a continuous function
  into $L^\rho(M;\L^1\otimes \kf)$ by \eref{rec713ab}. 
  This proves \eref{rec976}. Moreover \eref{rec713c} proves \eref{rec977}.
\end{proof}

\bigskip
\noindent
\begin{proof}[Proof of Lemma \ref{lemvert1}](Vertical solutions) 
We  wish to show that the function $z(t): =d_{A(t)} \alpha$ is an almost strong solution
 of the variational equation when $\alpha$ is an element of $H_1^\As(M; \kf)$.
 Note, by the way, that this hypothesis is satisfied by the elements $\alpha_\tau$ defined in \eref{rec3b},
 as we see from \eref{rec713e}.
Under our present hypothesis $z(t)\in L^2(M;\L^1\otimes \kf)$, although it 
 need not be in $H_1^\As$ for any $t > 0$. But by the Bianchi identity, 
\beq
   d_{A(t)} z(t) = [B(t), \alpha],                                  \label{rec01}
   \eeq
 which we will show  is in $H_1^\As$.  First notice that the computation that gives \eref{rec01}
 involves second derivatives of $\alpha$,  which may only exist as distributions. The second order 
 derivatives that appear give $d^2 \alpha$ which is zero in the distribution sense. 
 Thus although the right side of
 \eref{rec01} is a well defined function, the equation has to be interpreted in the distribution sense.
 Now the identity  $ d_{A(t)}^* [B(t), \alpha] = [d_{A(t)}^* B(t), \alpha]  - [d_{A(t)}\alpha \lrc B(t)]$, 
 together with \eref{rec01} and the Yang-Mills heat equation, $ -A'(t) = d_{A(t)}^* B(t)$, yields
 \begin{align}
  d_{A(t)}^* d_{A(t)} z(t) &= -[A'(t), \alpha]  - [z(t) \lrc B(t)], \label{rec02}\\
   d_{A(t)} d_{A(t)} z(t) &= [B(t)\wedge z(t)].    \label{rec03}
   \end{align}
   Since $B(t)$ and $A'(t)$ are both bounded for each $t >0$ by \eref{iby16} and \eref{iby15}, and
   since $\alpha \in L^2(M;\kf)$ and $z(t) \in L^2(M;\L^1\otimes \kf)$, the right sides of
   \eref{rec01}, \eref{rec02} and \eref{rec03} are all in $L^2$. The Gaffney-Friedrichs
   inequality now shows that $d_{A(t)} z(t)$ is in $H_1^{A(t)}$. At the same time, the definition 
   $z(t) = d_{A(t)} \alpha$ 
   shows that  
   $z(t)$ is a solution to the variational equation because $-z'(t) = -[A'(t), \alpha] =d_A^*d_A z + [z\lrc B]$,
   as follows from \eref{rec02}. Hence $z(t)$ is a vertical, almost strong solution to the variational equation.  
   
   In case $z(t_0) \in H_1^\As$ for some $t_0 > 0$ then the identity $z(t) = z(t_0) + [A(t) - A(t_0), \alpha]$ shows
   that $z(t)$ will also be in $H_1^\As $ if $[A(t) - A(t_0), \alpha] \in H_1^\As$.
   But for $t >0$ the latter is in $H_1^\As$ by another Gaffney-Friedrichs argument: Let $\beta = A(t) - A(t_0)$.
   Then $d_\As^* \beta \in L^3$ by \eref{ibA15c},  $d_\As \beta \in L^3$ by \eref{ibA14} 
   since $\As\wedge \beta \in L^6 \cdot L^6 \subset L^3$,  and $\beta \in L^\infty$ by \eref{ibA72}.
   In the meanwhile $d_\As \alpha \in L^2$ and $\alpha \in L^6\cap L^2$. 
   The product rule now shows that $[\beta, \alpha], d_\As [\beta, \alpha]$ and $d_\As^*[\beta, \alpha]$
   are all in $L^2$. Therefore $[\beta,\alpha] \in H_1^\As$.
      
   Thus $z(\cdot)$ is a strong solution if and only if $z(t_0) \in H_1^\As$ for some $t_0 > 0$. 
               This completes  the proof of Lemma \ref{lemvert1}.                
\end{proof}

\subsection{A non-gauge-invariant representation of $v_\tau$}

 The representations of $v$ and $v_\tau$ given in \eref{rec1} and \eref{rec2} 
   capture    the infinitesimal analog of the ZDS procedure. 
The next theorem  gives another, highly non-gauge invariant
 representation of both. It will be needed to prove that 
 $v(t)$ converges to $v_0$ in the $H_b^\As$ norm as $t\downarrow 0$ and to prove that 
 $v_\tau$ has finite b-action for $\tau>0$.

\begin{theorem} \label{thmrec5b}  
 $($A non-gauge-invariant represenstation of $v_\tau(t)$.$)$
Suppose that $w(s)$ is a solution to the augmented variational equation \eref{av1} on $(0,\infty)$.
Let $\psi(s) = d_{A(s)}^* w(s)$ again  as in \eref{ib10psi}. 
Fix $\tau \ge 0$ and define  $v_\tau$  by \eref{rec23}.
Let $P^\perp$ be the projection in $L^2(M; \L^1 \otimes \kf)$ onto the 
orthogonal complement of the null space of $d^*$.
Define
 \begin{align}
 \hat w(s) &= P^\perp w(s).\ \ \ \ \ \         \label{rec751.0}
\end{align}
Then $v_\tau$ is also given by 
      \begin{align}
v_\tau(t) &=w(t) +  \hat w(\tau) - \hat w(t)    \notag \\
  &+\int_\tau^t\Big( [A(t), \psi(s)] - P^\perp\(\zeta(s) + [A(s), \psi(s)]\)\) ds  \label{rec750o}                        
  \end{align}
  for $t >0$ and $\tau \ge 0$.
  The spatial derivatives 
 of the integrand  in \eref{rec750o} are given by
 \begin{align}
 d \Big( [A(t), \psi(s)] - &P^\perp\(\zeta(s) + [A(s), \psi(s)]\) \) = d[A(t), \psi(s)]      \label{rec756}\\
 d^* \Big( [A(t), \psi(s)] - & P^\perp\(\zeta(s) + [A(s), \psi(s)]\)\)  \notag\\
 = d^*[&A(t) - A(s), \psi(s)]    - [w(s)\lrc A'(s)]  +[ A(s)\lrc  \zeta(s) ].          \label{rec757o}     
 \end{align}
\end{theorem} 
    \begin{proof} 
Apply the projection $P^\perp$ to \eref{rec767}.   $P^\perp$ is the identity
operator on exact 1-forms since these span the orthogonal complement of the kernel of $d^*$.
Thus we have
\begin{align}
d\int_\tau^t  \psi(s) ds = \hat w(\tau) - \hat w(t) -\int_\tau^t P^\perp \(\zeta(s) +[A(s), \psi(s)]\) ds \notag
\end{align}
and therefore
\begin{align}
d_{A(t)}&\int_\tau^t  \psi(s) ds  \notag \\
&= \hat w(\tau) - \hat w(t) + \int_\tau^t\([A(t), \psi(s)] - P^\perp \(\zeta(s) +[A(s), \psi(s)]\)\) ds. \notag
\end{align}
Add $w(t)$ to find \eref{rec750o}.

The identity \eref{rec756} follows immediately from  the identity $dP^\perp =0$ on any 1-form.
For the proof of  \eref{rec757o} 
 we can use the identity  $d^*P^\perp \w = d^*\w$  and \eref{pi11a} to find 
\begin{align*}
d^*P^\perp \( \zeta(s) + [A(s), \psi(s)]\)  &= d^*  \Big(\zeta + [A, \psi]\Big)\\
&= d^*   \zeta  +d^* [A, \psi] \\
&= d_A^*\zeta - [A \lrc  \zeta(s) ]   +d^* [A, \psi]  \\
& =[w(s)\lrc A'(s)]  - [A(s) \lrc  \zeta(s) ]   +d^* [A(s), \psi(s)].
\end{align*} 
This completes the proof of the theorem.
\end{proof}

      \begin{remark}  {\rm  
   In addition to the three representations of $v_\tau$, \eref{rec23}, \eref{rec750} and \eref{rec750o},
       there is a fourth representation of $v(t)$ that has a more
 gauge invariant structure than \eref{rec750o}.
Let $\As = A(T)$ as before and denote by $P_\As$ the projection in 
$L^2(M; \L\otimes \kf)$ onto the null space of $d_\As^*$. Much like in the derivation of \eref{rec750o}
 from \eref{rec767} we can deduce  that 
 \begin{align}
&d_{A(t)} \int_\tau^t  \psi(s) ds             \label{rec750A} \\
&= \hat w(\tau) - \hat w(t) +
\int_\tau^t \Big\{ [A(t) - \As, \psi(s)] -P_\As^\perp \(\zeta(s) +[A(s) - \As, \psi(s)]\)\Big\} ds   \notag
\end{align}
where $\hat w(t) = P_\As^\perp w(t)$. 
         This representation has the advantage that only differences $A(t) - \As$ occur, making the
 representation gauge invariant. But the analog of the identity \eref{rec756} fails because 
 $d_\As P_\As^\perp \ne 0$.      In fact $d_\As P_\As^\perp \w = [B(T), G_\As d_\As^*\w]$ on 1-forms
 $\w$, where $G_\As =(d_\As^* d_\As)^{-1} \phi$  is the Green operator on $\kf $ valued scalars.
 Attempts to use this Green operator in our context have not been successful.
}
\end{remark}

\subsection{Initial value of the almost strong solution in the $H_b^\As$ sense} \label{secivas}

     Our techniques in Sections  \ref{secivas} and \ref{secrecab} are going to rely on using the 
     non-gauge invariant Sobolev space $H_1^0$ defined by
     \beq
     \|\w\|_{H_1^0}^2 
     = \int_M \Big(\sum_{j=1}^3|\p_j \w(x)|_{\L^1\otimes \kf}^2  +|\w(x)|_{\L^1\otimes \kf}^2\Big)dx.
     \eeq     
     All results in the preceding sections have made (usually unavoidable) use of the gauge 
     invariant Sobolev norm $H_1^\As$.  In order to transfer information from preceding sections
     to the present two sections it will be necessary to show equivalence of these two norms.
     Under our standing assumption, that $A(\cdot)$ is a strong solution to the Yang-Mills heat equation,
     these two norms are automatically equivalent  because $\As \equiv A(T) \in L^6 \cap L^2 \subset L^3 $,
      and the next lemma assures that the two norms are equivalent when $\As \in L^3$.     
     The equivalence of these norms is therefore  not an issue for this paper. 
     But we plan to use a weaker notion of solution in \cite{G72} 
      in order to include sections of instantons into the Yang-Mills configuration space.  
      The condition $A(T) \in L^3(\R^3)$ will be disallowed. We therefore
     wish to keep track of  exactly where the  condition $A(T)  \in L^3(\R^3)$ is used in this paper.
     The hypothesis that $\As  \in L^3(M)$ in the theorems of Sections \ref{secivas} and \ref{secrecab}
     are therefore purposefully made explicit even though they already are implied by
      the assumption that $A(\cdot)$ is a  strong solution.
       The condition that $\As \in L^3(M)$ has not been used in any previous part  of this  paper.
            This discussion is of substance  only if $M = \R^3$  because $L^6(M)\subset L^3(M)$ 
            when $M$ is bounded  while $A(T)$ will always be in $L^6(M)$ in \cite{G72}. 
    
      The equivalence of $H_1$
     norms is based on the following elementary lemma.

     \begin{lemma} \label{lemeqSob} Suppose that $A_1$ and $A_2$ are two connection forms over
      a region ${\mathcal R}\subset \R^3$  lying in $W_1(\mathcal R)$.
      Then for all $\w$ in the  domains of the following operators one has
     \begin{align}
     \|\n^{A_1} \w\|_2^2 &\le 
     C\| \n^{A_2} \w \|_2^2 \qquad\qquad\ \ \ \ \, \text{if} \ \ \mathcal R = \R^3    \label{se5}\\
      \|\n^{A_1} \w\|_2^2 +\| \w\|_2^2  &\le  
     C\Big( \| \n^{A_2} \w \|_2^2 +\|\w\|_2^2\Big)\ \ \ \text{if} \ \ \mathcal R = \R^3\ \text{or}\ M,  \label{se6}
     \end{align}     
     where 
     \begin{align}
     C &= 1 + \(1 +c\kappa_6\|A_1 - A_2\|_{L^3(\mathcal R)}\)^2          \label{se7}
     \end{align}
     and $\kappa_6$ is a Sobolev constant. In particular, the norms $H_1^0(M)$ and $H_1^\As(M)$ are 
     equivalent if  $\As \in L^3(M)$ or if $A(s) \in L^3(M)$ for some $s >0$.
          \end{lemma}
                  \begin{proof} In case $\mathcal R = \R^3$ we have
           \begin{align}
           \|\n^{A_1} \w\|_2  &\le \|\n^{A_2} \w\|_2  +\|(A_1 - A_2)\w\|_2 \notag   \\
           &\le   \|\n^{A_2} \w\|_2  + c\|(A_1 - A_2)\|_3 \|\w\|_6    \notag   \\
           &\le   \|\n^{A_2} \w\|_2  + c\|(A_1 - A_2)\|_3\, \kappa_6 \| \n^{A_2} \w\|_2,       \label{se8}
           \end{align}
           where $\kappa_6$ is a Sobolev constant for which 
           $\|\w\|_6 \le \kappa_6 \|\n^{A_2} \w\|_2\ \forall \w \in C_c^\infty(\R^3)$.
           Square \eref{se8} to find \eref{se5}. Add $\|\w\|_2^2$ to the square of  \eref{se8} 
            to find \eref{se6} (without the initial term $1$) in case $\mathcal R = \R^3$. 
            The same proof applies in case $\mathcal R = M$
            but one needs the additional term $\|\w\|_2^2$ from the start to use the  Sobolev inequality
            $\|\w\|_6^2 \le \kappa_6^2 \Big( \|\n^{A_2} \w\|_2^2 + \| \w\|_2^2\Big)$ for Neumann
            or Dirichlet boundary conditions.
            
            In particular if $\As \in L^3(M)$ then, choosing $A_1= 0$ and $A_2= \As$ and vice versa,
            we see that $H_1^0$ and $H_1^\As$ are equivalent.     
                   Moreover if $A(s) \in L^3(M)$     then so is $A(T) \in L^3(M)$ because $A(T) - A(s) \in L^3(M)$
                   by \eref{vst356}.  Therefore $H_1^0$ and $H_1^\As$ are equivalent.                    
   \end{proof}

         \begin{theorem} \label{thmrec7o}  
 Let $ 1/2 \le a <1$ and $1/2 \le b <1$. Assume that $\|A(s)\|_{L^3(M)} < \infty$ for some $s >0$.
 Suppose that $w$ is a solution to \eref{av1} with finite b-action.
 Define $v(t)$ by \eref{rec1}.   Then
 \begin{align}
 \|v(t) - v_0\|_{H_b^\As} \rightarrow 0 \ \ \  \text{as}  \ \ t\downarrow 0.   \label{rec779}
 \end{align} 
 Furthermore $v$ is a continuous function into $H_b^\As$ on all of $[0, \infty)$.
 \end{theorem}
 
 For the proof of continuity of $v$ at $t=0$ we need to show that
the second term in \eref{rec1}, i.e. the vertical correction,  converges to zero in $H_b^\As$,
since $w(t)$ converges to $v_0$ in $H_b^\As$ by Theorem \ref{thmmild1b}.
Proof of continuity at $t =0$ is more delicate than at  $t >0$  and  will be proved in the next theorem.

 \begin{theorem} \label{thmrec7} $($The vertical correction$)$           
 Assume that $\|A(T)\|_{L^3(M)} <\infty$.  
  Let $ 1/2 \le a <1$ and $1/2 \le b <1$. Suppose that $w$ is a solution to \eref{av1} with finite b-action.
 Then
\begin{align}
\| d_{A(t)} \int_0^t   \psi(s) ds \|_{H_b^\As }&\rightarrow 0\ \  \text{as}  \ \ t\downarrow 0.\label{rec780}
\end{align}
In particular the left hand side of \eref{rec780} is finite for $0 < t <\infty $.

\end{theorem}
 
     The proof depends on the following lemmas.
     
  \begin{lemma}\label{lemstrat2}   $($Riesz avoidance$)$  
  Let $0 \le b \le 1$. Define $p$ in the interval $[6/5, 2]$ by  
\beq
p^{-1} = 2^{-1} +(1-b)/3.    \label{rec864}
\eeq  
 If $\w$ is a $\kf$ valued 1-form in $L^2(M)$ with $d\w \in L^p$ and $d^*\w \in L^p$  then $ \w \in H_b$.
  There is a Sobolev constant   $c_p$ 
  such that
  \begin{align}
\|\w\|_{H_b^0} \le c_p\Big(\|d^* \w\|_p + \| d\w\|_p + \|\w\|_2\Big).    \label{h514s}
\end{align}
 \end{lemma}
          \begin{proof} This is a slightly simplified version of \cite[Lemma 6.17]{G70} 
 \end{proof}

\begin{lemma}\label{lemrec30}  
 Assume that $1/2 \le a < 1$
 and $1/2 \le b < 1$. Let $ p^{-1} = (5/6) - (b/3)$ as in  \eref{rec864}. 
 Write $\zeta(s)= d_{A(s)}^* d_{A(s)} w(s) +[w(s)\lrc B(s)]$ as in \eref{ib549}.
 Then for any $t \ge0$ we have
 \begin{align}
 \int_0^t\|\ [w(s)\lrc A'(s)]\ \|_p  ds &< \infty,                  \label{rec870} \\
 \int_0^t \|\ [A(s)\lrc \zeta(s)]\ \|_p ds &< \infty   \ \ \ \text{and}            \label{rec871}\\
 \int_0^t \|\zeta(s)\|_2 ds &< \infty.                       \label{rec872}  
 \end{align}
 Define $\eta(t) = \int_0^t \psi(s)ds$. Then
 \begin{align}
\int_0^t \| d^*[A(s) - A(t), \psi(s)] \|_p ds &\rightarrow 0 \ \ \text{as}\ \ \ t\downarrow 0,    \label{rec929c} \\
\int_0^t \| d[A(t), \psi(s)] \|_p ds &\rightarrow 0   \ \   \text{as}\ \ \ t \downarrow 0\ \ \text{and}     \label{rec929d} \\
 \ \ \ \ \ \|\ [B(t), \eta(t)]\ \|_p &\rightarrow 0 \ \   \text{as}\ \ \ t \downarrow 0. \ \             \label{rec850c} 
\end{align}
 \end{lemma}
        \begin{proof} Let $r^{-1} + 6^{-1} = p^{-1}$. Then $r^{-1} = (2/3) - (b/3)$ and 
\begin{align*}
 \int_0^t \|\  &[w(s)\lrc A'(s)]\ \|_p  ds \le c \int_0^t  \|w(s)\|_6 \| A'(s)\|_r ds \\
 &\le c \Big(\int_0^t s^{-b} \|w(s)\|_6^2 ds \Big)^{1/2} \Big(\int_0^t s^b \|A'(s)\|_r^2 ds \Big)^{1/2}.
 \end{align*}
 The first factor is finite because $w$ is assumed to have finite b-action. 
Using the interpolation inequality \eref{ib214} with $f(s,x) = |A'(s, x)|$  we see that
 the second factor is also  finite in view of the initial   behavior bounds  
\eref{iby8} and \eref{iby9} with $a = 1/2$. Here, as elsewhere, we are using the fact that $A$ has
finite (1/2)-action if it has finite $a$-action for some $a \in [1/2, 1)$. This proves \eref{rec870}.

      For the proof of \eref{rec871} choose $r$ as above. Then
      \begin{align*}
 \int_0^t \|\ [A(s)&\lrc \zeta(s)]\ \|_p ds \le \int_0^t \|A(s)\|_6 \|\zeta(s)\|_r ds \\
 &\le c \(\int_0^t s^{-1/2}\|A(s)\|_6^2 ds \)^{1/2} \(\int_0^t s^{1/2} \| \zeta(s) \|_r^2 ds \)^{1/2}.
\end{align*}
The first factor  is finite by \eref{ibA6a}, with $a = 1/2$.
The second factor is finite by the initial behavior bound  \eref{ib547}. This proves \eref{rec871}.

 For the proof of  \eref{rec872} we have 
  \begin{align}
   \int_0^t \|\zeta(s)\|_2 ds &\le \Big(\int_0^ts^{b-1} ds \Big)^{1/2}   \notag
   \Big(\int_0^t s^{1-b} \|\zeta(s)\|_2^2ds\)^{1/2}
   \end{align}
   which is finite when $b >0$ by \eref{ib10ba}. This proves \eref{rec872}.

To prove \eref{rec929c}  let $\alpha(s) = A(s) -A(t)$ and let $q^{-1} =p^{-1} - 3^{-1} =(1/2) - (b/3)$.
The identity  
\beq
d^*[\alpha(s), \psi(s)] = [\alpha(s)\lrc d\psi] + [d^* \alpha(s), \psi]  \label{rec931o}
\eeq
   allows the bounds
\begin{align}
\int_0^t \|d^*[\alpha(s), &\psi(s)]\, \|_p ds \le c\int_0^t\( \|\alpha(s)\|_6 \|d \psi(s)]\, \|_r  
+\|d^* \alpha(s)\|_3   \| \psi\|_q\)ds  \notag \\
&\le c \(\int_0^t  s^{-1/2}\|\alpha(s)\|_6^2 ds\)^{1/2} \(\int_0^t s^{1/2} \|d \psi(s)]\, \|_r^2 ds \)^{1/2} \notag\\
&+ c\(\int_0^t  \|d^* \alpha(s)\|_3^2ds\)^{1/2}   \(\int_0^t  \| \psi\|_q^2ds\)^{1/2}    \label{rec933}
\end{align}
The four factors are finite by \eref{ibA6} (with $a = 1/2$), \eref{ib547}, \eref{ibA15b}
 and \eref{ib546}, respectively. This proves \eref{rec929c}.
   
    The proof of \eref{rec929d} resembles, partly, the preceding proof.
        Similar to the bounds in \eref{rec933}, we have
    \begin{align}
 &\int_0^t \| d[A(t), \psi(s)]\, \|_p ds 
 \le  \int_0^t \( \|\, [A(t)\wedge d\psi(s)]\, \|_p  + \|\, [ dA(t), \psi(s)]\, \|_p \)ds  \notag\\
 & \le c \int_0^t\( \|A(t)\|_6 \|d\psi(s)\|_r  + \|dA(t)\|_3  \|\psi(s)\|_q\) ds \notag\\
 &\le c \|A(t)\|_6 \(\int_0^t s^{-1/2} ds \)^{1/2}\(\int_0^t s^{1/2} \|d\psi(s)\|_r^2 ds \)^{1/2} \label{rec934} \\
 &\ \ \ \ \ \ \ \ \ \ \ +c \| d A(t)\|_3\ t^{1/2} \( \int_0^t \|\psi(s)\|_q^2 ds \)^{1/2} .                \label{rec935}
 \end{align} 
 The last  integral in line \eref{rec934} is finite by \eref{ib547} while the first two
  factors are $\| A(t)\|_ 6 O(t^{1/4})$,
 which goes to zero as $t \downarrow 0$ in accordance with \eref{ibA5a}. 
 In line \eref{rec935} the product of the first two factors goes to zero by \eref{ibA14} while
  the last integral is finite by \eref{ib546}. This proves \eref{rec929d}.

       For the proof of \eref{rec850c} we have
    \begin{align}
    \|\, [B(t), \eta(t)]\, \|_p \le c \|B(t)\|_6 \|\eta(t)\|_r. \notag
      \end{align}
  Now
  \begin{align}
   \|\eta(t)\|_r &\le \int_0^t \|\psi(s)\|_r ds    \notag\\
   &\le \(\int_0^t s^{1/2} ds\)^{1/2} \(\int_0^t s^{-1/2} \|\psi(s)\|_r^2 ds \)^{1/2}. \notag
   \end{align}
   The second factor is bounded for $t \in [0, T]$ by \eref{ib216}. The first factor  is $O(t^{3/4})$.
   Since, by \eref{iby3},  $t^{3/4} \|B(t)\|_6 =o(1)$ as $ t \downarrow 0$ the  assertion \eref{rec850c} follows.
    \end{proof}

             \begin{lemma}\label{lemrec31}
 Assume that $1/2 \le a < 1$  and $1/2 \le b < 1$. Let $ p^{-1} = (5/6) - (b/3)$ as in  \eref{rec864}. 
 Let $t_0 >0$. Then 
 \begin{align}
& \| \int_0^t d[A(t) -A(t_0), \psi(s) ] ds\|_p 
+\|\int_{t_0}^t d[A(t_0), \psi(s)] ds\|_p            \label{rec940}  \\                
 &   \| \int_0^t d^*[A(t)-A(t_0), \psi(s) ] ds\|_p   
 +\|\int_{t_0}^t  d^*[A(t_0) - A(s),\psi(s)] ds \|_p       \rightarrow 0 \notag
 \end{align}
 as $t \rightarrow t_0$. 
 \end{lemma}  
              \begin{proof} Just as in the inequalities leading to \eref{rec935} we have
\begin{align}
&\int_0^t  \|d[A(t) -A(t_0), \psi(s) ] ds\|_p +\int_0^t  \|d^*[A(t) -A(t_0), \psi(s) ] ds\|_p \\
& \le 4c \|A(t)- A(t_0)\|_6 t^{1/4}\(\int_0^t  s^{1/2} \|d\psi(s)\|_r ^2ds\)^{1/2}  \notag\\
&+c \(\| d (A(t) -A(t_0))\|_3 + \| d^* (A(t) -A(t_0))\|_3\)\ t^{1/2} \( \int_0^t \|\psi(s)\|_q^2 ds \)^{1/2}  \notag
\end{align}
Now $\| A(t) - A(t_0)\|_6 \rightarrow 0$ as $t \rightarrow t_0$ because of \eref{iby5}.
 The two $L^3$ norms in the last line 
are also continuous in $t$ at $t_0 \ne 0$, as we can derive from the two identities \eref{ibA40} 
 and   $B'(s) =d_A A'(s) = dA'(s) +[A(s)\wedge A'(s)]$. Indeed we find
  \begin{align}
  \|d^*(A(t) - A(t_0))\|_3 &\le | \int_{t_0}^t c \|A(s)\|_6 \|A'(s)\|_6 ds| \ \ \ \text{and} \\
  \|d (A(t) - A(t_0))\|_3 &\le |\int_{t_0}^t \|B'(s)\|_3 ds| +| \int_{t_0}^t c \|A(s)\|_6 \|A'(s)\|_6 ds|
  \end{align}
  As long as $t$ and $t_0$ stay away from zero all of these integrals are finite, 
  by  \eref{iby5} and \eref{ibA5a}, for the first line and part of the  second line, and 
  by \eref{iby12} and \eref{iby13}  for the rest of the second line, and
  go to zero as $t \rightarrow t_0$. To deal with the second term in \eref{rec940} just replace the integral over
  $[0, t]$ in \eref{rec934} and \eref{rec935} by the integral over $[t_0, t]$  and replace
   $A(t)$ by $A(t_0)$ everywhere. Use \eref{ibA14} to see that $\|dA(t_0)\|_3 < \infty$. It follows that the second
   term in \eref{rec940} goes to zero.
   Similarly, one need only replace the integrals over $[0, t]$ in \eref{rec933}  by integrals over $[t_0, t]$
   and put $t =t_0$ in the definition of    $\alpha(s)$ to see that the fourth term in \eref{rec940} goes to zero
   as $t \rightarrow t_0$. 
 \end{proof}

\begin{remark}{\rm
 
         It's interesting to observe that     one cannot replace $d$ by $d^*$ in the
          inequality  \eref{rec929d} because $\|d^*A(t)\|_3$ can be identically infinite under
  our hypothesis - that $A(\cdot)$ has finite action. It's tempting to believe that some power counting scheme
  could unify the many estimates in this paper. But it would have to take into 
  account the fact that $dA(t)$ and $d^* A(t)$
  have very different behavior, in spite of both being first derivatives of $A$.
 }
 \end{remark}

\bigskip
\noindent
     \begin{proof}[Proof of Theorem \ref{thmrec7}] 
 By the representation \eref{rec750o} 
with $\tau =0$ we may write
the vertical correction for $\tau =0$ as
\begin{align} 
d_{A(t)} \int_0^t \psi(s) ds =  \hat w(0) - \hat w(t)    
  +\int_0^t \gamma(t,s) ds     \ \ \   0< t <\infty,           \label{rec750o0}   
  \end{align}
  where
    \beq
  \gamma(t,s) = [A(t), \psi(s)] - P^\perp\(\zeta(s) + [A(s), \psi(s)]\) . \label{rec750o1}
  \eeq      
 Since $w(t)$ is continuous into $H_b^0$ so is  
     $\hat w(t)$ because $P^\perp:H_b^0 \rightarrow H_b^0$ is continuous by   \cite[Lemma 6.10]{G70}.
     Therefore $\|\hat w(0) - \hat w(t)\|_{H_b^0} \rightarrow 0$ as $t \downarrow 0$.
With a view toward applying Lemma \ref{lemstrat2} observe that \eref{rec756} and \eref{rec757o} give
\begin{align}
d\int_0^t \gamma(t,s) ds &= \int_0^t d[A(t), \psi(s) ] ds                             \label{rec752a}\\
d^*\int_0^t \gamma(t,s) ds&=  \int_0^t d^*[A(t)-A(s), \psi(s) ] ds              \label{rec752b}\\
 & \qquad\qquad + \int_0^t \( -[w(s) \lrc A'(s)] + [A(s) \lrc \zeta(s)] \) ds.     \label{rec752c}
\end{align}
By Lemma  \ref{lemstrat2} it suffices to show that the $L^p(M)$ norm of each of these three integrals 
goes to zero as $t\downarrow 0$ as well as  $\|\int_0^t \gamma(t,s) ds\|_2$.
           But  \linebreak 
    $\int_0^t \| d[A(t), \psi(s)]\, \|_p ds \rightarrow 0$ by  \eref{rec929d} and
$ \int_0^t \| d^*[A(t) -A(s), \psi(s)] \|_p ds \rightarrow 0$ by \eref{rec929c}. \eref{rec870} and \eref{rec871}
show that the line \eref{rec752c} also goes to zero in $L^p(M)$ norm.
 Furthermore,       
\begin{align*}
&\int_0^t\|\gamma(t,s)\|_2 ds \le \int_0^t\(\|\, [A(t),\psi(s)]\,\|_2  +\|\zeta(s)\|_2 + \|\, [A(s),\psi(s)]\,\|_2\) ds  \\
&\le c\|A(t)\|_6 \int_0^t \| \psi(s)\|_3 ds  +\int_0^t \|\zeta(s)\|_2 ds +c\int_0^t\| A(s)\|_6 \|\psi(s)\|_3 ds \\
&\le c\|A(t)\|_6 o(t^{1/4})   +\int_0^t \|\zeta(s)\|_2 ds   \\
&+c\(\int_0^t s^{b-(1/2)} \|A(s)\|_6^2 ds\)^{1/2} \(\int_0^t s^{(1/2) - b} \|\psi(s)\|_3^2 ds\)^{1/2}
\end{align*} 
by \eref{ib10d} (for $b \ge 0$).
 All three terms go to zero as $t\downarrow 0$,
the first by \eref{ibA5a}, the second by \eref{rec872} and the third by \eref{ibA6a} and \eref{ib10c}.
 This completes the proof that the vertical correction 
 goes to zero in $H_b^0$ norm.
 
      Now $\As \in L^3(M)$ as assumed in the statement of the theorem. Therefore, by Lemma \ref{lemeqSob}
      the $H_1^\As$ norm is equivalent to the $H_1^0$ norm. By interpolation between $H_1$ and $L^2$
      it follows that the $H_b^\As$ norm is equivalent to the $H_b^0$ norm. 
      This concludes the proof of  \eref{rec780}. 
\end{proof}

\bigskip
\noindent
\begin{proof}[Proof of Theorem \ref{thmrec7o}]
 Continuity of $v$ at $t =0$ has been proved in Theorem \ref{thmrec7}.  Suppose then that
$t_0 > 0$ and $t >0$. We need to show that the differences of the integral
 term in \eref{rec750o0} go to zero in $H_b^0$ norm.  Taking differences at $t$ and $t_0$ in 
 \eref{rec752a} - \eref{rec752c} we find
\begin{align}
    d\int_0^t \gamma(t,s) ds  - d&\int_0^{t_0} \gamma(t_0,s)
= \int_0^t d[A(t) -A(t_0), \psi(s) ] ds  \label{rec753a} \\
&+\int_{t_0}^t d[A(t_0), \psi(s)] ds                           \label{rec753b}\\
    d^*\int_0^t \gamma(t,s) ds- d^*&\int_0^{t_0} \gamma(t_0,s) ds
 =  \int_0^t d^*[A(t)-A(t_0), \psi(s) ] ds   \label{rec753c} \\
 &+\int_{t_0}^t  d^*[A(t_0) - A(s),\psi(s)] ds          \label{rec753d}\\
 &+ \int_{t_0}^t \( -[w(s) \lrc A'(s)] + [A(s) \lrc \zeta(s)] \) ds.     \label{rec753e}
\end{align}
To apply Lemma  \ref{lemstrat2} we need to show that each of the five integrals on the right go to
zero in $L^p(M)$ norm as $t\rightarrow t_0$. This is clear for the integral in line \eref{rec753e} because of
\eref{rec870} and \eref{rec871}. The $L^p(M)$ norm of each of the remaining four integrals goes to zero
as $t\rightarrow t_0$ by the corresponding assertion in Lemma \ref{lemrec31}.
Finally, to complete the application of Lemma \ref{lemstrat2} we need to show that 
$\|\int_0^t \gamma(t,s)ds - \int_0^{t_0}\gamma(t_0,s)ds\|_2 ds \rightarrow 0$ as $t\rightarrow t_0 \ne 0$. 
The cancelations and estimates
needed are similar to the preceding difference computations but a little simpler. We omit the details.
\end{proof}

\subsection{Finite action of the strong solution} \label{secrecab}

\begin{theorem}\label{thmrec9b} $($Finite action of the strong solution$)$ 
Suppose that $1/2 \le a < 1$ and $1/2\le b <1$ and that $w$
is  a strong solution of the augmented variational equation \eref{av1} with  finite strong b-action
in the sense that \eref{w3c} holds.
Assume that $\As \in L^3(M)$. Suppose also that $max(a,b) > 1/2$.  Let $c = min(a,b)$ and
let  $\tau >0$.  
 Define $v_\tau$ by \eref{rec2}. Then
\begin{align}
\int_0^\tau t^{-c}\Big(  \|\n^{\As} v_\tau(t)\|_2^2  +\|v_\tau(t)\|_2^2 \Big) dt < \infty.      \label{rec880}
\end{align}
\end{theorem}

For the proof we are once again  going to use the non-gauge invariant
 representation of $v_\tau$ given in \eref{rec750o}.
      Let
\begin{align}
u_\tau(t) = \int_\tau^t \gamma(t,s) ds, \label{rec960}
\end{align}
where $\gamma(t,s)$ is defined in \eref{rec750o1}.
Then, by \eref{rec750o}, we have
     \begin{align}
v_\tau(t) =w(t) +  \hat w(\tau) - \hat w(t)     +u_\tau(t).     \label{rec959}
\end{align}    
The  finite strong b-action of the first three terms in \eref{rec959} will be easily  established at
 the end of this section using the equivalence of norms
discussed at the beginning of Section \ref{secivas}. That equivalence will depend on use of the condition
$\As \in L^3(\R^3)$ in case $M = \R^3$.  The proof of finite  strong b-action of $u_\tau$ in the $H_1^0$ norm 
does not depend on this condition but  has a lengthy proof. We will focus on this first. 
  In the next theorem we will establish  finite strong b-action of the fourth term in \eref{rec959}
while merely assuming  finite b-action for $w$.

\begin{theorem} \label{thmrec9c}  
Suppose that $1/2 \le a < 1$ and $1/2\le b <1$ and that $w$ 
 is a solution to  the augmented variational equation \eref{av1} with finite b-action in the sense 
of \eref{w3d}. 
Assume also that $max(a,b) > 1/2$.  Let $c = min(a,b)$ and
let  $\tau >0$.   
  Then
 \begin{align}
 \int_0^\tau t^{-c} \|u_\tau(t)\|_{H_1^0}^2 dt < \infty.  \label{rec938o}
 \end{align}
\end{theorem}

\begin{remark}{\rm (Strategy) We will use the Gaffney-Friedrichs inequality \eref{gaf49} 
with connection form zero.
It  suffices to prove, then,  that 
\begin{align}
\int_0^\tau t^{-c} \Big( \|d^* u_\tau(t)\|_2^2 + \| d u_\tau(t)\|_2^2 + \|u_\tau(t)\|_2^2 \Big) dt < \infty.  \label{rec939o}
\end{align}
}
\end{remark}
The following proposition addresses each of these three terms.

 \begin{proposition} \label{propfa1} 
Under the hypotheses of Theorem \ref{thmrec9c} and with $\gamma(t,s)$ defined in \eref{rec750o1} we have  
\begin{align}
 &\int_0^\tau t^{-c} \(\int_t^\tau  \|d^*\gamma(t, s)\|_2 ds \)^2 dt < \infty \label{rec940o}\\
 & \int_0^\tau t^{-c} \(\int_t^\tau  \|d\gamma(t, s)\|_2 ds \)^2 dt < \infty  \label{rec941o}\\
&\int_0^\tau t^{-a} \(\int_t^\tau \|\gamma(t,s)\|_2 ds \)^2 dt < \infty.   \label{rec942o}
 \end{align}
 \end{proposition}
The proof depends on the following three lemmas and corollary.

       \begin{lemma}\label{lemrec9o}Let $1/2 \le a <1$ and $0 \le b <1$. 
    If $w$ is a solution to the augmented variational equation \eref{av1} with 
       finite b-action in the sense  of \eref{w3d} 
       then  \eref{rec942o} holds.
\end{lemma}
         \begin{proof}
From \eref{rec750o1} we see that 
\begin{align}
\| \gamma(t,s)\|_2 \le   \|\, [A(t), \psi(s)]\,\|_2 + \|\zeta(s)\|_2 +\|\, [A(s), \psi(s)]\, \|_2 \label{rec942i}
\end{align}
It suffices to prove \eref{rec942o} for each of these terms. Now  
 \begin{align}
 \int_0^\tau t^{-a}\( \int_t^\tau \|\, [A(t), \psi(s)]\,\|_2 ds\)^2 dt 
&\le c^2 \int_0^\tau t^{-a}\|A(t)\|_6^2 dt \(\int_0^\tau \|\psi(s)\|_3 ds \)^2,   \notag
 \end{align}
 which is finite by \eref{ibA6a} and  \eref{ib10d} for all $a \in [1/2, 1)$ and all $b \in [0, 1)$.
      
         Concerning the second term, observe that $\int_0^\tau s^{2-a} \|\zeta(s)\|_2^2 ds < \infty$
because $\int_0^\tau s^{1-b} \|\zeta(s)\|_2^2 ds < \infty$ in accordance with \eref{ib10ba}, while
$2-a \ge 1-b$ for all $a \in [1/2, 1)$ and all $b \in [0, 1)$. Therefore Hardy's inequality 
\eref{H2} with $\beta =a$ shows that
\eref{rec942o} holds in this range for the second term in \eref{rec942i}.
Finally,
   \begin{align*}
   \(\int_t^\tau\|\, [A(s), \psi(s)]\,\|_2ds\)^2 &\le \(\int_t^\tau c \|A(s)\|_6 \|\psi(s)\|_3ds \)^2 \\
   &\le c^2\int_0^\tau s^{-1/2} \|A(s)\|_6^2 ds \int_0^\tau s^{1/2}\|\psi(s)\|_3^2 ds, 
   \end{align*}
   which is bounded by \eref{ibA6a} and \eref{ib10c} for all $a \in [1/2, 1)$ and all $b \in [0, 1)$.
   Therefore the third term in \eref{rec942i} makes a finite contribution to \eref{rec942o}.
\end{proof}

\begin{lemma} \label{lemrec9cd}   
  Let $1/2 \le a < 1$ and $0 \le b <1$.
   If $w$ is a solution to the augmented variational equation \eref{av1} with 
    finite b-action in the sense of \eref{w3d}  then 
\begin{align}
&\int_0^\tau s^{2-b}  \|d [A(s),\psi(s)]\, \|_2^2 ds < \infty \ \ \ \text{and}     \label{rec891} \\
& \int_0^\tau s^{2-b} \|\,  [ A(s)\lrc  \zeta(s)]     - [w(s)\lrc A'(s)]\,   \|_2^2 ds 
                  < \infty .         \label{rec892}
\end{align}
\end{lemma}
\begin{proof}
From the identity  $d[A,\psi] = [dA, \psi] - [A\wedge d\psi]$ 
  we find
 \beq
  \|d [A(s),\psi(s)]\, \|_2 \le   c\|dA(s)\|_3 \| \psi(s)\|_6 +c\|A(s)\|_6 \|d\psi(s)\|_3.  \notag
  \eeq
  It suffices, therefore, to prove that
 \begin{align}
 \int_0^\tau s^{2-b}\Big(\|dA(s)\|_3^2 \| \psi(s)\|_6^2 
                   +\|A(s)\|_6^2 \|d\psi(s)\|_3^2  \) ds < \infty. \label{rec883c}
 \end{align}
 But 
 \begin{align}
 \int_0^\tau  s^{2-b}\|dA(s)\|_3^2 \| \psi(s)\|_6^2 ds 
               = \int_0^\tau \(s \|dA(s)\|_3^2\) \(s^{1-b} \| \psi(s)\|_6^2\)ds.     \notag
 \end{align}
 The first factor in the integrand is bounded, by \eref{ibA14}, and the
  second is integrable, by \eref{ib10ba}. 
  Similarly, 
  \beq
   s^{2-b}\|A(s)\|_6^2 \|d\psi(s)\|_3^2 = \(s^{1/2} \|A(s)\|_6^2\)\(s^{(3/2)- b}\|d\psi(s)\|_3^2 \), \notag
   \eeq
   wherein the first factor is bounded, by \eref{ibA5a}, and the second factor is integrable, by \eref{ib548}.
This proves \eref{rec891}.

For the proof of \eref{rec892} we have the bound
 \begin{align}
 \|- [w(s)\lrc A'(s)]  &+[ A(s)\lrc  \zeta(s)]\, \|_2     \notag\\
&  \le c\|w(s)\|_6 \| A'(s)\|_3   +c \|A(s)\|_6 \|\zeta(s)\|_3.  \label{rec881ac} 
\end{align}
Therefore it suffices to show that
\begin{align}
\int_0^\tau s^{2-b}\Big\{ \|w(s)\|_6^2 \|A'(s)\|_3^2 
 + \|A(s)\|_6^2\|\zeta(s)\|_3^2   \Big\}ds  \notag
 \end{align}
 is finite. The first term may be written $(s^{-b}\|w(s)\|_6^2) ( s^2\|A'(s)\|_3^2)$, which is an integrable function
 times a bounded function by respectively \eref{ce401} and \eref{iby7} (for all $a \ge 1/2$). 
 The second term may  be written $(s^{1/2} \| A(s)\|_6^2) ( s^{(3/2) -b} \|\zeta(s)\|_3^2)$,
 which is a bounded function (by \eref{ibA5a}) times an integrable function by \eref{ib548}.
This completes the proof of the lemma.
\end{proof}

\begin{corollary} \label{correc9} 
  Let $1/2 \le a < 1$ and $0 \le b <1$.
  If $w$ is a solution to the augmented variational equation \eref{av1} with finite b-action
  in the sense of \eref{w3d}  then 
\begin{align}
\int_0^\tau t^{-b}\Big(\int_t^\tau \|d [A(s),\psi(s)]\, \|_2 ds \Big)^2 dt < \infty  \ \ \ \text{and}    \label{rec883b} \\
\int_0^\tau t^{-b} \( \int_t^\tau  \|- [w\lrc A']  
                 +[ A\lrc  \zeta(s)] \|_2 ds\)^2 dt  <\infty.                                \label{rec887a}
\end{align}
\end{corollary}
\begin{proof}
By Hardy's inequality, \eref{H2}, the left hand side of \eref{rec883b} is at most 
 \begin{align}
 \frac{4}{(1-b)^2}\int_0^\tau s^{2-b}  \|d [A(s),\psi(s)]\, \|_2^2 ds,
 \end{align}
 which has been shown to be finite in Lemma \ref{lemrec9cd}.  By the same argument, \eref{rec887a}
 follows from \eref{rec892}.
\end{proof}

\begin{lemma}\label{lemrecef} 
Suppose that $1/2 \le a < 1$ and $1/2\le b <1$. Assume that $max(a,b) > 1/2$.  Let $c = min(a,b)$.
 If $w$ is a solution to the augmented variational equation \eref{av1} with 
 finite b-action in the sense of \eref{w3d}   then 
\begin{align}
\int_0^\tau t^{-c}\( \int_t^\tau \| d^*[A(t) - A(s), \psi(s)]\, \|_2 ds \)^2 dt &< \infty\ \ \text{and}  \label{rec887c}\\
\int_0^\tau t^{-c}\( \int_t^\tau \| d[A(t) - A(s), \psi(s)]\, \|_2 ds \)^2 dt &< \infty.        \label{rec887d}
\end{align}
\end{lemma}
       \begin{proof}  Unlike the inequalities in Corollary \ref{correc9}, we cannot use Hardy's inequality now
       because the integrands in \eref{rec887c} and \eref{rec887d} depend on $t$.
       Let $\alpha =\alpha(s,t) = A(s) -A(t)$. Then, similar to the identity \eref{rec931o}, we have
       \begin{align}
d[\alpha, \psi(s)] &= [d\alpha, \psi(s)]    
               -[\alpha \wedge d\psi(s)]  .         \label{rec932o}
\end{align}  
Therefore
 \begin{align}
\| d^*[\alpha(s,t), \psi(s)]\, \|_2         
&\le  c\|d^*\alpha(s,t)\|_2 \|\psi(s)\|_\infty +c \|\alpha(s,t)\|_6 \| d \psi(s)\|_3     \label{rec902c}    \\
\| d[\alpha(s,t), \psi(s)]\, \|_2         
&\le  c\|d\alpha(s,t)\|_2 \|\psi(s)\|_\infty +c \|\alpha(s,t)\|_6 \| d \psi(s)\|_3  .        \notag                  
\end{align} 

The statements   \eref{ibA10b} and \eref{ibA10c} show that $\int_0^s t^{-a} \| d^\# \alpha(s,t)\|_2^2 dt$
is bounded for $s \in [0, T]$, where $d^\# = d$ or $d^*$. It will suffice to prove only \eref{rec887c}.

We need to show that each of the two terms on the right side of \eref{rec902c} 
makes a finite contribution to \eref{rec887c}.
For the first term we have
\begin{align}
&\int_0^\tau t^{-c} \(\int_t^\tau \|d^* \alpha(s,t)\|_2  \|\psi(s)\|_\infty ds \)^2 dt   \notag \\
&\le \int_0^\tau t^{-c} \(\int_t^\tau s^{b-(3/2)}\|d^* \alpha(s,t)\|_2^2ds\)
   \(\int_0^\tau s^{(3/2) -b} \| \psi(s)\|_\infty^2  ds\) dt   \notag\\
   &=\int_0^\tau s^{b -(3/2) }\(\int_0^s  t^{-c}\|d^* \alpha(s,t)\|_2^2dt\) ds 
          \(\int_0^\tau s^{(3/2) -b} \|  \psi(s)\|_\infty^2  ds\)       \label{rec902d}
\end{align}
The last factor in \eref{rec902d} is finite for any $b \in [0,1)$ by \eref{ib200p} with $p =\infty$.
Moreover  $\int_0^s  t^{-c}\|d^* \alpha(s,t)\|_2^2dt$   is bounded for $0\le s \le \tau$ because $c \le a$.
Therefore if $b > 1/2$ then the right hand side of \eref{rec902d}  is finite.  
 If $b =1/2$ then $c = 1/2$ and $a > 1/2$. In this case we may write $t^{-c} = t^{a -(1/2)} t^{-a}$
to find, since $b-(3/2) = -1$ that the integral is at most
\begin{align*}
\int_0^\tau s^{-1} s^{a - (1/2)}\(\int_0^s  t^{-a}\|d^* \alpha(s,t)\|_2^2dt\) ds 
          \(\int_0^\tau s^{(3/2) -b} \|  \psi(s)\|_\infty^2  ds\) < \infty.
          \end{align*}
The contribution to \eref{rec887c} of the second term in \eref{rec902c} can be estimated
by the same use of Schwarz' inequality 
and reversal of order of integration as for the first term, giving
\begin{align}
&\int_0^\tau t^{-c} \(\int_t^\tau \|\alpha(s,t)\|_6  \| d \psi(s)\|_3  ds \)^2 dt   \notag \\
   &\le \int_0^\tau s^{b -(3/2) }\(\int_0^s  t^{-c}\|\alpha(s,t)\|_6^2dt\) ds 
          \(\int_0^\tau s^{(3/2) -b} \| d \psi(s)\|_3^2  ds\).                   
\end{align}
Since $\int_0^s t^{-a}\|\alpha(s,t)\|_6^2dt$ is bounded for $s \in (0, \tau]$, by \eref{ibA6},
and  since \linebreak
$\int_0^\tau    s^{(3/2) -b} \| d \psi(s)\|_3^2  ds <\infty$ by \eref{ib548}, the same argument used for the first term applies to the second term as  well.
\end{proof}

\bigskip
\noindent
\begin{proof}[Proof of Proposition \ref{propfa1}] We see from  
\eref{rec750o1}, \eref{rec756}  and\eref{rec757o} that 
\begin{align}
d\gamma(t,s) &=     d[A(t), \psi(s)] \ \ \ \ \    \ \ \ \text{and} \\
d^*\gamma(t,s) &= d^*[A(t) - A(s), \psi(s)]                                
 - [w(s)\lrc A'(s)]  +[ A(s)\lrc  \zeta(s) ].
\end{align}
Therefore, combining 
 \eref{rec887a} and \eref{rec887c} we find \eref{rec940o}.
By combining \eref{rec887d} with \eref{rec883b} we find \eref{rec941o} because 
$d[A(t), \psi(s)] = d[A(t) -A(s), \psi(s)] + d[A(s),\psi(s)]$.
In applying Corollary \ref{correc9} one should observe that $c \le b$.
Therefore all assertions in Proposition \ref{propfa1} have been proved.
\end{proof}

\bigskip
\noindent
\begin{proof}[Proof of Theorem \ref{thmrec9c}]  
 Since $c \le a$  the inequality
\eref{rec942o} holds for $a$ replaced by $c$. Therefore all three integrals in \eref{rec939o} have been
 shown to be finite. This proves \eref{rec939o} and therefore \eref{rec938o}. 
\end{proof}

\bigskip
\noindent
\begin{proof}[Proof of Theorem \ref{thmrec9b}]  
 In the representation  \eref{rec959}  of the strong solution  $v_\tau(t)$ 
 the first term  $w(t)$ has 
 finite strong b-action  for $H_1^\As$ by assumption
  and therefore  
   finite strong  b-action for $H_1^0$ by the equivalence of norms.
  Here we are using the hypothesis $\As \in L^3(M)$,
 which is a substantial hypothesis in case $M = \R^3$ but is automatic when $M$ is bounded because
 $\As$ is always in $L^6(M)$, being a strong solution to the Yang-Mills heat equation.
 Since $c \le b$, the first term also has finite strong c-action for $H_1^0$.

Concerning the second term, note that $w(\tau) \in H_1^\As = H_1^0$. Since $P^\perp$ is continuous
on $H_1^0$ it follows  that $\hat w(\tau) \in H_1^0$. 
 Therefore 
 $\int_0^\tau t^{-c} \| \hat w(\tau)\|_{H_1^0}^2 dt = \| \hat w(\tau)\|_{H_1^0}^2 \int_0^\tau t^{-c} dt < \infty$.
 
 The strong c-action of the third   term is bounded by \linebreak 
 $\|P^\perp\|_{H_1^0 \rightarrow H_1^ 0}^2
 \int_0^\tau s^{-c} \| w(s)\|_{H_1^0}^2 ds$, which is finite   by  \eref{w3c} 
 since the $H_1^0$ norm is equivalent to the $H_1^\As$ norm and  $c \le b$.
 
 The fourth term in \eref{rec959} has been shown to have 
 finite strong c-action for $H_1^0$ in 
 Theorem \ref{thmrec9c}    even under the weaker condition that $w$ merely 
 has finite b-action   in the sense  of \eref{w3d}.  
 
 Hence
 $\int_0^\tau t^{-c} \| v_\tau(t)\|_{H_1^0}^2 dt < \infty$. Finally we can shift back to the Sobolev
 norm $H_1^\As$ using once more the equivalence of norms under the
  hypothesis that $\As \in L^3(\R^3)$. This completes the proof of Theorem \ref{thmrec9b}.
\end{proof}

\section{Proofs of the main theorems} \label{secrecpf}
\subsection{Existence}

\noindent
    \begin{proof}[Proof of Theorem \ref{thmrec1} (recovery) and part of Theorem \ref{thmveu1e},(existence)] 
     Since $ v$ and $v_\tau$, defined in \eref{rec1} and \eref{rec2} respectively,  
     are both given by \eref{rec23} 
     with $\tau =0$ and $\tau >0$, respectively,
     we already know from Theorem \ref{thmrec5} that  they  
    are  solutions to the variational equation \eref{ve},
     at least at the level of informal computation. 
           In Theorem \ref{thmrec10} we proved that $v(t)$ is an almost strong solution
     and $v_\tau(t)$ is a strong solution.

 Concerning the initial values of these two solutions, Theorem \ref{strinit} 
 shows that both the almost strong solution $v$ and the strong solution $v_\tau$ converge
  to their initial values in the  sense of $L^\rho(M; \L^1\otimes \kf)$ for $2 \le \rho < 3$.
  \eref{ve22a} follows from \eref{rec713c}.
    Theorem \ref{thmrec7o}  shows that the almost
 strong solution $v(t)$ is continuous on $[0, \infty)$ into $H_b^\As$,  as required in \eref{rec1.1},
  and in particular   converges to its correct initial value $v_0$  in $H_b^\As$ norm. 
 This completes the proof of Theorem \ref{thmveu1e}, Parts 1., 2. and 3.  Moreover the extension from
 $L^2$ to $L^\rho$ discussed in Remark \ref{remweak} has also been proved.
 
 All of these proofs are proofs about  the functions constructed in the recovery theorem,
  Theorem \ref{thmrec1}, which is now proved.
  
  Of these two theorems it remains only to prove uniqueness (Part 4.  of Theorem \ref{thmveu1e}.)
  This will be proved in Section \ref{secrecu}.
\end{proof}

\bigskip
\noindent
\begin{proof}[Proof of Theorem \ref{thmveu2b}]   (Finite action for  $\tau >0$)
The proof that the strong solution $v_\tau(\cdot)$ constructed by the ZDS procedure has 
 finite strong c-action when $c = min(a,b)$
 and $max(a,b) > 1/2$ is restated in Theorem \ref{thmrec9b} and proved in Section \ref{secrecab}. 
\end{proof}

\subsection{Uniqueness} \label{secrecu}

Parts 1., 2, and 3. of Theorem \ref{thmveu1e}, asserting the existence of solutions $v$ to the variational equation and their properties, have been proved in the proof of Theorem  \ref{thmrec1}. 
 It remains to prove the uniqueness of strong solutions.

\bigskip
\noindent
As in the  case of the Yang-Mills heat equation itself, the coefficients
in the variational equation \eref{ve} are too singular for the standard proof of uniqueness,
based on Gronwall's inequality, to be applicable. We will instead adapt the proof
used in \cite{G70}. 
     It suffices to prove uniqueness in the largest class of solutions
 of interest to us.  Accordingly we will assume that $a = 1/2$ and $b = 1/2$. 
 
 \begin{theorem}
 Suppose that $a = b = 1/2$ and that $A(\cdot)$ is a strong solution of the Yang-Mills heat equation
 of finite action. Let $v_1$ and $v_2$ be two strong solutions of the variational equation along $A(\cdot)$
 with finite $(1/2)$-action in the sense that \eref{w3d} holds for both.
 If $v_1(0) = v_2(0)$ then $ v_1(t) = v_2(t)$ for all $t \ge 0$.
 \end{theorem}
                \begin{proof}
 Let $v_1$ and $v_2$ be two solutions of finite (1/2)-action. Then
so is $v \equiv v_1 - v_2$. We want to show that if $v(0) = 0$ then $v(t) =0$ for all $t \ge 0$.
Denote by $B(t)$ the curvature of $A(t)$.
We have then, for $t>0$, 
\begin{align*}
(d/dt) \|v(t)\|_2^2 &= 2 (v'(t), v(t)) \\
&=2( -d_A^* d_Av(t), v(t)) -2 ([v(t)\lrc B(t)], v(t)) \\
&= -2 \| d_{A} v(t)\|_2^2 - 2 ([v(t)\lrc B(t)], v(t)) \\
&\le 2 c \|B(t)\|_\infty \|v(t)\|_2^2.
\end{align*}
Let $\beta(t) = 2c\|B(t)\|_\infty$ and $ f(t) = \| v(t)\|_2^2$. Then 
\beq
f'(t) \le \beta(t) f(t), \ \ \ \ t >0.
\eeq
We know from \eref{iby17}    
that
\beq
\int_0^T t \beta(t)^2 dt < \infty
\eeq
for any $T \in (0, \infty)$ when $A(\cdot)$ has finite action.
Hence the proof in \cite[Theorem (7.20)]{G70} 
 that $ f(t) \equiv 0$ if $f(0) = 0$ goes through once we have shown 
that 
\beq
f(t) = o(t^{1/2}).       \label{u10}
\eeq
It will be shown in  \cite[Theorem 5.5]{G72} 
that if $v$ has finite (1/2)-action then 
\beq
\int_0^T s^{1/2} \|v'(s)\|_2^2 ds < \infty         \label{u12}
\eeq
Thus, if $v(0) = 0$ then 
\begin{align*}
\|v(t)\|_2 &\le \int_0^t \|v'(s)\|_2 ds \\
&\le \Big(\int_0^t s^{-1/2} ds \Big)^{1/2} 
\Big( \int_0^t s^{1/2} \|v'(s)\|_2^2 ds \Big)^{1/2} \\
&= t^{1/4} \sqrt{2}\Big( \int_0^t s^{1/2} \|v'(s)\|_2^2 ds \Big)^{1/2}.
\end{align*}
This proves \eref{u10} and concludes the proof of uniqueness asserted in Part 4. of
Theorem  \ref{thmveu1e}.      
 This completes the proof of Theorem \ref{thmveu1e}.
\end{proof}

\section{Bibliography}

\bibliographystyle{amsplain}
\bibliography{ymh}

\end{document}